\documentclass[english,3p]{elsarticle}
\usepackage[T1]{fontenc}
\usepackage[utf8]{luainputenc}
\usepackage{float}
\usepackage{amsthm}
\usepackage{amsmath}
\usepackage{amssymb}
\usepackage{graphicx}
\usepackage{setspace}
\usepackage{esint}

\makeatletter

\providecommand{\tabularnewline}{\\}

\theoremstyle{plain}
\newtheorem{thm}{\protect\theoremname}
  \theoremstyle{remark}
  \newtheorem{rem}[thm]{\protect\remarkname}

\makeatother

\usepackage{babel}
  \providecommand{\remarkname}{Remark}
\providecommand{\theoremname}{Theorem}

\begin{document}

\title{On the Gross-Pitaevskii equation with pumping and decay: stationary
states and their stability}

\author[rvt]{Jesús~Sierra}

\ead{jesus.sierra@kaust.edu.sa}

\author[rvt]{Aslan Kasimov\corref{cor1}}

\ead{aslan.kasimov@kaust.edu.sa}

\author[rvt]{Peter~Markowich}

\ead{peter.markowich@kaust.edu.sa}

\author[focal]{Rada-Maria Weishäupl}

\ead{rada.weishaeupl@univie.ac.at}

\cortext[cor1]{Corresponding author}

\address[rvt]{King Abdullah University of Science and Technology, Box 4700, Thuwal
23955-6900, Saudi Arabia}

\address[focal]{Faculty of Mathematics, Vienna University, Oskar-Morgenstern-Platz
1, 1090 Wien, Austria}
\begin{abstract}
We investigate the behavior of solutions of the complex Gross-Pitaevskii
equation, a model that describes the dynamics of pumped decaying Bose-Einstein
condensates. The stationary radially symmetric solutions of the equation
are studied and their linear stability with respect to two-dimensional
perturbations is analyzed. Using numerical continuation, we calculate
not only the ground state of the system, but also a number of excited
states. Accurate numerical integration is employed to study the general
nonlinear evolution of the system from the unstable stationary solutions
to the formation of stable vortex patterns. \end{abstract}
\maketitle

\section{Introduction}

In this paper, we explore numerically the behavior of solutions of
the complex Gross-Pitaevskii (GP) equation:

\begin{equation}
\begin{array}{lll}
i\psi_{t}=-\Delta\psi+V\left(x\right)\psi+\left|\psi\right|^{2}\psi+i\left[\omega\left(x\right)-\sigma\left|\psi\right|^{2}\right]\psi, & t>0, & x\in D=\mathbb{R}^{2},\\
\psi\left(x,0\right)=\psi_{0}\left(x\right),\: x\in D,
\end{array}\label{eq:cGPE}
\end{equation}
where $\psi=\psi\left(x,t\right)$ represents the wave function, $V\left(x\right)$
is the trapping potential, $\omega=\omega\left(x\right)\geq0$ is
the pumping term, and $\sigma>0$ is the strength of the decaying
term. Equation (\ref{eq:cGPE}) was proposed by \citet{keeling2008spontaneous}
to study pumped decaying condensates, particularly the Bose-Einstein
condensates (BEC) of exciton-polaritons. 

We use accurate numerical techniques to investigate the nature of
the radially symmetric stationary solutions of (\ref{eq:cGPE}), their
linear stability properties, and the long-time nonlinear evolution
of the solutions of (\ref{eq:cGPE}) that start with the unstable
stationary states as initial conditions. The stationary radially symmetric
solutions of the complex GP equation are computed by using a numerical
collocation method. Such stationary solutions are found to be linearly
unstable with respect to two-dimensional perturbations when the pumping
region is small as well as when it is large.  Linearly stable solutions
are used as a starting point in a numerical continuation method, wherein
the stationary solutions are computed when a parameter in the equation
is varied. A number of different stationary solutions are found at
any given set of parameters. Among these solutions, the linearly stable
one is seen to have the smallest chemical potential. Such solution
is denoted as the ground state, while the other solutions are termed
the excited states. With the stationary solutions and their linear
stability properties known, we then implement a time-splitting spectral
method to solve (\ref{eq:cGPE}) and explore the nonlinear dynamics
of the solutions.

The GP equation and its variants are widely used to understand BEC
in various systems. The possibility of condensation of bosons was
predicted by \citet{bose1924plancks} and \citet{einstein1924BEC,einstein1925quantum}
in 1924-1925. The condensate was obtained experimentally for the first
time in 1995 \citep{anderson1995observation,bradley1995evidence,davis1995physical}
in a system consisting of about half a million alkali atoms cooled
down to nanokelvin-level temperatures. The principal interest in BEC
lies in its nature as a macroscopic quantum system. Some of the dynamics
of atomic BEC have been successfully described by the GP equation
\citep{gross1963hydrodynamics,pitaevskii1961vortex,pitaevskii2003bose},
a nonlinear Schrödinger equation, 
\begin{equation}
i\hbar\psi_{t}=-\frac{\hbar^{2}}{2m}\triangle\psi+V\left(\mathbf{x}\right)\psi+\delta\left|\psi\right|^{2}\psi,\label{eq:GPE}
\end{equation}
derived from the mean field theory of weakly interacting bosons. Here,
$\psi=\psi(\mathbf{x},t)$ is the wave function of the condensate,
$\delta$ is a constant characterizing the strength of the boson-boson
interactions, $m$ is the mass of the particles, and $V(\mathbf{x})$
is the trapping potential.

\begin{figure}[H]
\label{cavity} 

\begin{centering}
\includegraphics[scale=0.3,angle=-90]{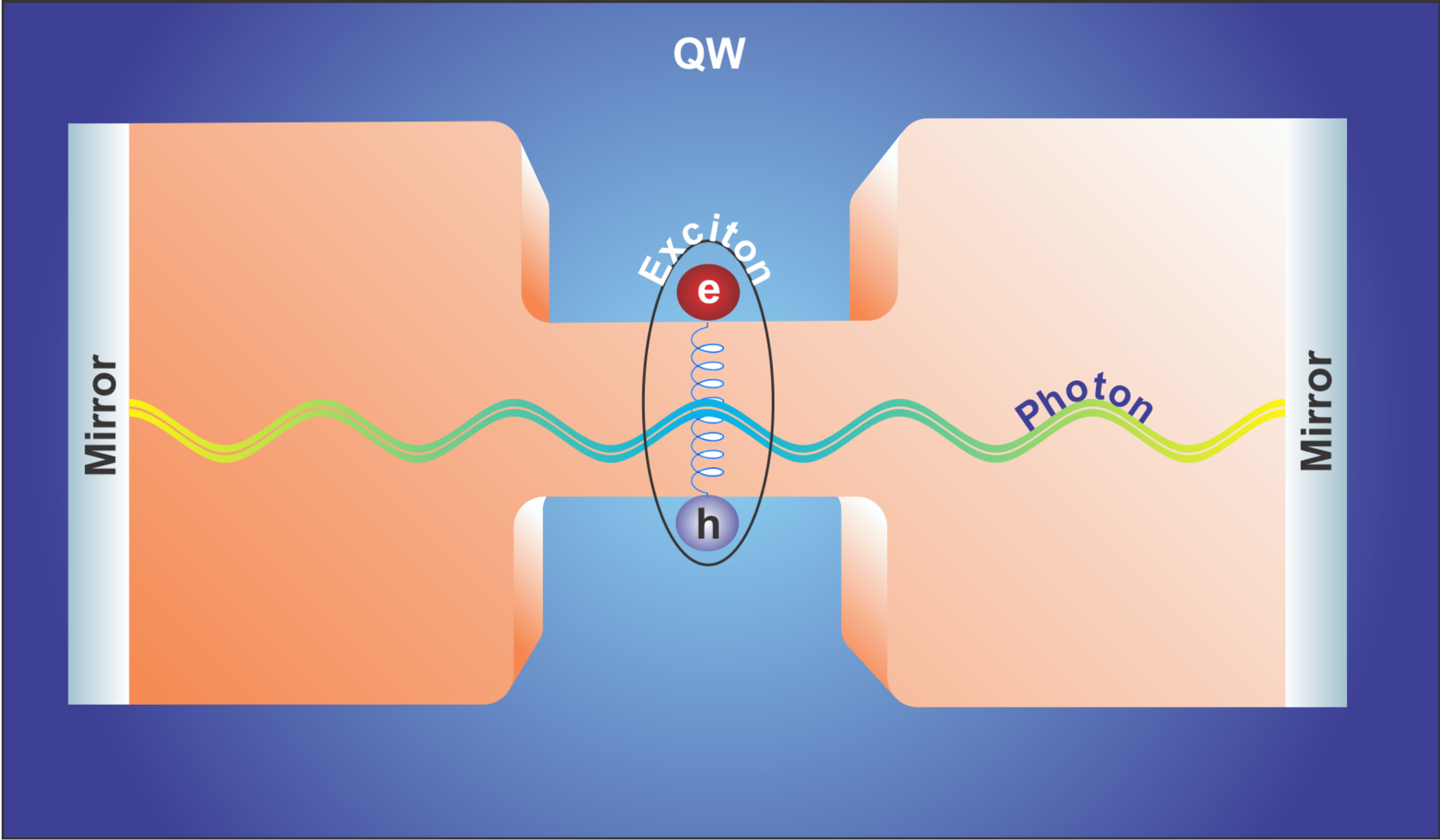} 
\par\end{centering}

\caption{Schematics of the exciton-polariton microcavity.}
\end{figure}

A serious obstacle in the study of BEC in atomic systems is the extremely
low temperatures required to create the condensate. Because of this
difficulty, other, non-atomic systems are being explored that can
undergo condensation at higher temperatures. One possible candidate
is a system of exciton-polaritons, which are quasi-particles that
can be created in semi-conductor cavities as a result of interaction
between excitons and a laser field in the cavity \citep{kasprzak2006bose,coldren2012diode}.
A two-dimensional quantum structure consisting of coupled quantum
wells embedded in an optical microcavity is used (see Fig. \ref{cavity}).
Excitons are electron-hole pairs, produced in the coupled quantum
wells, that interact with the photons trapped inside the optical cavity
by means of two highly reflective mirrors. Due to this confinement,
the effective mass of the polaritons is very small: $10^{-4}$ times
the free electron mass \citep{kasprzak2006bose}. Since the temperature
of condensation is inversely proportional to the mass of the particles,
the exciton-polariton systems afford relatively high temperatures
of condensation. There are, however, two drawbacks in this new condensate:
the polaritons are highly unstable and exhibit strong interactions.
The excitons disappear with the recombination of the electron-hole
pairs through emission of photons. One way to deal with this problem
is to introduce a polariton reservoir: polaritons are ``cooled''
and ``pumped'' from this reservoir into the condensate to compensate
for the decay. At the same time, a low density level is kept in order
to reduce the interactions between polaritons. For a detailed study
of this system, see, e.g., \citep{bramati2013physics,borgh2012robustness,keeling2011exciton}.

Various mathematical models have been proposed for this new condensate
\citep{keeling2008spontaneous,d2010persistent,wouters2007excitations}.
One of them, called complex GP equation \citep{keeling2008spontaneous},
is explored in this paper. The complex GP equation reflects the non-equilibrium
dynamics described above by adding pumping and decaying terms to the
GP equation. In \citep{keeling2008spontaneous}, the authors observed
the formation of two-dimensional vortex lattices. Linear stability
of stationary solutions and the formation of dark solitons in the
one-dimensional complex GP equation were analyzed by \citet{cuevas2011nonlinear}.
The role of damping in the absence of the pumping term in the GP equation
is studied by, e.g., \citet{bao2004three} and \citet{antonelli2013nonlinear}. 

The remainder of the paper is organized as follows. In Section 2,
the stationary radial solutions of the complex GP equation are calculated.
In Section 3, a numerical collocation method is used to compute the
solutions described in Section 2. In Section 4, the linear stability
of the stationary states is analyzed. In Section 5, the numerical
continuation of one of the linearly stable solutions is carried out,
showing some of the excited states. In Section 6, the full complex
GP equation is solved numerically by a second order time-splitting
spectral method. Conclusions are drawn in Section 7.

\section{Stationary radial solutions}

Let us analyze solutions of (\ref{eq:cGPE}) of the form $\psi\left(x,t\right)=\exp\left(-i\mu t\right)\phi\left(x\right)$,
where $\mu$ corresponds to the real-valued chemical potential and
$\phi$ represents the amplitude function that vanishes as $\left|x\right|\rightarrow\infty$.
Introducing this ansatz into (\ref{eq:cGPE}) leads to the following
equation for $\phi$,

\begin{equation}
\mu\phi\left(x\right)=\left[-\Delta+V\left(x\right)+\left|\phi\left(x\right)\right|^{2}+i\left(\omega\left(x\right)-\sigma\left|\phi\left(x\right)\right|^{2}\right)\right]\phi\left(x\right).\label{eq:st_cGPE}
\end{equation}
Multiplying (\ref{eq:st_cGPE}) by $\phi^{*}$ (the complex conjugate
of $\phi$) and integrating over $D=\mathbb{R}^{2}$ yields

\begin{equation}
\mu\int_{D}\left|\phi\left(x\right)\right|^{2}dx=\int_{D}\left[\left|\nabla\phi\left(x\right)\right|^{2}+V\left(x\right)\left|\phi\left(x\right)\right|^{2}+\left|\phi\left(x\right)\right|^{4}+i\left(\omega\left(x\right)-\sigma\left|\phi\left(x\right)\right|^{2}\right)\left|\phi\left(x\right)\right|^{2}\right]dx.\label{eq:eq_mu}
\end{equation}
It follows immediately that, for $\mu$ to be real, it must be required
that

\begin{equation}
\int_{D}\left(\omega\left(x\right)-\sigma\left|\phi\left(x\right)\right|^{2}\right)\left|\phi\left(x\right)\right|^{2}dx=0.\label{eq:cond}
\end{equation}
Note that allowing $\mathrm{Im}(\mu)\neq0$ would lead to exponential
growth ($\mathrm{Im}\left(\mu\right)>0$) or exponential decay to
zero ($\mathrm{Im}\left(\mu\right)<0$) of the solution, since in
this case $\psi\left(x,t\right)=\exp\left(-i\mathrm{Re}\left(\mu\right)t\right)\exp\left(\mathrm{Im}\left(\mu\right)t\right)\phi\left(x\right)$.

An alternative approach to study the stationary solutions of (\ref{eq:cGPE})
is given by its quantum hydrodynamic version. For this, define $S=S\left(x,t\right)$
as the phase of $\psi$ and $\rho=\rho\left(x,t\right)=\left|\psi\left(x,t\right)\right|^{2}$.
Hence, inserting the ansatz $\psi=\sqrt{\rho}\exp\left(iS\right)$
into (\ref{eq:cGPE}) and separating real and imaginary parts yields
\begin{align}
\frac{1}{2}\rho_{t}+\mathrm{div}J-\left(\omega\left(x\right)-\sigma\rho\right)\rho= & 0,\label{eq:Mad_1}\\
S_{t}+\left|\nabla S\right|^{2}+V-\frac{\triangle\sqrt{\rho}}{\sqrt{\rho}}+\rho= & 0,\label{eq:Mad_2}
\end{align}
where $J=J\left(x,t\right)=\mathrm{Im}\left(\psi^{*}\nabla\psi\right)$
is the current density. Taking the gradient of (\ref{eq:Mad_2}),
multiplying by $\rho$, and using (\ref{eq:Mad_1}) and $\mathrm{div}\left(J\otimes J/\rho\right)=\frac{1}{2}\rho\nabla\left|\nabla S\right|^{2}+\nabla S\cdot\mathrm{div}J$
gives the following equivalent expression of (\ref{eq:Mad_2}):

\begin{equation}
J_{t}+2\mathrm{div}\left(\frac{J\otimes J}{\rho}\right)+\rho\nabla V-\rho\nabla\left(\frac{\triangle\sqrt{\rho}}{\sqrt{\rho}}\right)+\frac{1}{2}\nabla\rho^{2}=2\left(\omega\left(x\right)-\sigma\rho\right)J.\label{eq:Mad-3}
\end{equation}
\citet{gasser1997quantum} showed that each of the terms in (\ref{eq:Mad-3}),
in particular, the nonlinear terms with singularities in the vacuum
state ($\rho\equiv0$), are well defined in the sense of distributions
if finite kinetic energy solutions are considered, i.e., 

\[
\int_{D}\left|\nabla\psi\right|^{2}dx<\infty.
\]
Using this result, the next theorem can be deduced from (\ref{eq:Mad_1})
and (\ref{eq:Mad-3}).
\begin{thm}
Let $\omega=\omega\left(x\right)\geq0$ be smooth with $\left\{ \left.x\right|\omega\left(x\right)>0\right\} $
simply connected. Then there exists a stationary solution of (\ref{eq:Mad_1})-(\ref{eq:Mad-3})
with $J\equiv0$ if and only if either

1. $\rho\equiv0$, or

2. $\rho=\frac{\omega\left(x\right)}{\sigma}$. In this case, $V(x)=C+\frac{\triangle\sqrt{\omega\left(x\right)}}{\sqrt{\omega\left(x\right)}}-\frac{1}{\sigma}\omega\left(x\right)$
in $\left\{ \left.x\right|\omega\left(x\right)>0\right\} $, where
$C$ is a constant.\end{thm}
\begin{proof}
Set $\rho_{t}\equiv0$, $J_{t}\equiv0$, and $J\equiv0$ in (\ref{eq:Mad_1})
and (\ref{eq:Mad-3}). Then, statement 1 is a trivial solution of
the system. For statement 2, notice that (\ref{eq:Mad_1}) yields
$\rho=\omega\left(x\right)/\sigma$. Inserting this expression into
(\ref{eq:Mad-3}) gives

\[
\begin{array}{cc}
\nabla\left(V\left(x\right)-\frac{\triangle\sqrt{\omega\left(x\right)}}{\sqrt{\omega\left(x\right)}}+\frac{1}{\sigma}\omega\right)=0, & \mathrm{in\:}\left\{ \left.x\right|\omega\left(x\right)>0\right\} \end{array}.
\]
Hence,

\begin{equation}
\begin{array}{cc}
V\left(x\right)=C+\frac{\triangle\sqrt{\omega\left(x\right)}}{\sqrt{\omega\left(x\right)}}-\frac{1}{\sigma}\omega, & \mathrm{in\:}\left\{ \left.x\right|\omega\left(x\right)>0\right\} ,\end{array}\label{eq:pot_car_c2}
\end{equation}
where $C$ is a constant. 

Now, if statement 1 is true, then $J\equiv0$ follows immediately.
On the other hand, if statement 2 is valid, then (\ref{eq:Mad_2})
gives $\left|\nabla S\right|^{2}=0$, which implies $\nabla S=0$.
Thus, $J\left(=\rho\nabla S\right)\equiv0$ .\end{proof}
\begin{rem}
The stationary GP equation, 

\begin{equation}
\begin{array}{cc}
V(x)=\mu+\frac{\triangle\phi}{\phi}-\frac{1}{\sigma}\left|\phi\right|^{2}, & \phi\neq0,\end{array}\label{eq:st_GP}
\end{equation}
has exactly the form of (\ref{eq:pot_car_c2}), with $C=\mu$ (the
chemical potential). Then, letting $\omega\left(x\right)=\left|\phi\right|^{2}$
(the density profile of any solution of (\ref{eq:st_GP})), the solution
of the corresponding complex GP equation is given by $\rho=\left|\phi\right|^{2}/\sigma$
with $J\equiv0$.
\end{rem}
Now, going back to (\ref{eq:Mad_1}), and integrating in space gives

\[
\frac{1}{2}\frac{dM}{dt}=\int_{D}\left(\omega\left(x\right)-\sigma\rho\right)\rho dx,
\]
where

\[
M\equiv\int_{D}\rho dx
\]
is the $mass$ of the system. Therefore, we can see that equation
(\ref{eq:st_cGPE}) along with condition (\ref{eq:cond}) lead to
solutions where the pumping and decaying parts equilibrate, i.e.,
the time derivative of the mass of the system is zero. In addition,
condition (\ref{eq:cond}) ensures that the chemical potential is
a nonnegative real quantity. 

For the rest of the paper, $V(x)=\left|x\right|^{2}$ is the harmonic
potential and $\omega\left(x\right)=\alpha\Theta\left(R-\left|x\right|\right)$,
where $\alpha>0$ is the strength of the pumping term and $R>0$ is
the radius of the pumping region delimited by the smoothed Heaviside
function $\Theta\left(x\right)=\left(1+\tanh\left(\kappa x\right)\right)/2$
for some fixed parameter $\kappa>0$. 

In the following we investigate the stationary radially symmetric
solutions of (\ref{eq:st_cGPE}). In this case, (\ref{eq:st_cGPE})
becomes $(\begin{array}{cc}
'=\frac{d}{dr}, & \Theta_{R}=\Theta\left(R-r\right)\end{array})$

\begin{equation}
\mu\phi=-\phi''-\frac{1}{r}\phi'+r^{2}\phi+\left|\phi\right|^{2}\phi+i\left(\alpha\Theta_{R}-\sigma\left|\phi\right|^{2}\right)\phi,\label{eq:rad_eigen}
\end{equation}
where $r=\left|x\right|\geq0$, with the boundary conditions $\phi\left(r\right)=0$
as $r\rightarrow\infty$ and $\phi'\left(0\right)=0$. Condition (\ref{eq:cond})
is written as

\begin{equation}
\intop_{0}^{\infty}\left(\alpha\Theta_{R}-\sigma\left|\phi\left(r\right)\right|^{2}\right)\left|\phi\left(r\right)\right|^{2}rdr=0.\label{eq:rad_cond}
\end{equation}
By defining 

\begin{equation}
\xi\left(r\right)=\int_{0}^{r}\left(\alpha\Theta_{R}-\sigma\left|\phi\left(s\right)\right|^{2}\right)\left|\phi\left(s\right)\right|^{2}sds,\label{eq:cond_eq}
\end{equation}
condition (\ref{eq:rad_cond}) can be expressed as

\begin{equation}
\xi'\left(r\right)=\left(\alpha\Theta_{R}-\sigma\left|\phi\left(r\right)\right|^{2}\right)\left|\phi\left(r\right)\right|^{2}r,\label{eq:cond_ODE}
\end{equation}
with

\begin{equation}
\xi\left(\infty\right)=0.\label{eq:cond_BC}
\end{equation}
The variable $\xi$ is convenient for subsequent numerical integration.
The constancy of the chemical potential $\mu$ can be written as

\begin{equation}
\mu'\left(r\right)=0.\label{eq:Eq_mi}
\end{equation}
Because the phase of $\phi$ is arbitrary, we let 

\begin{equation}
\mathrm{Im}\left(\phi\left(0\right)\right)=0.\label{eq:Phase_cond}
\end{equation}

Notice that the second order ordinary differential equation (ODE)
(\ref{eq:rad_eigen}) can be expanded as a system of two first-order
ODE. Due to the harmonic trapping potential and the finite pumping
region, the solution is expected to be concentrated on a bounded domain.
This fact is used in the numerical computations, where the integration
domain is chosen as $D=\left[0,b\right]$ for large enough $b$. Then,
the right boundary conditions, $\phi\left(\infty\right)=0$ and $\xi\left(\infty\right)=0$,
are replaced by $\phi\left(b\right)=0$ and $\xi\left(b\right)=0$,
respectively. As a result, the following set of ODE is obtained for
$r\in\left[0,b\right]$:

\begin{equation}
\begin{array}{ccl}
\phi' & = & \varphi,\\
\varphi' & = & -\frac{1}{r}\varphi-\mu\phi+r^{2}\phi+\left|\phi\right|^{2}\phi+i\left(\alpha\Theta_{R}-\sigma\left|\phi\right|^{2}\right)\phi,\\
\xi' & = & \left(\alpha\Theta_{R}-\sigma\left|\phi\left(r\right)\right|^{2}\right)\left|\phi\left(r\right)\right|^{2}r,\\
\mu' & = & 0,
\end{array}\label{eq:Sys_ODE}
\end{equation}
with

\[
\begin{array}{cccc}
\phi\left(b\right)=0, & \varphi\left(0\right)=0, & \mathrm{Im}\left(\phi\left(0\right)\right)=0, & \xi\left(b\right)=0.\end{array}
\]
This nonlinear boundary value problem with a singularity of the first
kind (see, e.g., \citep{de1976difference}) is consistent with respect
to the number of boundary conditions. The value of $\varphi'$ at
$r=0$ can be determined using 

\begin{equation}
\underset{r\rightarrow0}{\lim}\frac{\varphi\left(r\right)-\varphi\left(0\right)}{r-0}=\underset{r\rightarrow0}{\lim}\frac{\varphi\left(r\right)}{r}=\varphi_{r}\left(0\right),\label{eq:lim}
\end{equation}
due to the boundary condition $\varphi\left(0\right)=0$. Then, the
second equation in (\ref{eq:Sys_ODE}) gives

\begin{equation}
\varphi'\left(0\right)=\frac{1}{2}\left[-\mu\phi\left(0\right)+\left|\phi\left(0\right)\right|^{2}\phi\left(0\right)+i\left(\alpha-\sigma\left|\phi\left(0\right)\right|^{2}\right)\phi\left(0\right)\right].\label{eq:val_at_zero}
\end{equation}

It is convenient to make (\ref{eq:Sys_ODE}) real valued, autonomous,
and restricted to the interval $\left[0,1\right]$. Substituting $\phi=\theta+i\eta$
into (\ref{eq:rad_eigen}), separating real and imaginary parts, and
writing the resulting equations as a system of first order ODE, leads
to a real valued system. In order to make (\ref{eq:Sys_ODE}) autonomous,
define $\varrho(r)=r$ such that $\varrho'(r)=1$ and $\varrho\left(0\right)=0$
are added to the system. Finally, rescaling $r$ as $\tilde{r}=\frac{r}{b}$,
restricts the domain to the interval $[0,1]$. Thus, (\ref{eq:Sys_ODE})
becomes ($'=\frac{d}{d\tilde{r}}$):

\begin{doublespace}
\begin{equation}
\begin{array}{ccl}
\frac{1}{b}\theta' & = & \varphi,\\
\frac{1}{b}\varphi' & = & -\frac{1}{\varrho}\varphi-\mu\theta+\varrho^{2}\theta+\left(\theta^{2}+\eta^{2}\right)\theta-\left(\alpha\Theta\left(R-\varrho\right)-\sigma\left(\theta^{2}+\eta^{2}\right)\right)\eta,\\
\frac{1}{b}\eta' & = & \zeta,\\
\frac{1}{b}\zeta' & = & -\frac{1}{\varrho}\zeta-\mu\eta+\varrho^{2}\eta+\left(\theta^{2}+\eta^{2}\right)\eta+\left(\alpha\Theta\left(R-\varrho\right)-\sigma\left(\theta^{2}+\eta^{2}\right)\right)\theta,\\
\frac{1}{b}\xi' & = & \left(\alpha\Theta\left(R-\varrho\right)-\sigma\left(\theta^{2}+\eta^{2}\right)\right)\left(\theta^{2}+\eta^{2}\right)\varrho,\\
\frac{1}{b}\mu' & = & 0,\\
\frac{1}{b}\varrho' & = & 1,
\end{array}\label{eq:Sys_2}
\end{equation}
with
\end{doublespace}

\[
\begin{array}{ccccccc}
\varphi\left(0\right)=0, & \theta\left(1\right)=0, & \zeta\left(0\right)=0, & \eta\left(0\right)=0, & \eta\left(1\right)=0, & \xi\left(1\right)=0, & \varrho\left(0\right)=0.\end{array}
\]

\section{Collocation method}

System (\ref{eq:Sys_2}) can be solved with high accuracy by using
a collocation method. The basic idea involves forming an approximate
solution as a linear combination of a set of functions (usually polynomials),
the coefficients of which are determined by requiring the combination
to satisfy the system at certain points (the collocation points) as
well as the boundary conditions.

To introduce the algorithm, consider the following first-order system
of ODE:

\begin{equation}
\begin{array}{ccc}
u'\left(r\right)=f\left(u\left(r\right)\right), & r\in I=\left[0,1\right], & u\left(\cdot\right),f\left(\cdot\right)\in\mathbb{R}^{d}\end{array},\label{eq:col_ODE}
\end{equation}
subject to the boundary conditions

\[
\begin{array}{cc}
b\left(u\left(0\right),u\left(1\right)\right)=0, & b\left(\cdot\right)\in\mathbb{R}^{d}.\end{array}
\]
Let $I_{h}=\left\{ r_{n}:0=r_{0}<r_{1}<\cdots<r_{N}=1\right\} $ be
a given mesh on $I$, with $h_{n}=r_{n}-r_{n-1}$ $\left(n=1,\ldots,N\right)$.
Define the space of vector-valued piecewise polynomials, $\mathcal{P}^{m}\left(I_{h}\right)$,
as

\[
\mathcal{P}^{m}\left(I_{h}\right)=\left\{ v\in C\left(I\right):\begin{array}{cc}
v\mid_{\left[r_{n-1},r_{n}\right]}\in\pi_{m}, & \left(1\leq n\leq N\right)\end{array}\right\} ,
\]
where $\pi_{m}$ is the space of all vector-valued polynomials of
degree $\leq m$. The collocation method consists of finding an element
$q_{h}\in\mathcal{P}^{m}\left(I_{h}\right)$ that approximates the
solution of (\ref{eq:col_ODE}). Such approximation will be found
by requiring $q_{h}$ to satisfy (\ref{eq:col_ODE}) on a given finite
subset, $S_{h}$, of $I$, as well as the boundary conditions.

Let $S_{h}$, the set of collocation points, be given by

\[
S_{h}=\left\{ r_{n,i}=r_{n-1}+c_{i}h_{n}:\begin{array}{cc}
1\leq n\leq N, & 0=c_{1}<c_{2}<\cdots<c_{m}=1\end{array}\right\} ,
\]
where $\left\{ c_{i}\right\} $ are usually Gauss, Radau, or Lobatto
points (see, e.g., \citet{ascher1987numerical}). The collocation
solution $q_{h}\in\mathcal{P}^{m}\left(I_{h}\right)$ for (\ref{eq:col_ODE})
is defined by the equation

\[
\begin{array}{ccc}
q'_{h}\left(r_{n,i}\right)=f\left(q_{h}\left(r_{n,i}\right)\right), & r_{n,i}\in S_{h}, & b\left(q_{h}\left(0\right),q_{h}\left(1\right)\right)=0.\end{array}
\]
Now, consider a local Lagrange basis representation of $q_{h}$. Define

\[
\begin{array}{ccc}
L_{i}\left(\delta\right):=\overset{m}{\underset{k=1,k\neq i}{\prod}}\frac{\delta-c_{k}}{c_{i}-c_{k}}, & \left(i=1,...,m\right), & \delta\in\left[0,1\right].\end{array}
\]
Then the local polynomials can be written as

\[
q_{n}\left(r_{n-1}+\delta h_{n}\right)=\sum_{i=1}^{m}L_{i}\left(\delta\right)q_{h}\left(r_{n-1}+c_{i}h_{n}\right),\begin{array}{cc}
1\leq n\leq N, & \delta\in\left[0,1\right].\end{array}
\]
The collocation equations are, therefore,

\begin{equation}
\begin{array}{ccc}
q'_{n}\left(r_{n,i}\right)=f\left(q_{n}\left(r_{n,i}\right)\right), & r_{n,i}\in S_{h}, & 1\leq n\leq N,\end{array}\label{eq:col_eq_f}
\end{equation}
with the boundary conditions

\begin{equation}
\begin{array}{cc}
b_{l}\left(u_{0},u_{N}\right)=0, & l=1,...,d.\end{array}\label{eq:bcn_col_eq_f}
\end{equation}

Equations (\ref{eq:col_eq_f}) - (\ref{eq:bcn_col_eq_f}) can be solved
by a Newton-Chord iteration. It can be proven that if the solution
$u\left(r\right)$ of the boundary value problem is sufficiently smooth,
then the order of accuracy of the method is $m$, i.e., $\left\Vert q_{h}-u\right\Vert _{\infty}=\mathcal{O}\left(h^{m}\right)$,
where $h:=\max\left\{ h_{n}:1\leq n\leq N\right\} $ (the diameter
of the mesh $I_{h}$). For a detailed description of this method,
the reader is referred to \citet{ascher1987numerical} and \citet{brunner2004collocation}.

The Matlab function $\mathit{bvp5c}$ is used to solve system (\ref{eq:Sys_2})
by the collocation method. This function implements a four-point Lobatto
method capable of handling singularities of the first kind and mesh
adaptation. In addition, $\mathit{bvp5c}$ can converge even if the
initial guess is not very close to the solution; that is, the function
converges by using as an initial guess the Thomas-Fermi profile \citep{keeling2008spontaneous}: 

\begin{equation}
\left|\phi\right|^{2}=\begin{cases}
\begin{array}{cc}
\tilde{\mu}-r^{2}, & r<\sqrt{\tilde{\mu}},\\
0, & r\geq\sqrt{\tilde{\mu}},
\end{array} & \tilde{\mu}=3\alpha/2\sigma.\end{cases}\label{eq:Thomas-Fermi}
\end{equation}
It is important to note that $\mathit{bvp5c}$ always evaluates the
ODE at the end points. Then, due to the singularity, the value obtained
in (\ref{eq:val_at_zero}) must be supplied to $\mathit{bvp5c}$ when
evaluating the ODE at $r=0$. 

With 6000 grid points (initially equispaced) on the domain $D=[0,15]$,
the maximum residual obtained is less than $10^{-11}$ for the different
cases studied: $\alpha=4.4$, $\sigma=0.3$, $\kappa=10$, and $0<R<10$.
In the following section, a linear stability analysis is performed
on these results, which are later used as a starting point in the
numerical continuation.

\section{Linear stability analysis of the radially symmetric solutions}

Consider the radially symmetric solutions $\psi\left(r,t\right)=\exp\left(-i\mu t\right)\phi\left(r\right)$
obtained in the previous section. To study their linear stability,
small perturbations of the wave function can be expressed as

\begin{equation}
\psi\left(r,\theta,t\right)=\exp\left(-i\mu t\right)\left[\phi\left(r\right)+\varepsilon\left(u\left(r\right)\exp\left(-i\left(m\theta+\omega t\right)\right)+v^{*}\left(r\right)\exp\left(i\left(m\theta+\omega^{*}t\right)\right)\right)\right],\label{eq:pert_wf}
\end{equation}
with $\varepsilon<<1$, $m=1,2,3,...$

This form of the perturbation and the subsequent calculations are
known as Bogoliubov-De Gennes analysis. Notice that the perturbation
is time dependent with frequency $\omega$ and complex amplitude functions
$u\left(r\right)$ and $v\left(r\right)$. In addition, the perturbation
is both radial and angular, where the angular part is tested with
different modes given by $m$. The radially symmetric solutions are
linearly unstable if $\mathrm{Im}\left(\omega\right)>0$.

We proceed by evaluating the complex GP equation for the trial wave
function (\ref{eq:pert_wf}) to $\mathcal{O}\left(\varepsilon\right)$.
To simplify the notation, we write $\phi\left(r\right)$, $u\left(r\right)$
and $v\left(r\right)$ as $\phi$, $u$ and $v$, respectively. Equating
terms in $\exp\left(-i\mu t\right)$ yields the equation for the radially
symmetric solutions:

\begin{equation}
\mu\phi=\left(-\frac{d^{2}}{dr^{2}}-\frac{1}{r}\frac{d}{dr}+r^{2}+\left|\phi\right|^{2}+i\left(\alpha\Theta_{R}-\sigma\left|\phi\right|^{2}\right)\right)\phi.\label{eq:a}
\end{equation}
Equating terms in $\exp\left(-i\left(m\theta+\left(\mu+\omega\right)t\right)\right)$
leads to

\begin{equation}
\begin{array}{ccc}
\left(\mu+\omega\right)u & = & \left(-\frac{d^{2}}{dr^{2}}-\frac{1}{r}\frac{d}{dr}+\frac{m^{2}}{r^{2}}+r^{2}+2\left|\phi\right|^{2}+i\left(\alpha\Theta_{R}-2\sigma\left|\phi\right|^{2}\right)\right)u+\left(1-i\sigma\right)\phi^{2}v.\end{array}\label{eq:b}
\end{equation}
Equating terms in $\exp\left(i\left(m\theta-\left(\mu-\omega^{*}\right)t\right)\right)$
gives

\begin{equation}
\begin{array}{ccc}
\left(\mu-\omega^{*}\right)v^{*} & = & \left(-\frac{d^{2}}{dr^{2}}-\frac{1}{r}\frac{d}{dr}+\frac{m^{2}}{r^{2}}+r^{2}+2\left|\phi\right|^{2}+i\left(\alpha\Theta_{R}-2\sigma\left|\phi\right|^{2}\right)\right)v^{*}+\left(1-i\sigma\right)\phi^{2}u^{*}.\end{array}\label{eq:c}
\end{equation}
Taking the complex conjugate of (\ref{eq:c}) produces

\begin{equation}
\begin{array}{ccc}
\left(\mu-\omega\right)v & = & \left(-\frac{d^{2}}{dr^{2}}-\frac{1}{r}\frac{d}{dr}+\frac{m^{2}}{r^{2}}+r^{2}+2\left|\phi\right|^{2}-i\left(\alpha\Theta_{R}-2\sigma\left|\phi\right|^{2}\right)\right)v+\left(1+i\sigma\right)\left(\phi^{*}\right)^{2}u.\end{array}\label{eq:c-1}
\end{equation}
Finally, equations (\ref{eq:b}) and (\ref{eq:c-1}) can be written
as 

\begin{equation}
\left[\begin{array}{rr}
L_{1} & L_{2}\\
-L_{2}^{*} & -L_{1}^{*}
\end{array}\right]\left[\begin{array}{c}
u\\
v
\end{array}\right]=\omega\left[\begin{array}{c}
u\\
v
\end{array}\right],\label{eq:eigen_prob}
\end{equation}
with

\[
\begin{array}{ccl}
L_{1} & = & -\mu-\frac{d^{2}}{dr^{2}}-\frac{1}{r}\frac{d}{dr}+\frac{m^{2}}{r^{2}}+r^{2}+2\left|\phi\right|^{2}+i\left(\alpha\Theta_{R}-2\sigma\left|\phi\right|^{2}\right)\\
 & = & -\mu-\frac{d^{2}}{dr^{2}}-\frac{1}{r}\frac{d}{dr}+\frac{m^{2}}{r^{2}}+r^{2}+2\left(1-i\sigma\right)\left|\phi\right|^{2}+i\alpha\Theta_{R},
\end{array}
\]
and

\[
L_{2}=\left(1-i\sigma\right)\phi^{2}.
\]

The eigenvalue problem (\ref{eq:eigen_prob}) is discretized by using
centered differences and solved with the Matlab function $eig$. The
results for $R=2$ are plotted in Fig. (\ref{fig:St_R_2}). Figure
(\ref{fig:St_R_2}.b) shows the maximum imaginary part of the eigenvalues,
$\omega$, for $m=1,2,...,50$. Notice that all of them are negative;
thus, the solution is linearly stable under these perturbations. Taking
this into account, this solution will be used as a starting point
in the numerical continuation method presented in the following section.

Figures (\ref{fig:St_R_8}) and (\ref{fig:St_R_9}) show two cases
where the solutions are linearly unstable. An interesting observation
is the fact that the radially symmetric solutions are the same for
large enough values of $R$; we note that this is the case for $R=8$
and $R=9$. This characteristic was also observed by \citet{cuevas2011nonlinear}
for the one-dimensional complex GP equation. 

Fig (\ref{fig:stab_curve}) shows the maximum imaginary part of the
eigenvalues, $\omega$, for $0<R<10$, with spacing of 0.1 in $R$
and for $m=1,2,...,50$. Notice that for large values of $R$, which
in this case also implies large mass, the solutions become unstable.
However, the solutions are also unstable for small values of $R$
(small mass), i.e. at $R\lesssim0.6$. For $R\lesssim0.6$ it was
observed that a different branch of solutions with smaller chemical
potential, $\mu$, becomes stable. The case of large values of $R$
is studied in detail in Section 6. 

Finally, Fig (\ref{fig:currents_R2-R8}) shows the magnitude of the
current, $\left|J\right|$, corresponding to the solutions for $R=2,8,9$. 

\begin{figure}[H]
\begin{centering}
\begin{tabular}{cc}
a)\includegraphics[scale=0.34]{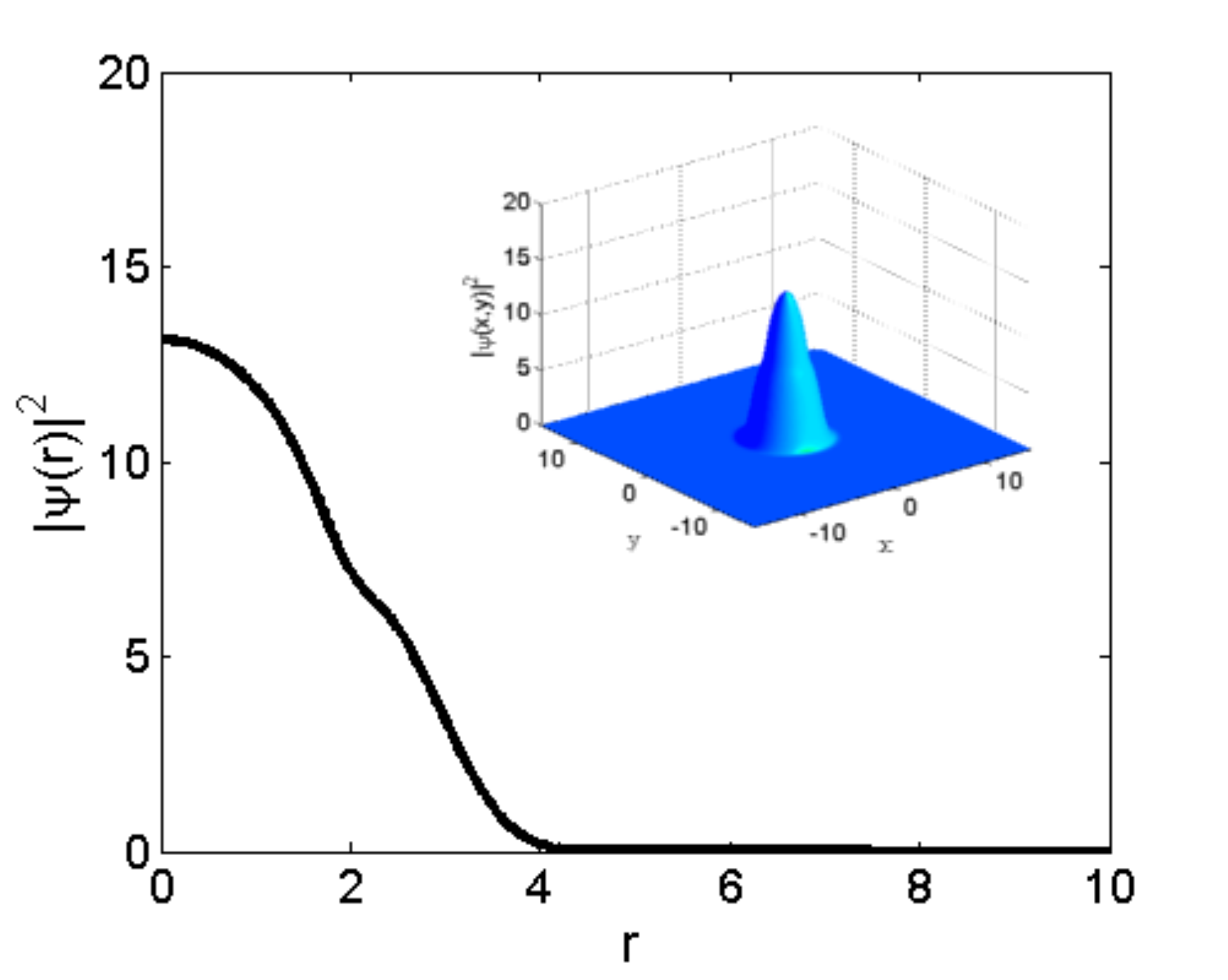} & b)\includegraphics[scale=0.34]{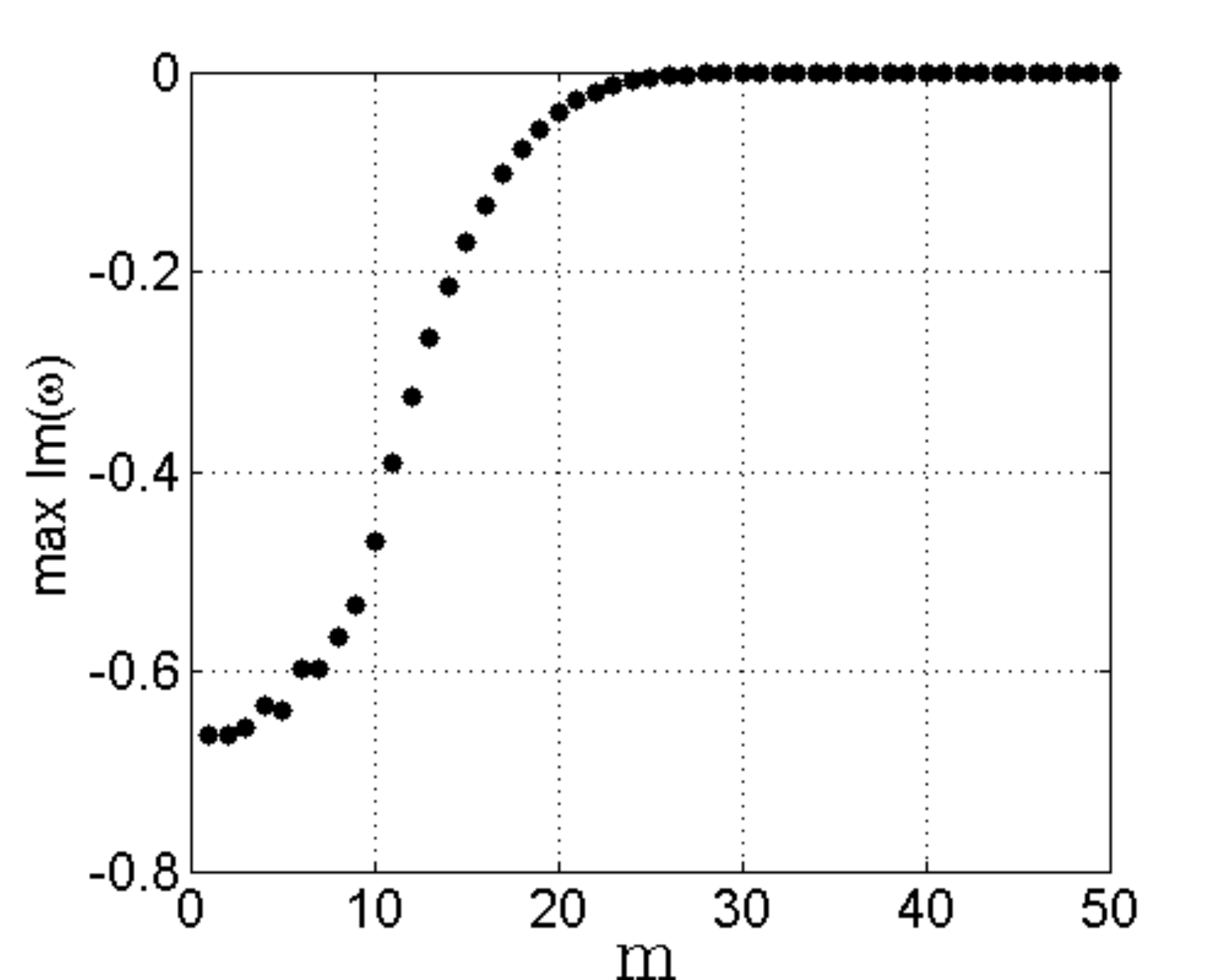}\tabularnewline
\end{tabular}
\par\end{centering}

\caption{a) Density profile of the radially symmetric solution for $R=2$,
with the two-dimensional view shown in the inset. b) Result from the
linear stability analysis: the plot shows the maximum imaginary part
of the eigenvalues, $\omega$, for $m=1,2,...,50$.\label{fig:St_R_2}}
\end{figure}

\begin{figure}[H]
\begin{centering}
\begin{tabular}{cc}
a)\includegraphics[scale=0.34]{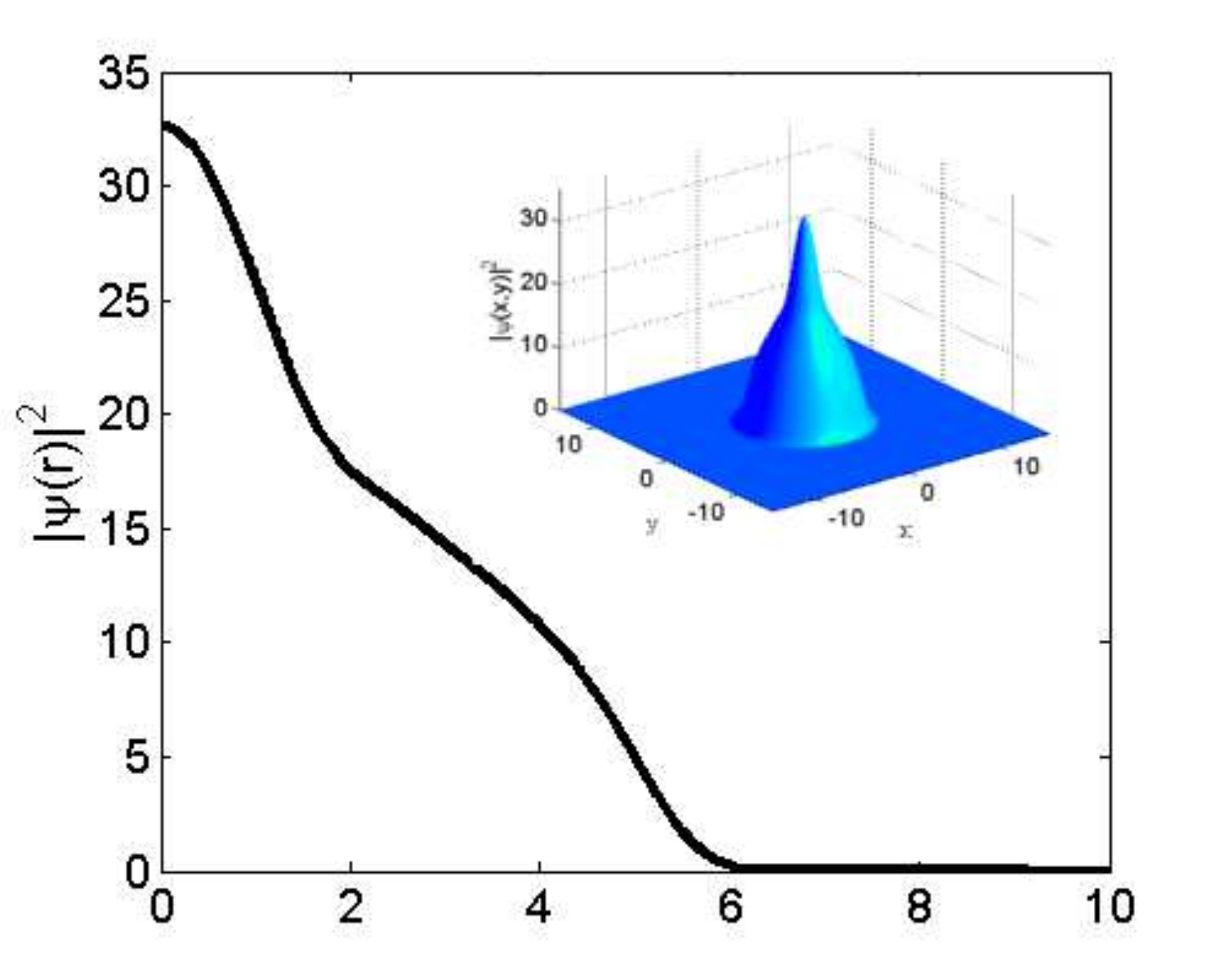} & b)\includegraphics[scale=0.34]{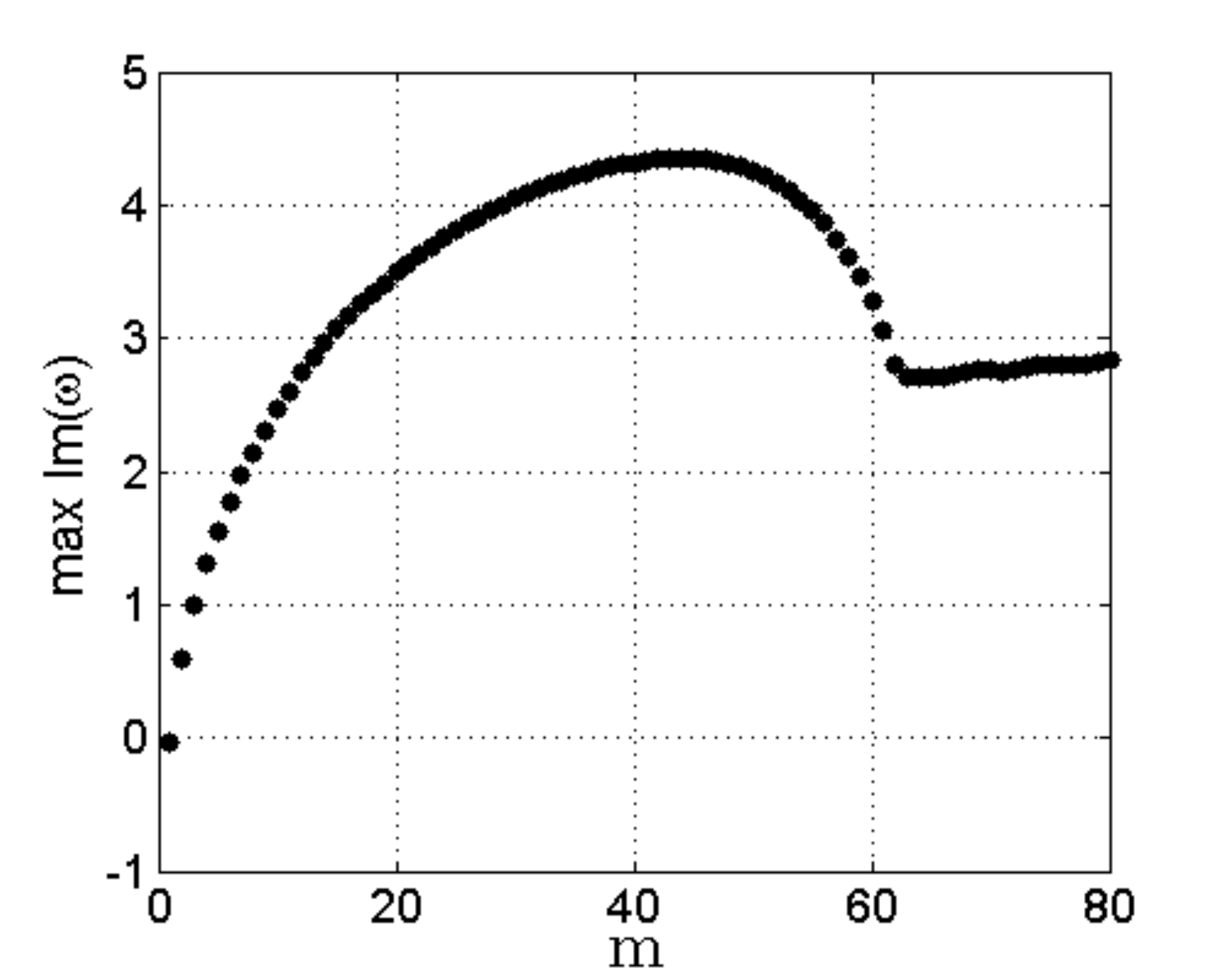}\tabularnewline
\end{tabular}
\par\end{centering}

\caption{a) Density profile of the radially symmetric solution for $R=8$,
with the two-dimensional view shown in the inset. b) Result from the
linear stability analysis: the plot shows the maximum imaginary part
of the eigenvalues, $\omega$, for $m=1,2,...,50$. \label{fig:St_R_8}}
\end{figure}

\begin{figure}[H]
\begin{centering}
\begin{tabular}{cc}
a)\includegraphics[scale=0.34]{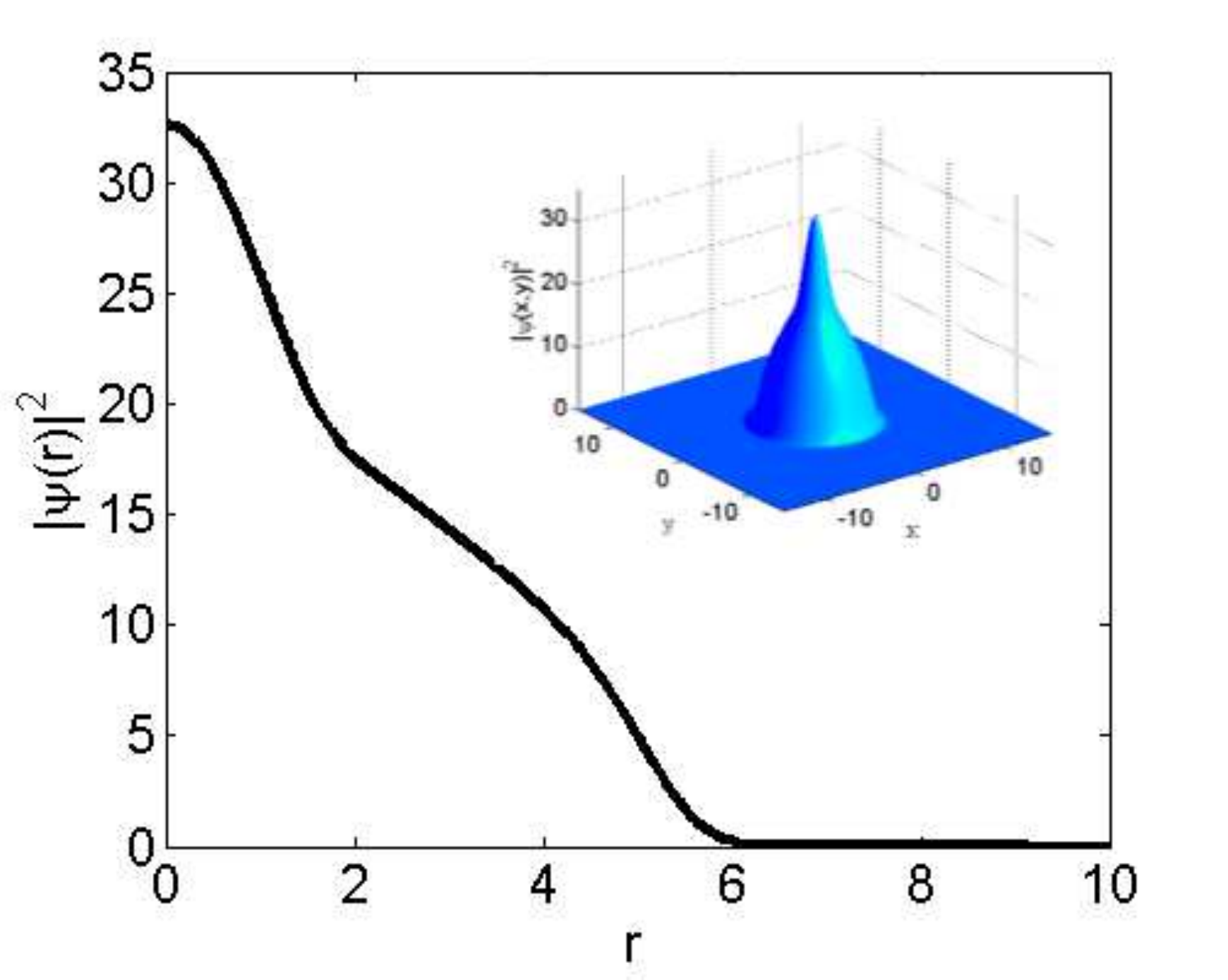} & b)\includegraphics[scale=0.34]{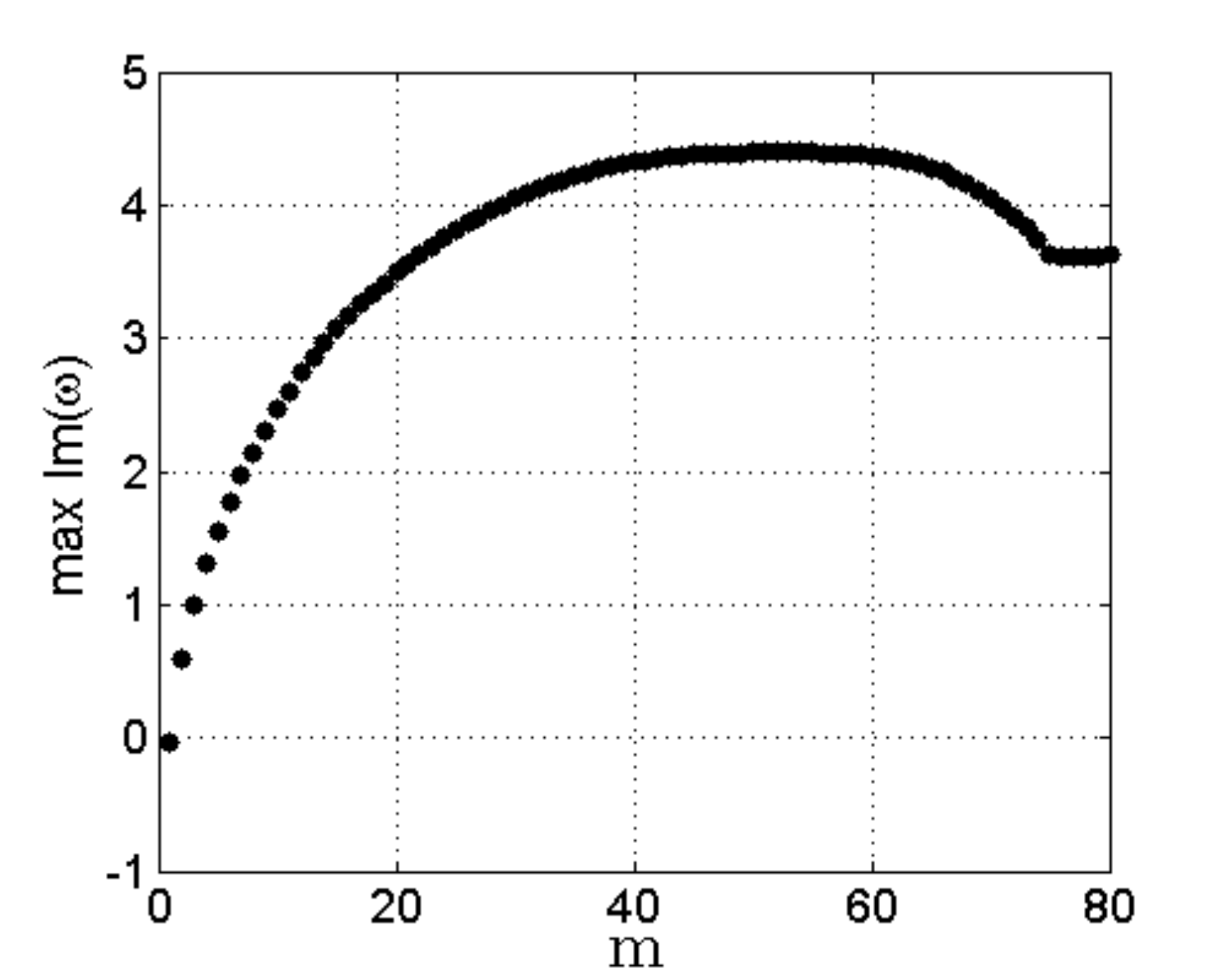}\tabularnewline
\end{tabular}
\par\end{centering}

\caption{a) Density profile of the radially symmetric solution for $R=9$,
with the two-dimensional view shown in the inset. b) Result from the
linear stability analysis: the plot shows the maximum imaginary part
of the eigenvalues, $\omega$, for $m=1,2,...,50$. \label{fig:St_R_9}}
\end{figure}

\begin{figure}[H]
\begin{centering}
\includegraphics[scale=0.3]{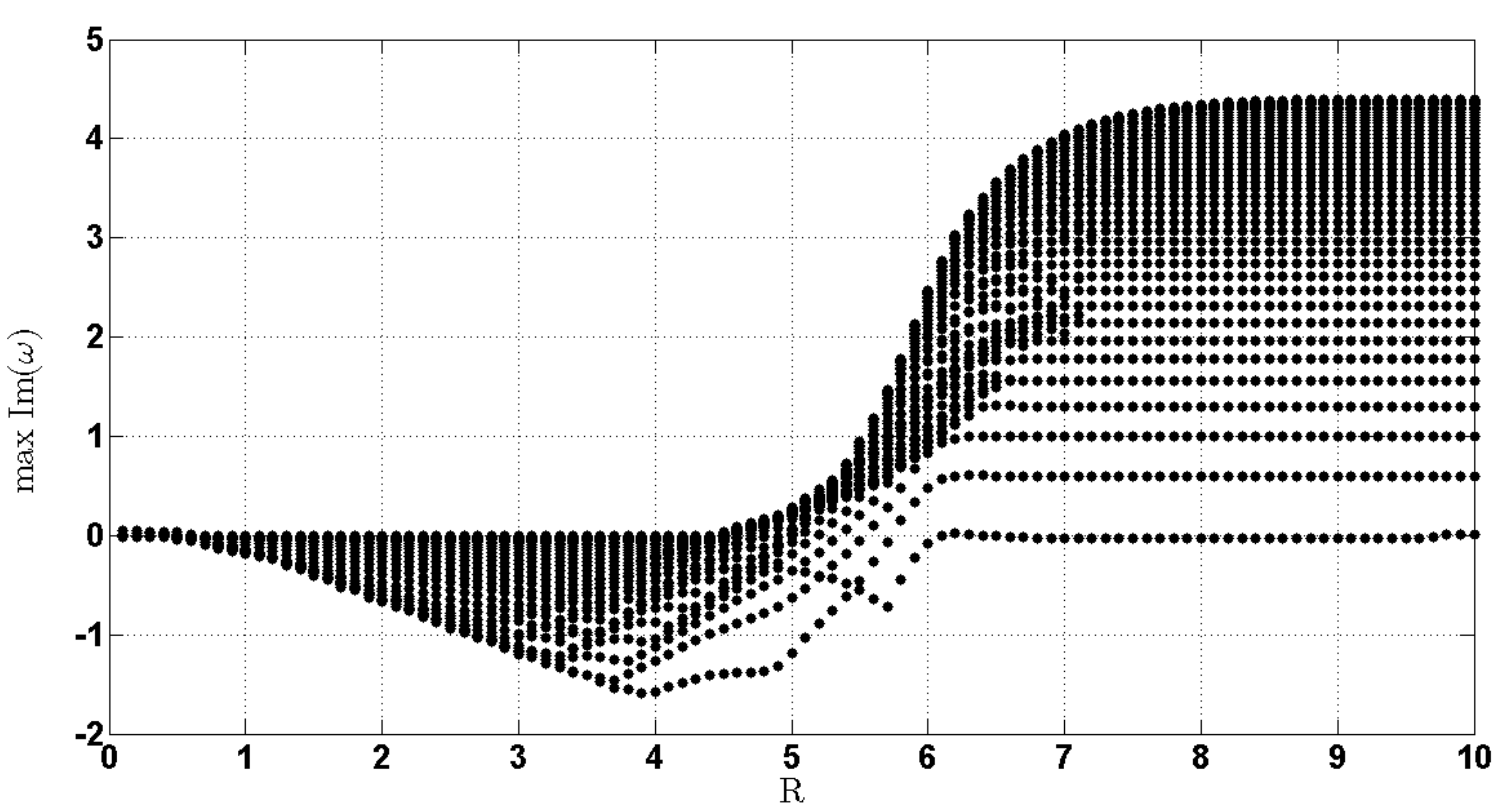}
\par\end{centering}

\caption{Results for $0<R<10$ (with spacing of 0.1) showing the maximum imaginary
part of the eigenvalues $\omega$ for $m=1,2,...,50$. Instability
exists at both $R\lesssim0.6$ and $R\gtrsim4.4$.\label{fig:stab_curve}}
\end{figure}

\begin{figure}[H]
\begin{centering}
\begin{tabular}{ccc}
\includegraphics[scale=0.35]{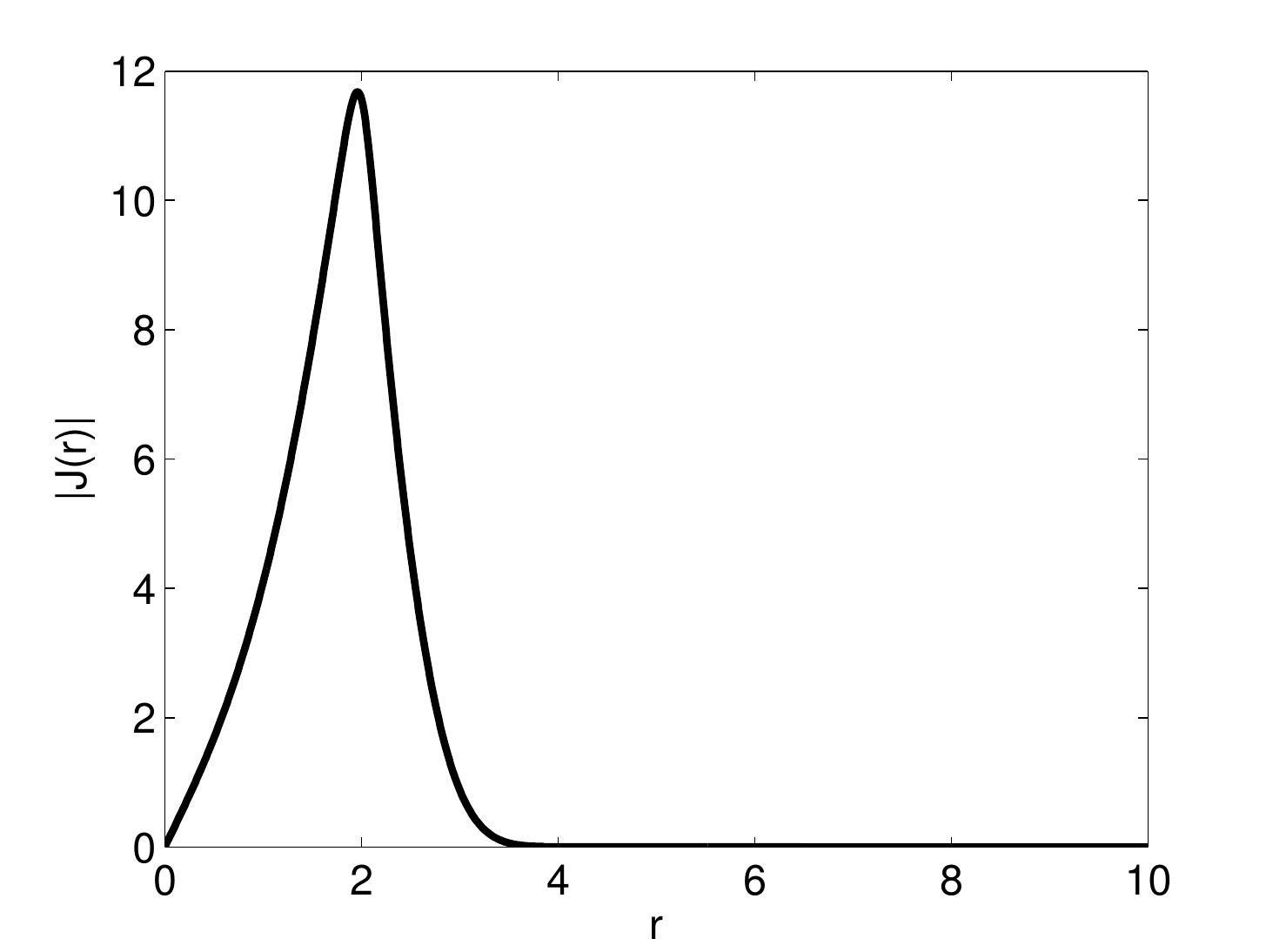} & \includegraphics[scale=0.35]{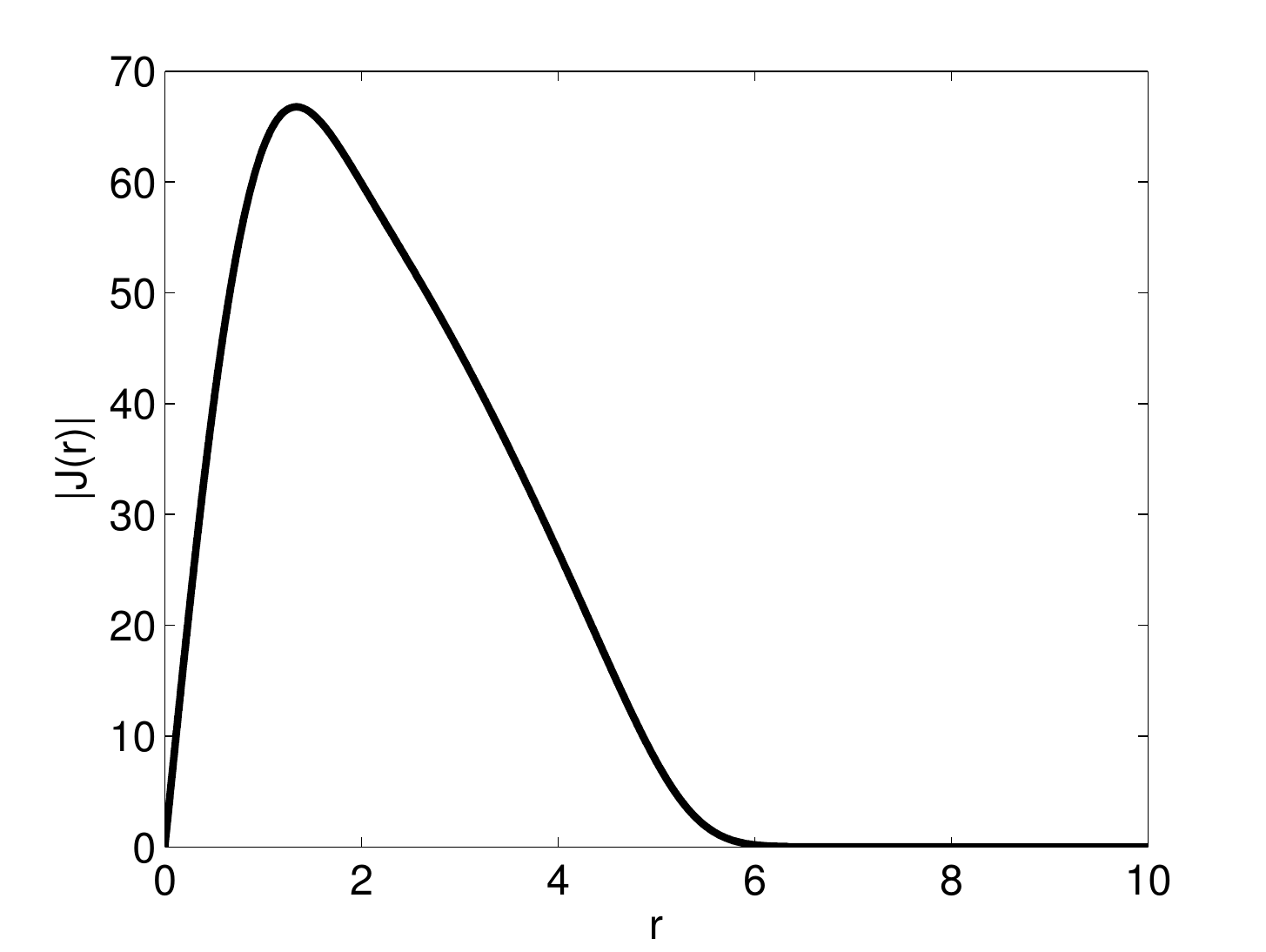} & \includegraphics[scale=0.35]{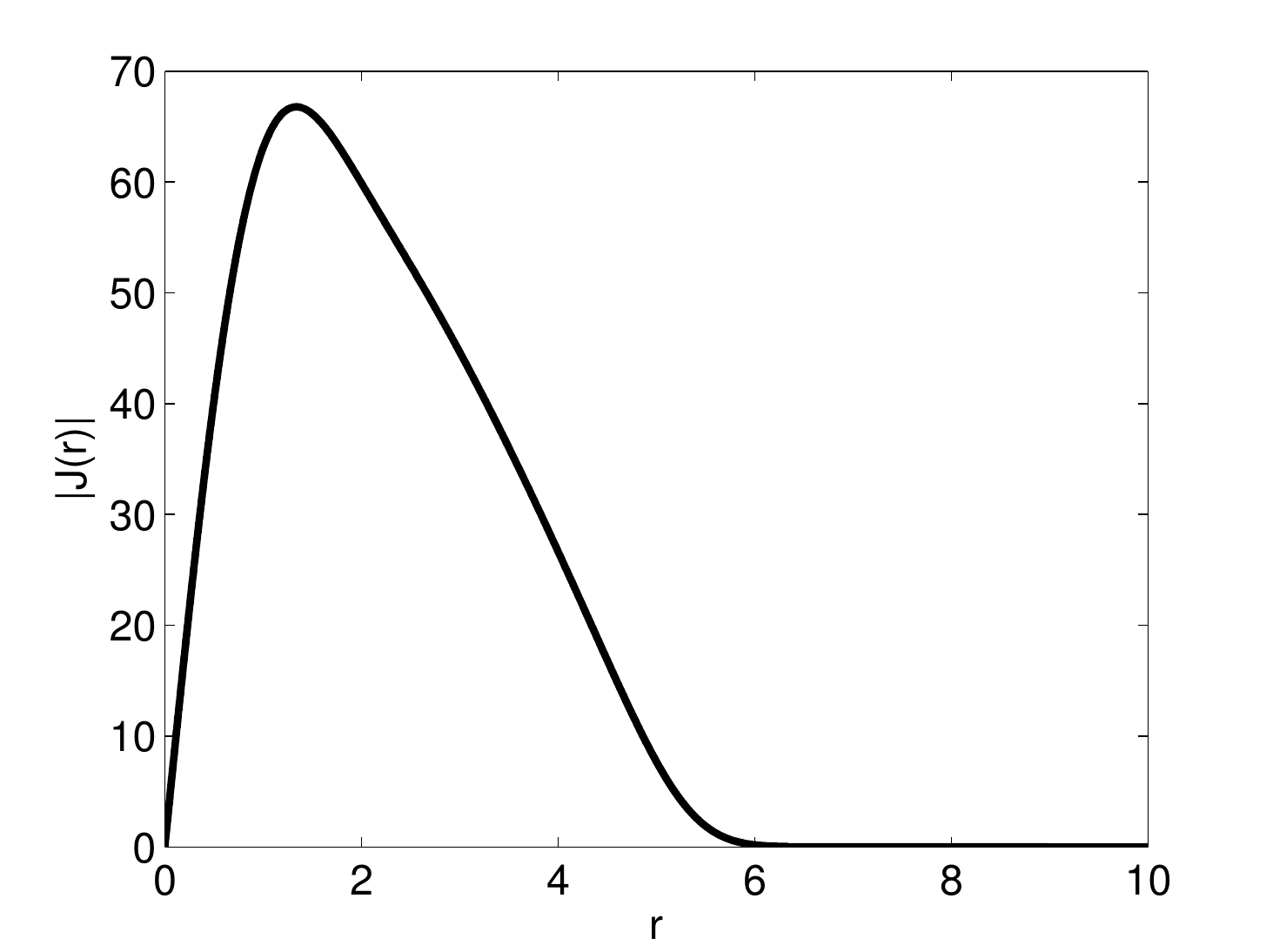}\tabularnewline
\multicolumn{3}{c}{}\tabularnewline
\end{tabular}
\par\end{centering}

\caption{Magnitude of the current $J$ corresponding to the solutions for $R=2$
(left), $R=8$ (center), and $R=9$ (right).\label{fig:currents_R2-R8}}
 
\end{figure}

\section{Numerical continuation for the stationary solutions}

Let us rewrite system (\ref{eq:Sys_2}) as

\begin{equation}
F\left(u,\lambda\right)=0,\label{eq:Sys_2_n}
\end{equation}
where $u$ is the solution of the system ($u\in\mathcal{B}$, a real
Banach space), $\lambda$ is a real parameter ($\lambda$ can be $\alpha$,
$\sigma$, or $R$), and $F$ is a continuously differentiable operator
such that $F:\mathcal{B}\times\mathbb{R}\rightarrow\mathcal{B}$.
The idea in this section is to study the dependence of the solution,
$u\left(\lambda\right)$, on the parameter, $\lambda$, i.e., to trace
the solution branches $\left[u\left(\lambda\right),\lambda\right]$
of (\ref{eq:Sys_2_n}). Since the operator, $F$, is nonlinear in
$u$ and $\lambda$, this can be performed numerically by Newton's
method: suppose that $\left(u_{0},\lambda_{0}\right)$ is a solution
of the discretized problem (\ref{eq:Sys_2_n}) and that the directional
derivative $\dot{u}_{0}=du_{0}/d\lambda$ is known. Then, the solution
$u_{1}$ at $\lambda_{1}=\lambda_{0}+\Delta\lambda$ (where $\Delta$
represents a small increment) can be computed as

\[
\begin{cases}
\begin{array}{c}
F_{u}\left(u_{1}^{i},\lambda_{1}\right)\Delta u_{1}^{i}=-F\left(u_{1}^{i},\lambda_{1}\right),\\
u_{1}^{i+1}=u_{1}^{i}+\Delta u_{1}^{i},\: i=0,1,2,...
\end{array}\end{cases}
\]
with

\[
u_{1}^{0}=u_{0}+\Delta\lambda\dot{u}_{0},
\]
where $F_{u}\left(u,\lambda\right)$ is the Jacobian matrix. 

This method, however, is unable to handle the continuation near singular
points known as turning points or folds, where the solution branch
bends back on itself and $F_{u}\left(u,\lambda\right)$ becomes singular.
These points are particularly important to study the excited states
of the complex GP equation. 

A pseudo-arclength continuation method can be used to overcome this
problem. The main difference from the previous algorithm is the parametrization
of the solution in terms of a new quantity, $\nu$, that approximates
the arclength, $s$, in the tangent space of the curve (instead of
the parametrization by $\lambda$). This is usually accomplished by
appending an auxiliary equation to (\ref{eq:Sys_2_n}) that approximates
the arclength condition

\begin{equation}
\left\Vert \dot{u}\left(s\right)\right\Vert ^{2}+\left|\dot{\lambda}\left(s\right)\right|^{2}=1.\label{eq:arc_cond}
\end{equation}
This leads to the augmented system

\begin{equation}
\begin{cases}
\begin{array}{c}
F\left(u\left(\nu\right),\lambda\left(\nu\right)\right)=0,\\
N\left(u\left(\nu\right),\lambda\left(\nu\right),\nu\right)=0,
\end{array}\end{cases}\label{eq:Sys_aug}
\end{equation}
with unknowns $u\left(\nu\right)$ and $\lambda\left(\nu\right)$.
Then, Newton's method (or one of its variants) can be used to solve
(\ref{eq:Sys_aug}), in which case the Jacobian of the system becomes
the bordered matrix:

\[
G=\left[\begin{array}{cc}
F_{u} & F_{\lambda}\\
N_{u} & N_{\lambda}
\end{array}\right].
\]
For a detailed description of bordered matrices, see, e.g., \citet{govaerts2000numerical}.
The function $N$ is defined such that the Jacobian $G$ is nonsingular
on the solution branch, including turning points and their neighborhoods.
$N\left(u,\lambda,\nu\right)\equiv\dot{u}_{0}^{T}\left(u-u_{0}\right)+\dot{\lambda}_{0}\left(\lambda-\lambda_{0}\right)-\nu$
is one of the most frequently used definitions of $N$. It was introduced
by \citet{keller1977numerical}. With this particular definition of
$N$, the method is known as Keller's algorithm.

The pseudo-arclength continuation method can be summarized as follows:
the unit tangent $\left(\dot{u}_{0},\dot{\lambda}_{0}\right)$ at
$\left(u_{0},\lambda_{0}\right)$ is obtained from its definition:

\[
\begin{cases}
\begin{array}{c}
F_{u}\left(u_{0},\lambda_{0}\right)\dot{u}_{0}+F_{\lambda}\left(u_{0},\lambda_{0}\right)\dot{\lambda}_{0}=0,\\
\left\Vert \dot{u}_{0}\right\Vert ^{2}+\left|\dot{\lambda}_{0}\right|^{2}=1.
\end{array}\end{cases}
\]
Then, an approximate solution $\left(u_{a},\lambda_{a}\right)$ is
computed as

\[
\begin{array}{c}
u_{a}=u_{0}+\nu\dot{u}_{0},\\
\lambda_{a}=\lambda_{0}+\nu\dot{\lambda}_{0},
\end{array}
\]
which is used as an initial guess in a Newton-like method for solving
(\ref{eq:Sys_aug}) to obtain ($u\left(\nu\right),\lambda\left(\nu\right)$).

The software package AUTO was used for the continuation of the two-dimensional
radially symmetric solutions of the complex GP equation. AUTO implements
Keller's pseudo-arclength continuation algorithm. A detailed description
of this algorithm can be found in \citet{keller1987lectures}. In
addition to detecting turning points, AUTO can also find branch points,
a feature that will be used to study the excited states of the complex
GP equation. For the documentation of AUTO, see \citet{doedelauto}.

AUTO discretizes system (\ref{eq:Sys_2}) using polynomial collocation
with Gaussian points. The number of free parameters, $p$, controlled
by AUTO during the continuation process is given by $p=n_{bc}+1-n$,
where $n_{bc}$ is the number of boundary conditions and $n$ is the
dimension of the system. Hence, the continuation for system (\ref{eq:Sys_2})
can be done in one parameter. For the results presented in this section,
the number of mesh intervals is 6000, the number of Gaussian collocation
points per mesh interval is 7, and the maximum absolute value of the
pseudo-arclength stepsize is 0.0001. The parallel version of AUTO
was run on a Beowulf-class heterogeneous computer cluster using 16
nodes for a total of 128 cores (Intel Xeon X5570, 2.93GHz).

Figure (\ref{fig:Cont_sigma_for_mu}) shows the chemical potential,
$\mu$, of the different solutions obtained by the continuation in
$\sigma$ with $\alpha=4.4$ and $R=2$ fixed. By using different
initial guesses (multi-bump profiles) in the collocation method described
in Section 3, it was possible to find some of the solutions shown
in Figs. (\ref{fig:Ex_states}) and (\ref{fig:Cont_ex_states}). These
solutions, along with the capability of AUTO to detect branch points
and turning points, were employed to produce the result displayed
in Fig. (\ref{fig:Cont_sigma_for_mu}). In this figure, the red dashed
line indicates the solutions corresponding to $\sigma=0.3$, $\alpha=4.4$
and $R=2$. Figs. (\ref{fig:Ex_states}) and (\ref{fig:Cont_ex_states})
show the density profiles of these solutions. Notice that the linearly
stable solution studied in Section 4 has the smallest chemical potential.
In this context, this solution can be considered as the ground state
solution, whereas the other solutions are excited states. In any case,
an appropriate definition of the energy for this equation is necessary
to determine a ground state solution. From the numerical results obtained
up to now, we expect that the solution labeled 0 in Figs. (\ref{fig:Cont_sigma_for_mu})
and (\ref{fig:Ex_states}) is the minimizer of such energy. 

Equation (\ref{eq:eq_mu}) can also be used for the computation of
the chemical potential. In this case

\begin{equation}
\mu\int_{0}^{b}\left|\phi\left(r\right)\right|^{2}rdr=\int_{0}^{b}\left[\left|\phi'\left(r\right)\right|^{2}+r^{2}\left|\phi\left(r\right)\right|^{2}+\left|\phi\left(r\right)\right|^{4}\right]rdr,\label{eq:exp_mu}
\end{equation}
where the integration domain is $\left[0,b\right]$. Using Simpson's
rule, this expression was used for a second validation of the results
presented in Fig. (\ref{fig:Cont_sigma_for_mu}).

\begin{figure}[H]
\begin{centering}
\includegraphics[scale=0.65]{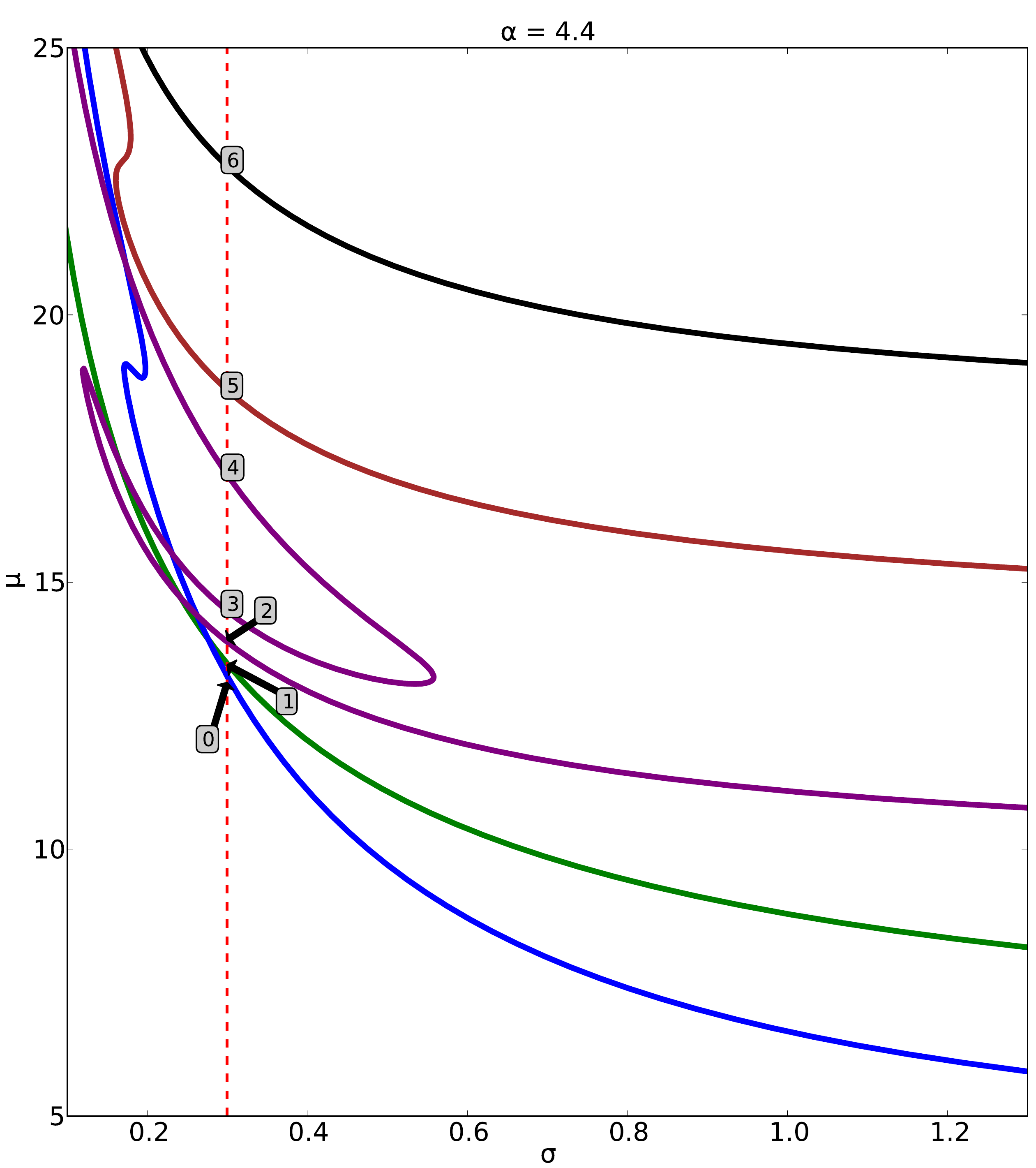}
\par\end{centering}

\caption{Result from the continuation method showing the chemical potential
of the system for different values of $\sigma$ with $\alpha=4.4$
and $R=2$ fixed. The red dashed line indicates the solutions for
$\sigma=0.3$. Figs. (\ref{fig:Ex_states}) and (\ref{fig:Cont_ex_states})
show the density profiles corresponding to these solutions. \label{fig:Cont_sigma_for_mu}}
\end{figure}

\begin{figure}[H]
\begin{centering}
\begin{tabular}{cc}
\includegraphics[scale=0.4]{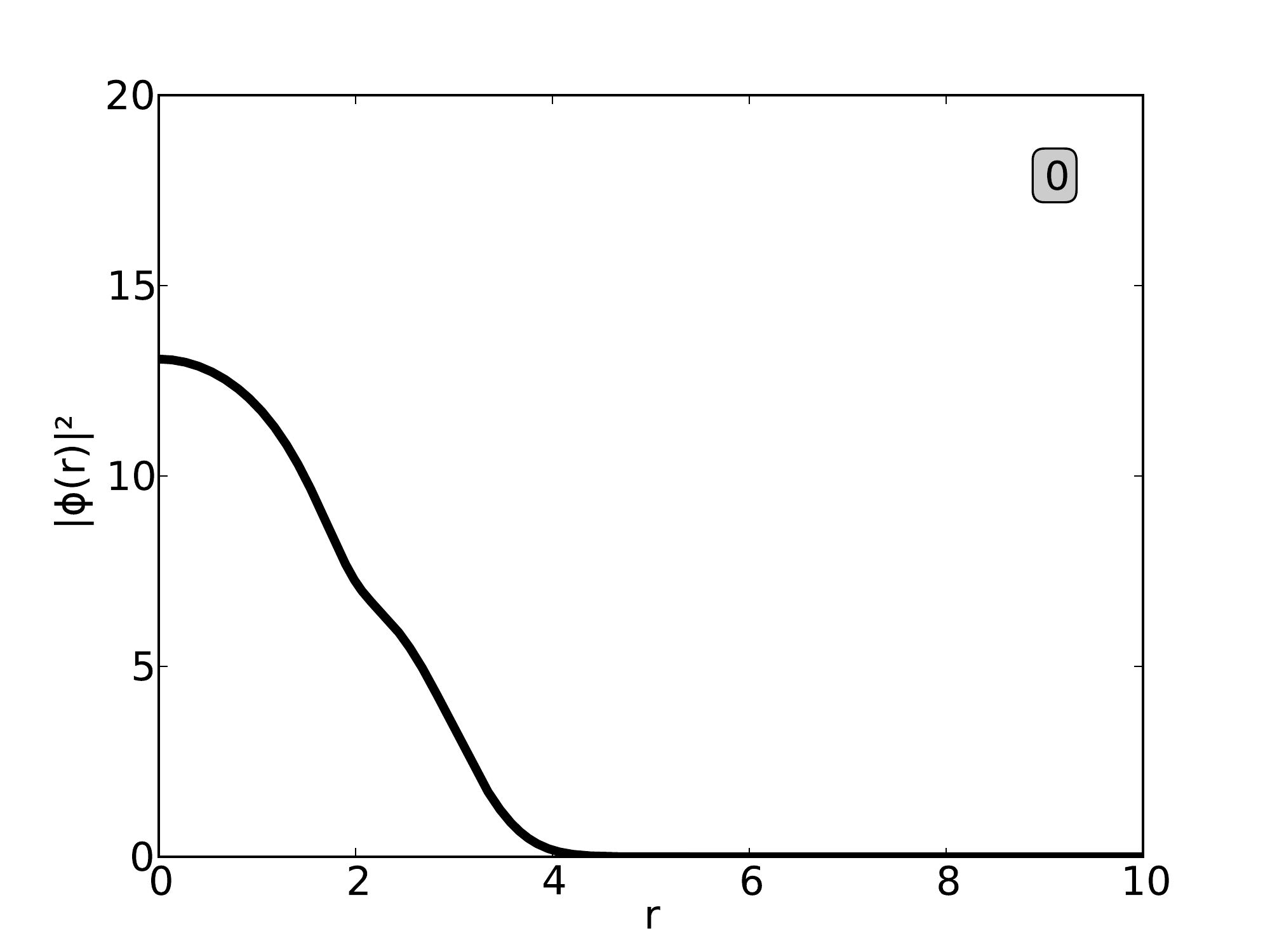} & \includegraphics[scale=0.4]{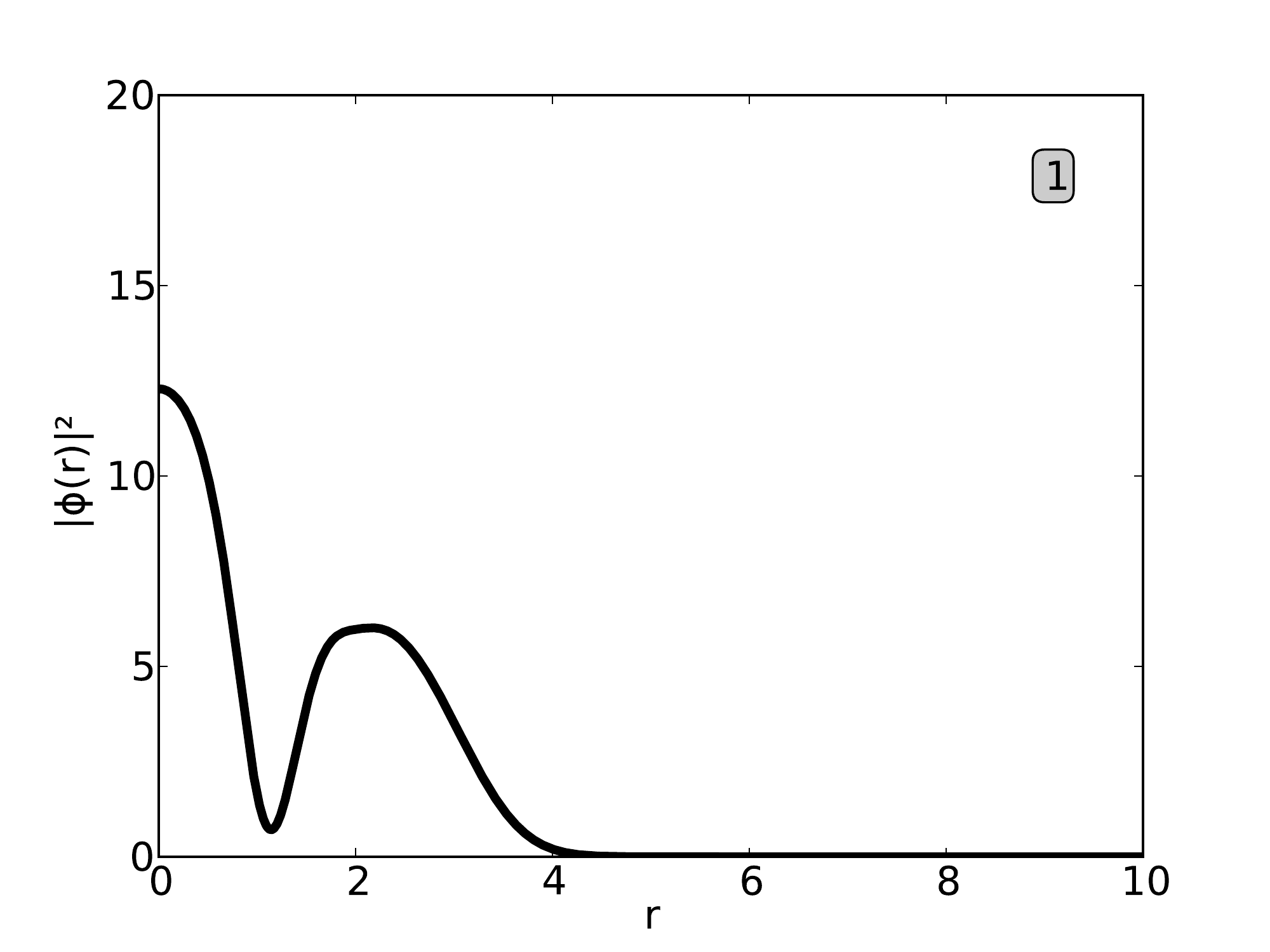}\tabularnewline
\includegraphics[scale=0.4]{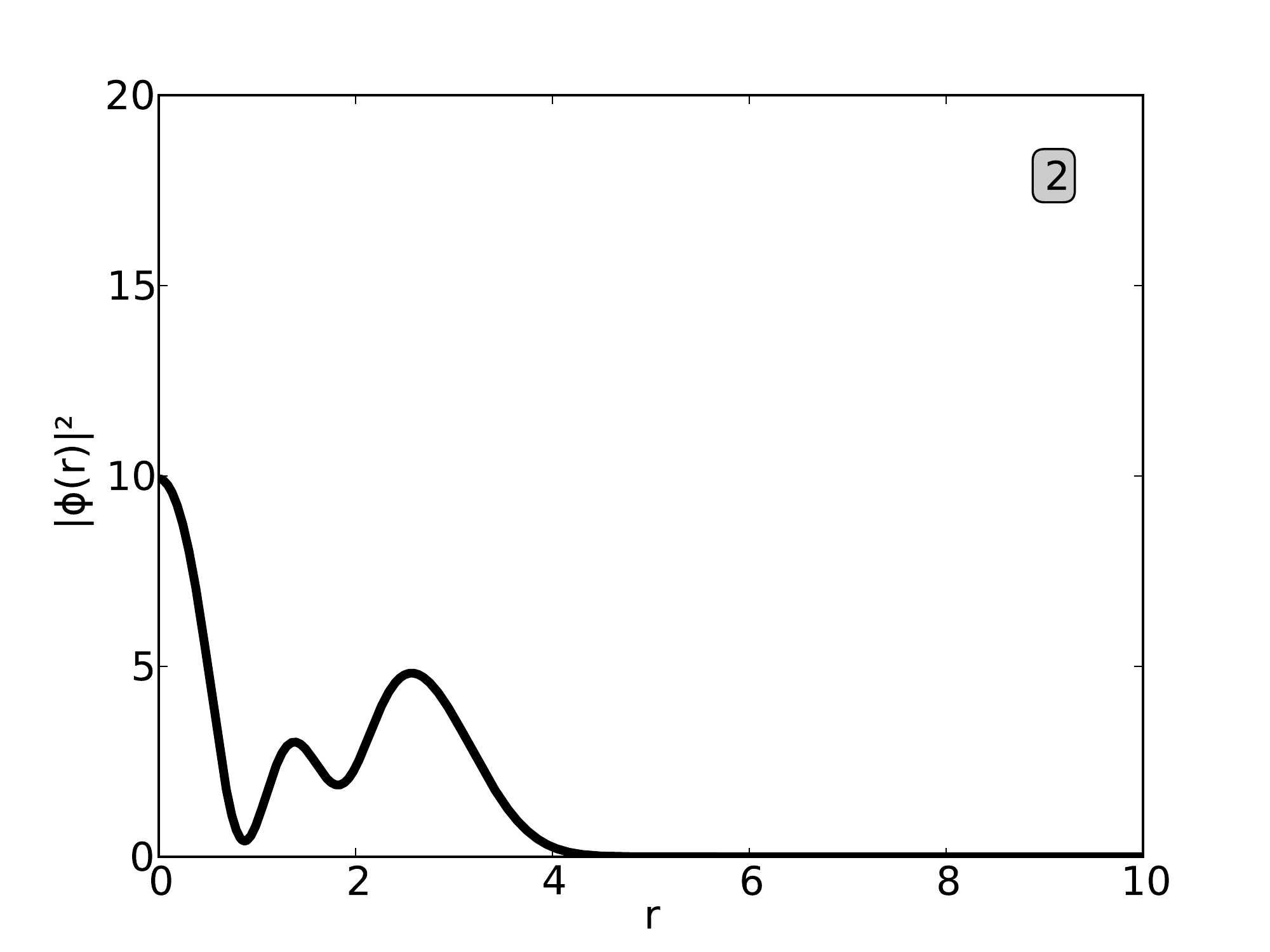} & \includegraphics[scale=0.4]{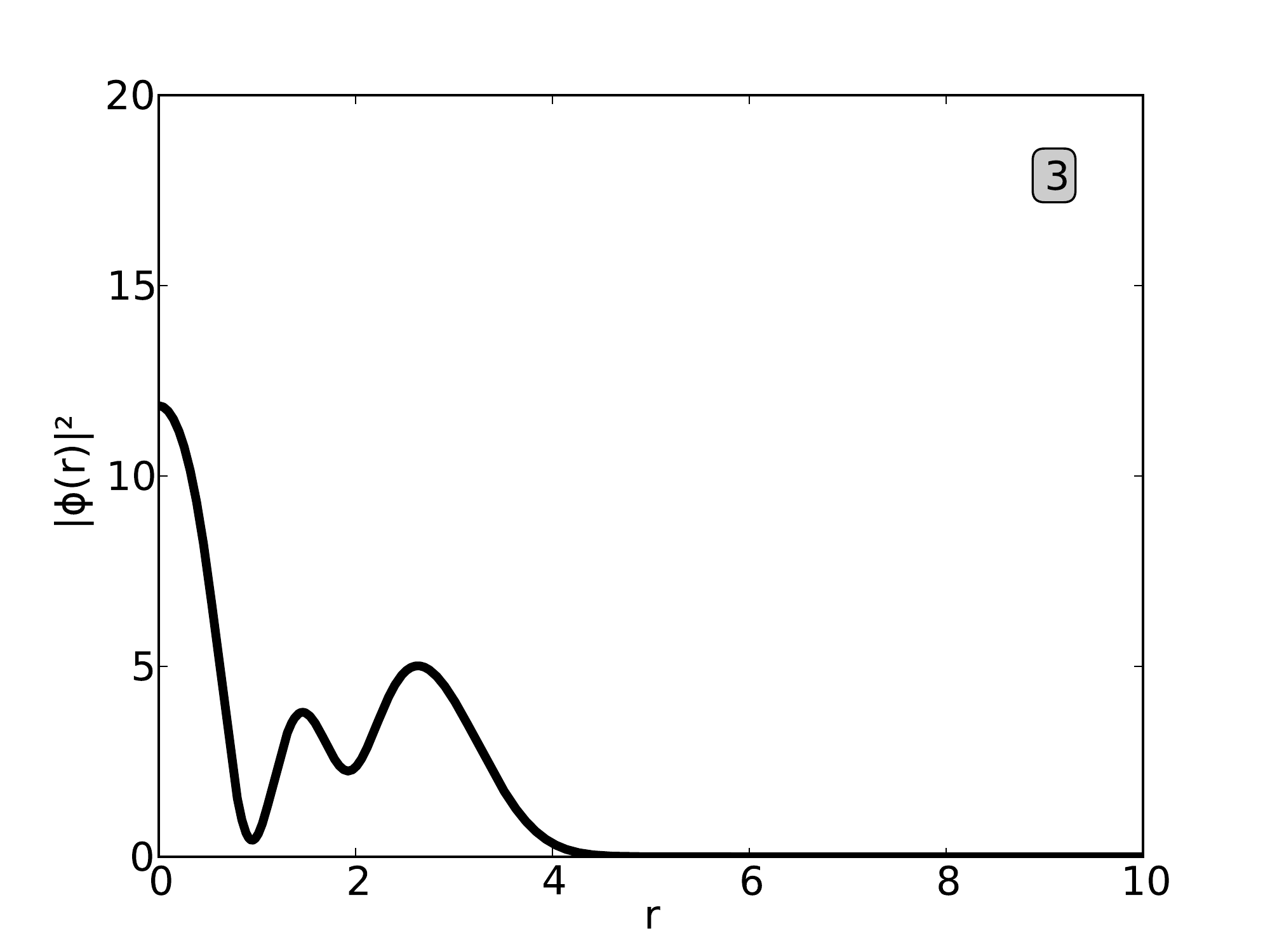}\tabularnewline
\includegraphics[scale=0.4]{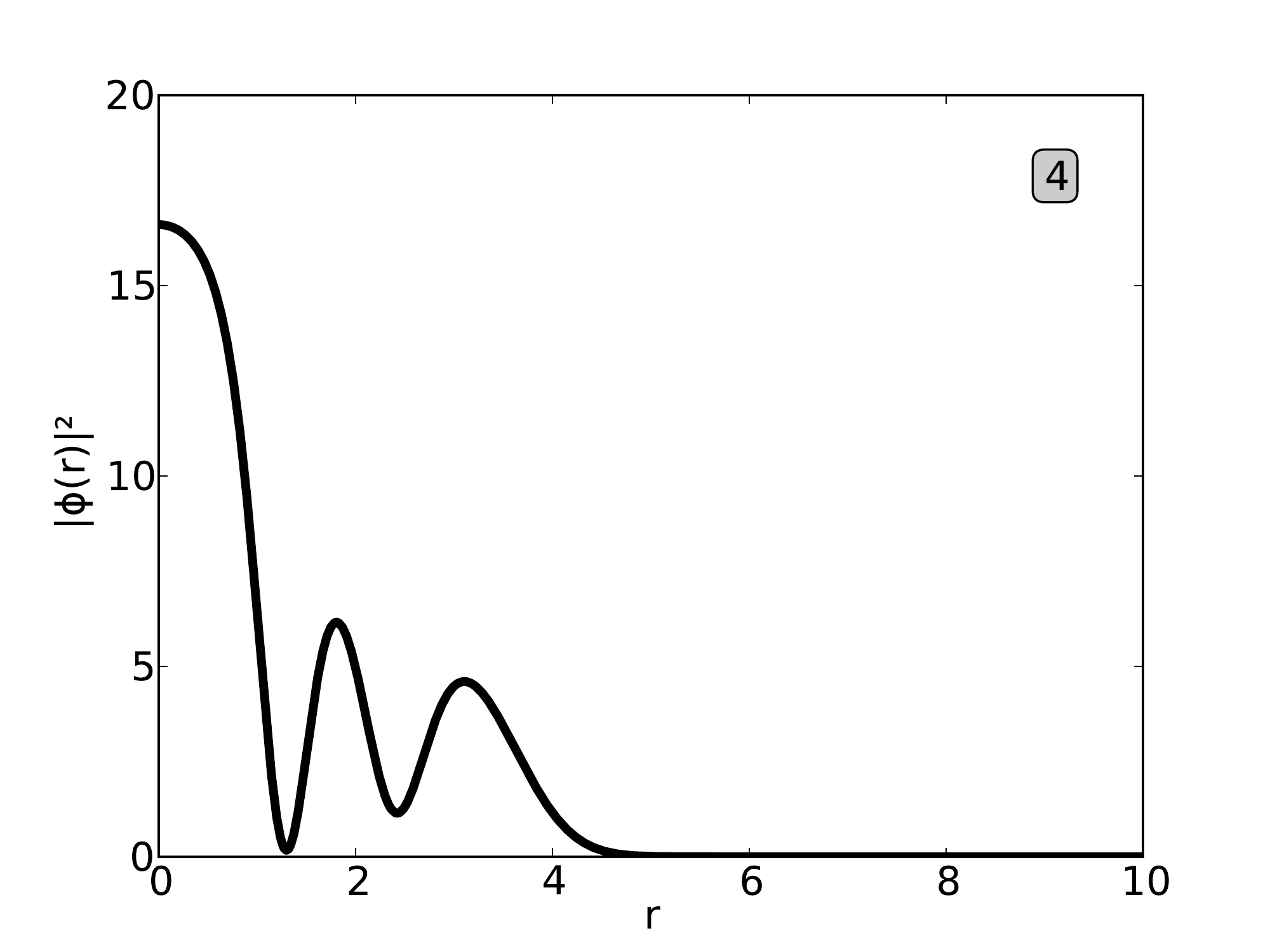} & \includegraphics[scale=0.4]{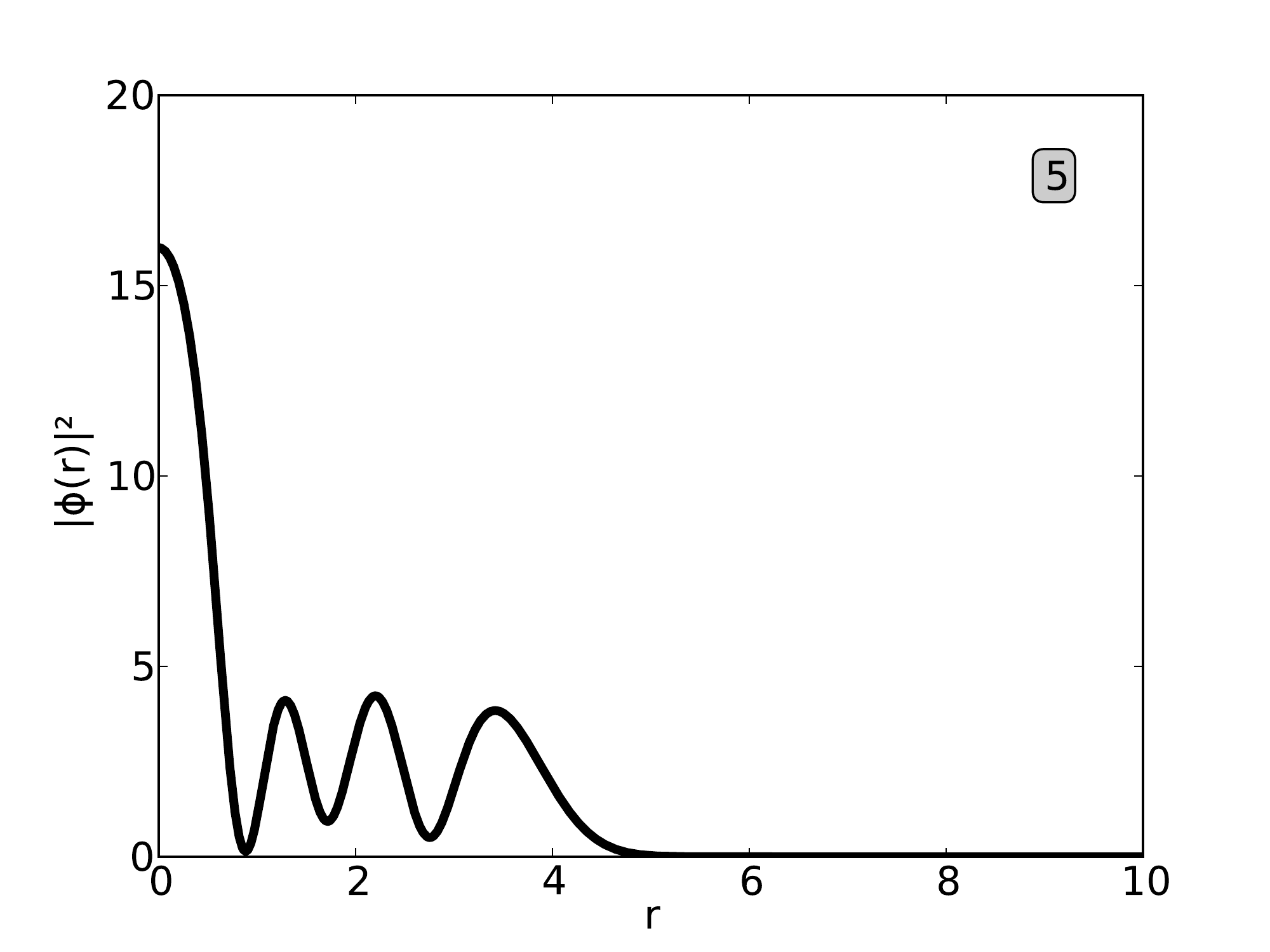}\tabularnewline
\end{tabular}
\par\end{centering}

\caption{Density profiles of the solutions labeled 0 to 5 in Fig. (\ref{fig:Cont_sigma_for_mu})
. \label{fig:Ex_states}}
\end{figure}

\begin{figure}[H]
\begin{centering}
\includegraphics[scale=0.4]{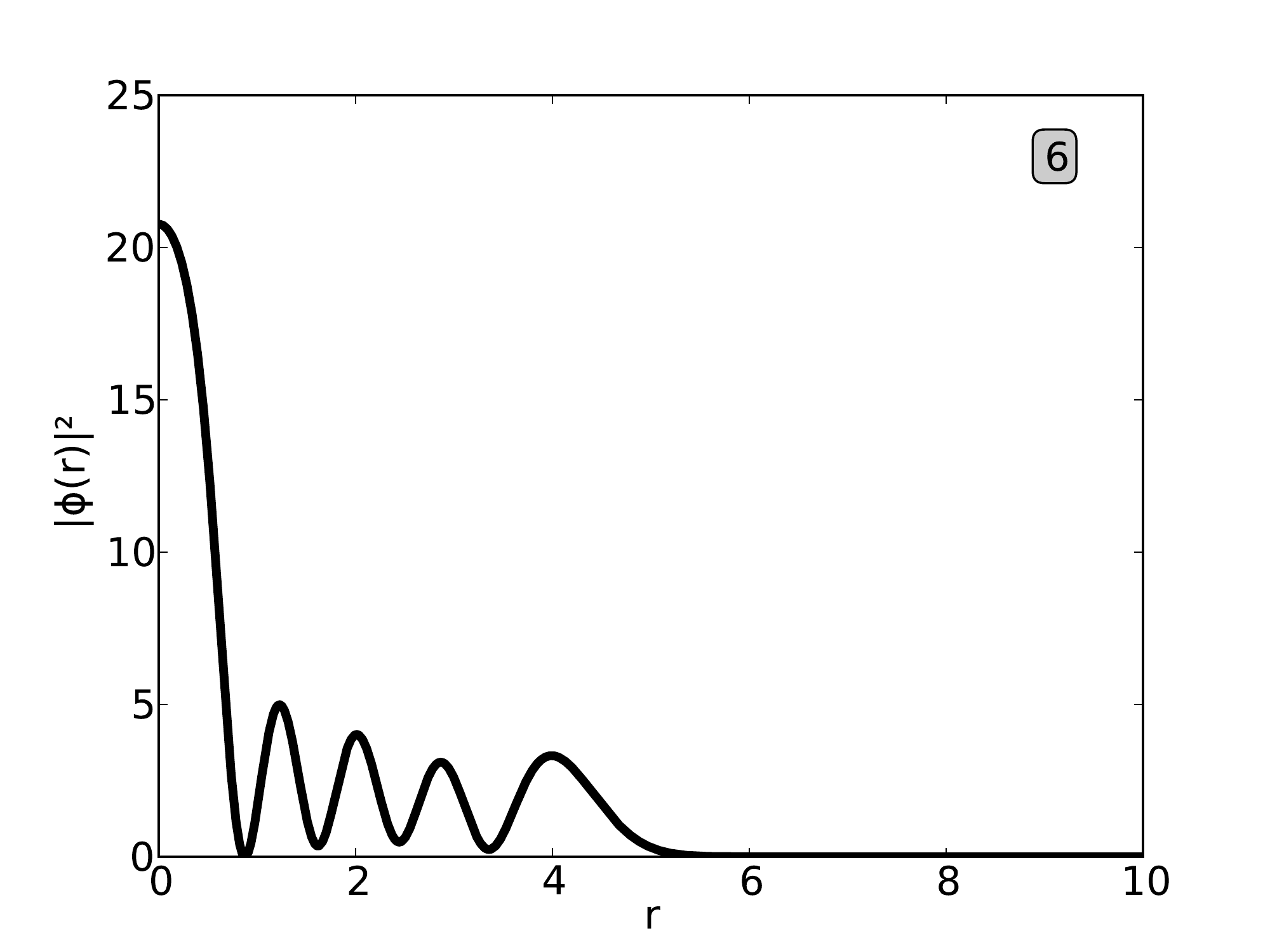}
\par\end{centering}

\caption{Density profile of the solution labeled 6 in Fig. (\ref{fig:Cont_sigma_for_mu})
\label{fig:Cont_ex_states} }
\end{figure}

\section{Numerical integration of the complex GP equation}

In this section, the numerical integration of the complex GP equation
is carried out by using a Strang-splitting Fourier spectral method.
This method was studied by \citet{bao2002time} for the Schrödinger
equation in the semiclassical regime and by \citet{bao2003numerical}
for the GP equation. For simplicity of notation, the method is presented
in one dimension. The generalization to higher dimensions is straightforward
using the tensor product grids. The problem can be stated as
\begin{equation}
\begin{array}{c}
i\psi_{t}=-\psi_{xx}+V(x)\psi+\left|\psi\right|^{2}\psi+i\left(\alpha\Theta_{R}-\sigma\left|\psi\right|^{2}\right)\psi,\\
\psi(x,0)=\psi_{0}\left(x\right),\quad a\leq x\leq b,\quad\psi(a,t)=\psi(b,t),\quad\psi_{x}(a,t)=\psi_{x}(b,t),\quad t>0.
\end{array}\label{eq:1D_TSSM_cGPE}
\end{equation}
Notice that in this case, periodic boundary conditions are specified,
which are necessary for the implementation, because the discrete Fourier
transform is used in one of the steps. In addition, $\psi$ is concentrated
due to the trapping potential and vanishes very quickly as $\left|x\right|$
increases, making these boundary conditions a good option provided
that the domain of integration is large enough.

Consider $h=\Delta x>0$ as the mesh size, with $h=(b-a)/M$, where
$M$ is an even positive integer. Define $\tau=\Delta t>0$ as the
time step. Let the grid points be $x_{j}=a+jh,\: j=0,1,...,M$, and
the time $t_{n}=n\tau,\: n=0,1,2,...$ Finally, define $\Psi_{j}^{n}$
as the numerical approximation of $\psi(x_{j},t_{n})$. Equation (\ref{eq:1D_TSSM_cGPE})
is now separated into

\begin{equation}
i\psi_{t}=V(x)\psi+\left|\psi\right|^{2}\psi+i\left(\alpha\Theta_{R}-\sigma\left|\psi\right|^{2}\right)\psi\label{eq:sp_1}
\end{equation}
and

\begin{equation}
i\psi_{t}=-\psi_{xx},\label{eq:sp_2}
\end{equation}
which are combined via Strang-splitting to approximate the solution
of (\ref{eq:1D_TSSM_cGPE}) on $\left[t_{n},t_{n+1}\right]$.

Step 1 of the Strang-splitting method requires the solution of (\ref{eq:sp_1})
on a half time step, from $t_{n}$ to $t_{n}+\tau/2$. Due to the
smooth Heaviside $\Theta$ in (\ref{eq:sp_1}), two equations are
solved: for $\Theta_{R}>0$

\begin{equation}
i\psi_{t}=V(x)\psi+\left|\psi\right|^{2}\psi+i\left(\alpha\Theta_{R}-\sigma\left|\psi\right|^{2}\right)\psi,\label{eq:Sp1a}
\end{equation}
and for $\Theta_{R}=0$

\begin{equation}
i\psi_{t}=V(x)\psi+\left|\psi\right|^{2}\psi-i\sigma\left|\psi\right|^{2}\psi.\label{eq:Sp1b}
\end{equation}
Substituting $\psi\left(x,t\right)=\rho\left(x,t\right)e^{i\theta\left(x,t\right)}$
in (\ref{eq:Sp1a}), where $\rho$ is the magnitude and $\theta$
is the phase of $\psi$, gives

\begin{equation}
i\rho_{t}-\rho\theta_{t}=V(x)\rho+\rho^{3}+i\left(\alpha\Theta_{R}-\sigma\rho^{2}\right)\rho.\label{eq:Sub1}
\end{equation}
Separating real and imaginary parts from (\ref{eq:Sub1}) and looking
for non-trivial solutions lead to

\begin{equation}
\theta_{t}=-V(x)-\rho^{2},\label{eq:RealSub1}
\end{equation}

\begin{equation}
\rho_{t}=\left(\alpha\Theta_{R}-\sigma\rho^{2}\right)\rho.\label{eq:ImagSub1}
\end{equation}
The solution of (\ref{eq:ImagSub1}) is given by

\[
\rho=\sqrt{\frac{\rho_{0}^{2}e^{2\alpha\Theta_{R}t}}{1+\frac{\sigma}{\alpha}\rho_{0}^{2}\left(e^{2\alpha\Theta_{R}t}-1\right)/\Theta_{R}}},
\]
where $\rho_{0}=\rho\left(x,0\right)$ . With this result, (\ref{eq:RealSub1})
can be solved, obtaining

\[
\theta=-V(x)t-\frac{1}{2\sigma}\mathrm{ln}\left|1+\frac{\sigma}{\alpha}\rho_{0}^{2}\left(e^{2\alpha\Theta_{R}t}-1\right)/\Theta_{R}\right|+\theta_{0},
\]
where $\theta_{0}$ is the phase of the initial condition.

Following the same procedure for (\ref{eq:Sp1b}) yields

\[
\rho=\frac{\rho_{0}}{\sqrt{2\sigma\rho_{0}^{2}t+1}},
\]

\[
\theta=-V(x)t-\frac{1}{2\sigma}\mathrm{ln}\left|2\sigma\rho_{0}^{2}t+1\right|+\theta_{0}.
\]
Hence, defining $\Theta_{j}\equiv\Theta\left(R-\left|x_{j}\right|\right)$,
the solution of Step 1 can be written as:

\[
\Psi_{j}^{(1)}=\left\{ \begin{array}{cc}
\Psi_{j}^{n}U_{j}^{(1)}e^{i\theta_{j}^{(1)}}, & \Theta_{j}>0,\\
\Psi_{j}^{n}W_{j}^{(1)}e^{i\phi_{j}^{(1)}}, & \Theta_{j}=0,
\end{array}\right.
\]
where

\[
U_{j}^{(1)}=\sqrt{\frac{e^{\alpha\Theta_{j}\tau}}{1+\frac{\sigma}{\alpha}\left|\Psi_{j}^{n}\right|^{2}\left(e^{\alpha\Theta_{j}\tau}-1\right)/\Theta_{j}}},
\]

\[
\theta_{j}^{(1)}=-V(x_{j})\tau/2-\frac{1}{2\sigma}\mathrm{ln}\left|1+\frac{\sigma}{\alpha}\left|\Psi_{j}^{n}\right|^{2}\left(e^{\alpha\Theta_{j}\tau}-1\right)/\Theta_{j}\right|,
\]

\[
W_{j}^{\left(1\right)}=\frac{1}{\sqrt{\sigma\left|\Psi_{j}^{n}\right|^{2}\tau+1}},
\]

\[
\phi_{j}^{(1)}=-V(x_{j})\tau/2-\frac{1}{2\sigma}\mathrm{ln}\left|\sigma\left|\Psi_{j}^{n}\right|^{2}\tau+1\right|.
\]

For Step 2, equation (\ref{eq:sp_2}) has to be solved on a complete
time step, from $t_{n}$ to $t_{n}+\tau$, using the solution of Step
1 as the initial condition. Equation (\ref{eq:sp_2}) is discretized
in space by the Fourier spectral method and integrated in time exactly,
giving

\[
\Psi_{j}^{(2)}=\frac{1}{M}\sum_{l=-M/2}^{M/2-1}e^{-i\omega_{l}^{2}\tau}\widehat{\Psi}_{l}^{(1)}e^{i\omega_{l}(x_{j}-a)},\quad j=0,1,2,...,M-1,
\]
where $\widehat{\Psi}_{l}^{(1)}$, the Fourier coefficients of $\Psi^{(1)}$,
are defined as

\[
\widehat{\Psi}_{l}^{(1)}=\sum_{j=0}^{M-1}\Psi_{j}^{(1)}e^{-i\omega_{l}(x_{j}-a)},\quad\omega_{l}=\frac{2\pi l}{b-a},\quad l=-\frac{M}{2},...,\frac{M}{2}-1.
\]

Finally, Step 3 requires the solution of (\ref{eq:sp_1}) on a half-time
step, from $t_{n}+\tau/2$ to $t_{n}+\tau$. Hence, the same expression
as in step 1 is used, but with $\Psi^{\left(2\right)}$ as an initial
condition. The result from this step corresponds to $\Psi^{n+1}$.

The method just presented is second-order accurate in time (due to
the Strang splitting) and spectrally accurate in space. Its stability
can be studied following the ideas in \citet{bao2003numerical}. Let
$\left\Vert \cdot\right\Vert _{l^{2}}$ be the discrete $l^{2}-$norm
$\left\Vert \Psi\right\Vert _{l^{2}}=\sqrt{\frac{b-a}{M}\sum_{j=0}^{M-1}\left|\psi_{j}\right|^{2}}$.
We get

\[
\left\Vert \Psi^{n+1}\right\Vert _{l^{2}}^{2}\leq e^{2\alpha\tau}\left\Vert \Psi^{n}\right\Vert _{l^{2}}^{2},
\]
which indicates that the method is unconditionally stable in the sense
of Lax–Richtmyer. 

For all the methods presented so far, it is important to emphasize
the use of the smooth Heaviside function. The original complex GP
equation proposed by \citet{keeling2008spontaneous} includes a Heaviside
function for the pumping part. With the latter, it is easy to see
that the radially symmetric solutions have a discontinuity in the
second derivative. This discontinuity reduces the accuracy of the
collocation method and produces the Gibbs phenomenon in the splitting
method due to the spectral part.

\subsection{Simulation results}

The results corresponding to the 2D case are presented in this section.
The integration domain is $\left[-15,15\right]\times\left[-15,15\right]$,
with $1024$ divisions per axis and $\tau=0.001$. The ground state
solution of the harmonic oscillator is used as an initial condition
in these simulations. 

For $R=2,3,4$, the solutions are displayed in Fig. \ref{fig:Sim_R2}.
In these cases, the system reaches the rotationally symmetric stationary-state
as time evolves. This is the expected behavior, since it was shown
in Section 4 that the radially symmetric solutions for $R=2,3,4$
are stable. Also, Fig. \ref{fig:Sim_R2} shows the comparison with
the corresponding solutions obtained by the collocation method introduced
in Section 3 (red line). Notice that both methods give the same result.
Moreover, Fig (\ref{fig:currents}) shows the magnitude of the current,
$\left|J\right|$, for these cases. Note that even for these radially
symmetric solutions, the current exhibits complicated behavior.

\begin{figure}[H]
\begin{centering}
\begin{tabular}{ccc}
\includegraphics[scale=0.3]{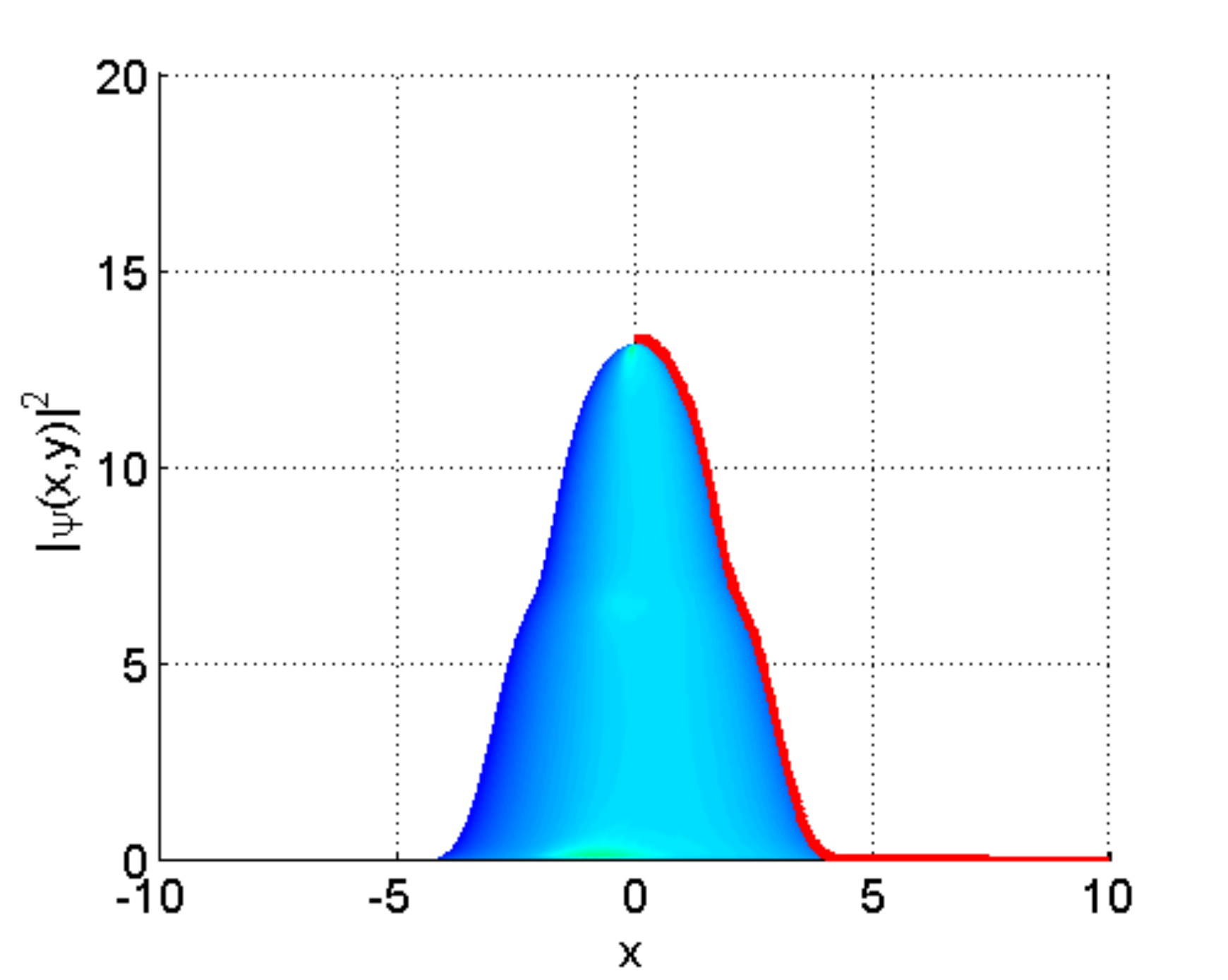} & \includegraphics[scale=0.3]{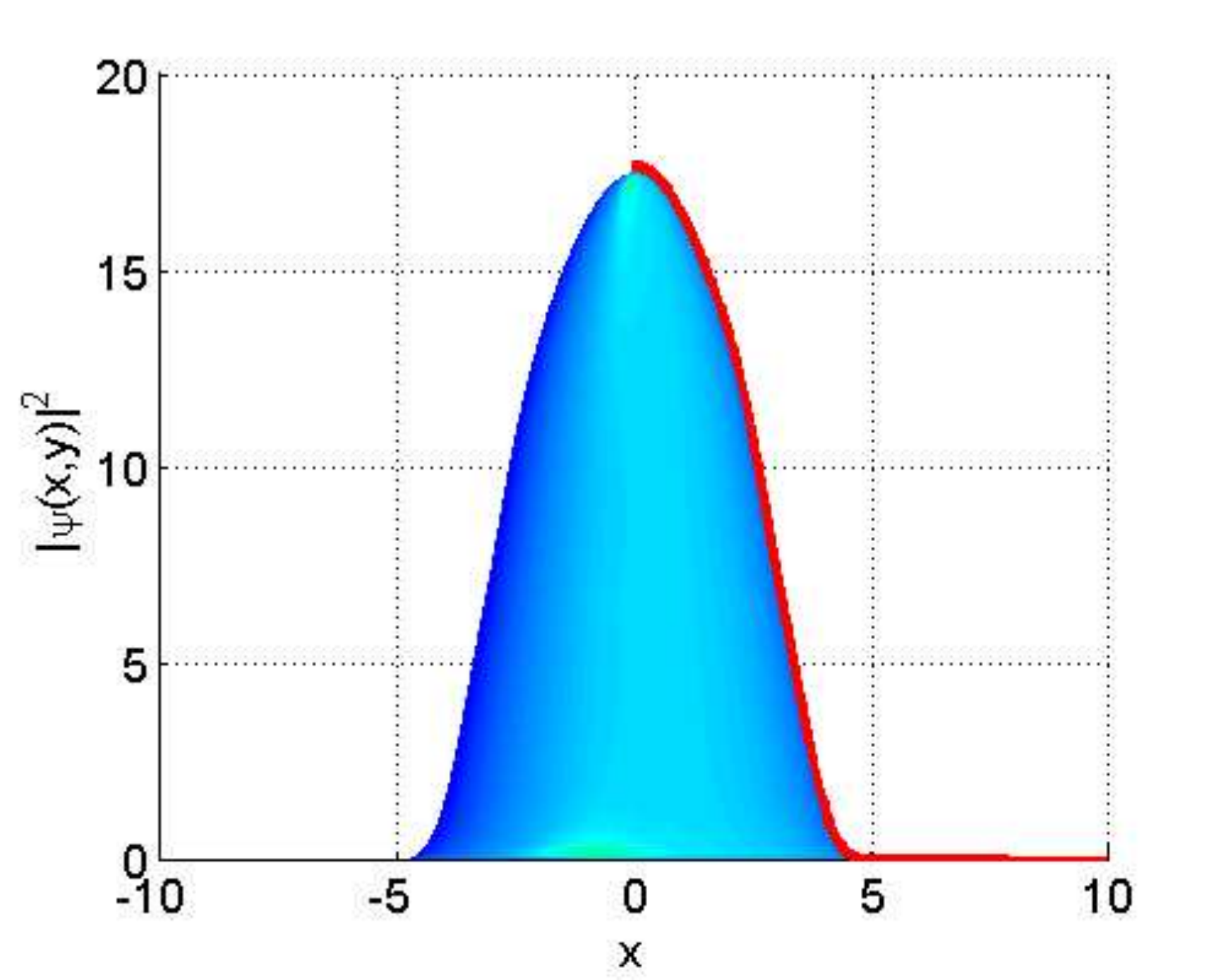} & \includegraphics[scale=0.3]{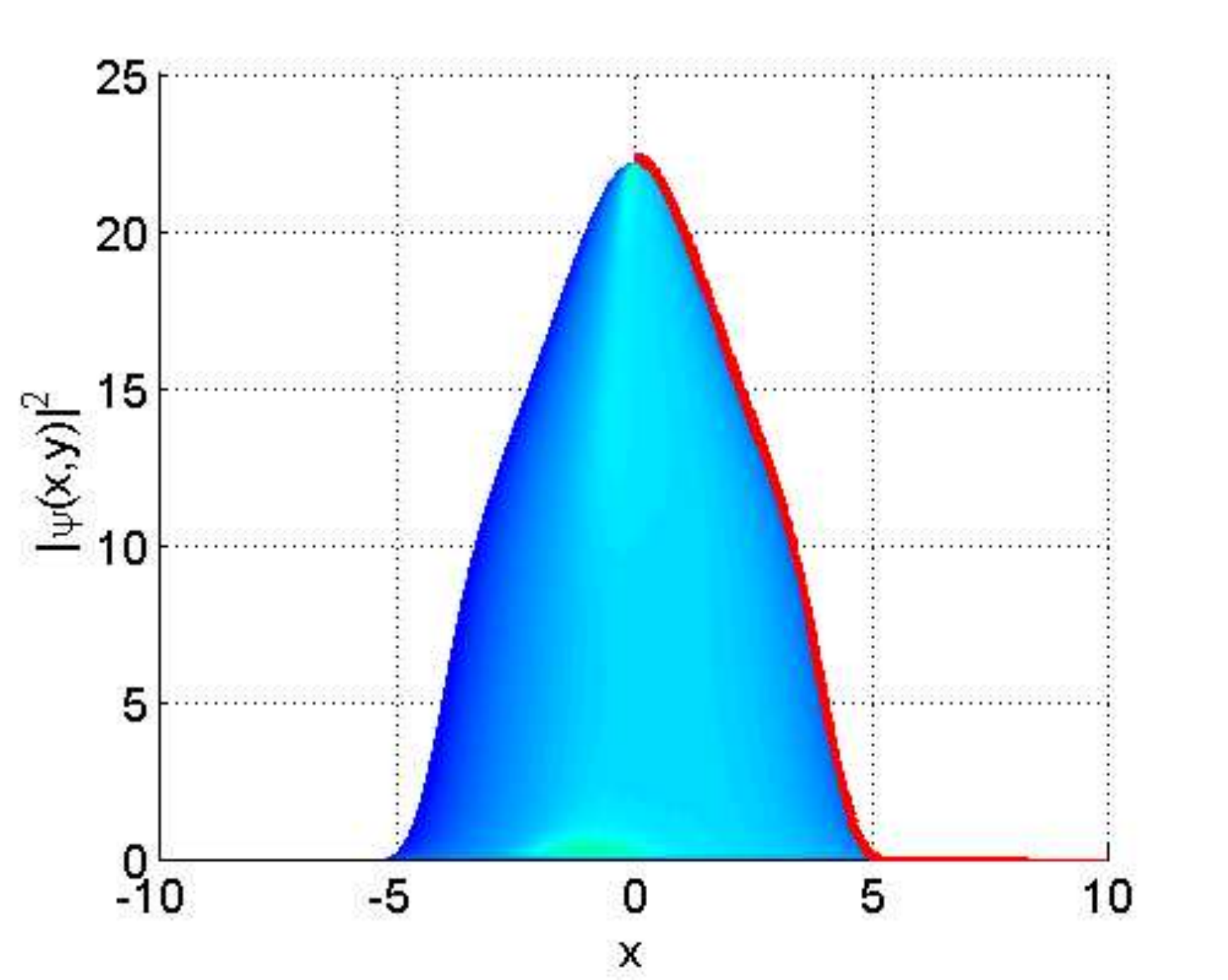}\tabularnewline
\multicolumn{3}{c}{}\tabularnewline
\end{tabular}
\par\end{centering}

\caption{Simulation results (density distributions, blue) and comparison with
the collocation method (red line) for $R=2$ (left), $R=3$ (center),
and $R=4$ (right). \label{fig:Sim_R2}}
 
\end{figure}

\begin{figure}[H]
\begin{centering}
\begin{tabular}{ccc}
\includegraphics[scale=0.35]{R2_current} & \includegraphics[scale=0.35]{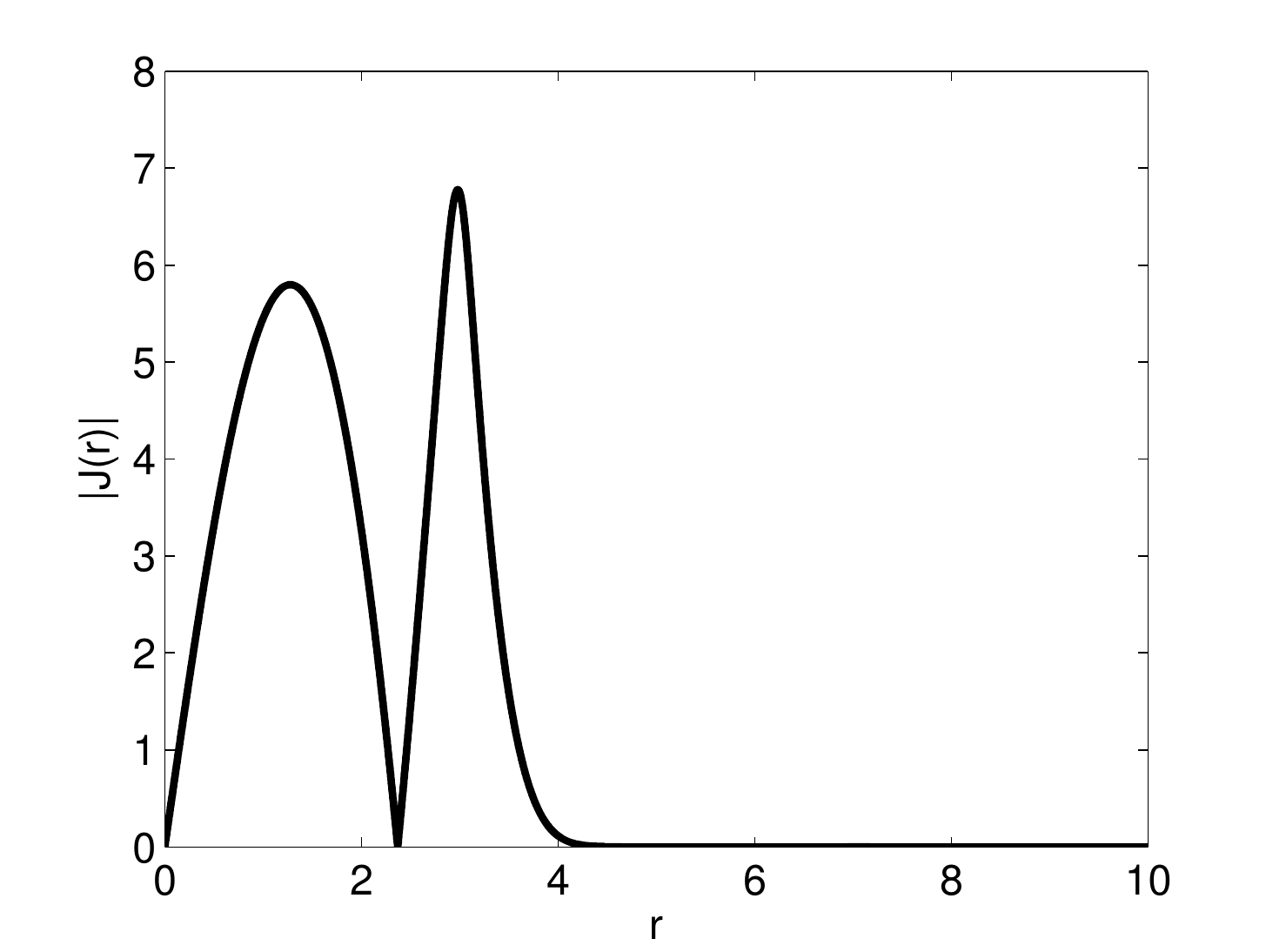} & \includegraphics[scale=0.35]{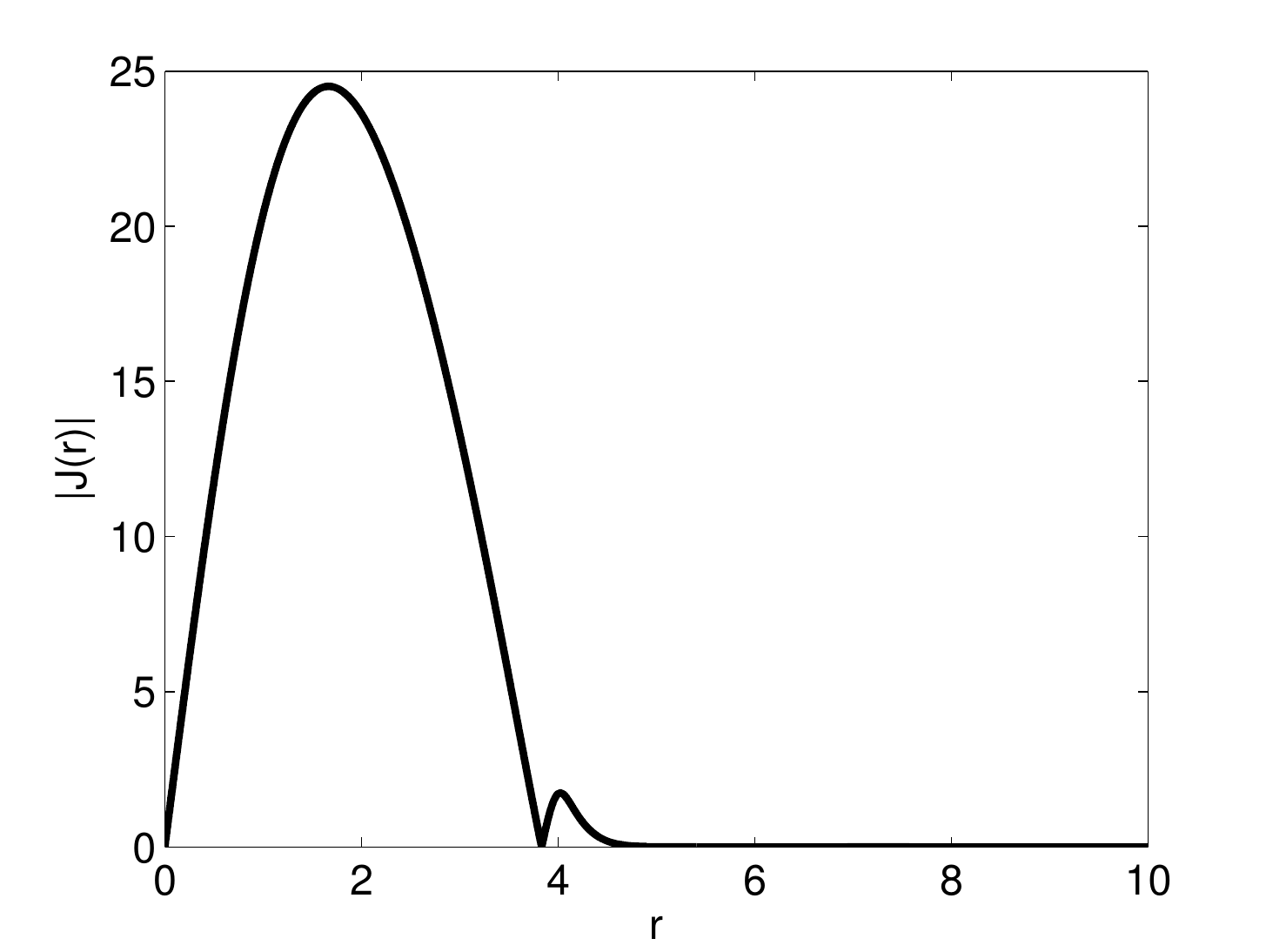}\tabularnewline
\multicolumn{3}{c}{}\tabularnewline
\end{tabular}
\par\end{centering}

\caption{Magnitude of the current $J$ corresponding to the solutions for $R=2$
(left), $R=3$ (center), and $R=4$ (right).\label{fig:currents}}
 
\end{figure}

For $R=5$, the solution becomes unstable, as predicted in Section
3. Fig. \ref{fig:Sim_R5a} shows the corresponding transient that,
after breaking the symmetry, leads to the emergence of vortex lattices.
These vortex lattices remain rotating at a constant angular velocity,
becoming the stable solution of the system. (This simulation is also
presented in the supplementary videos Sim1.avi and Sim2.avi.) Furthermore,
Figs. (\ref{fig:Vor_Core}) to (\ref{fig:vel1D}) show the characteristics
of these quantum vortices: the density profile drops to zero at the
center of the vortex core, as shown in Fig. \ref{fig:Vor_Core}; the
phase difference around the vortex is a multiple of $2\pi$, as seen
in Fig. \ref{fig:Phase} (left), where there is a change in the phase
from $0$ to $2\pi$ in all the vortices (indicated with blue circles);
the condensate circulates around the vortex, as depicted in Fig. \ref{fig:Vel_grad},
where the current $J$ is plotted for a central vortex (left) and
a satellite vortex (right). Finally, Fig. (\ref{fig:vel1D}) shows
the plot of $\left|\nabla\theta\right|$, the magnitude of the gradient
of the phase, which is proportional to the velocity of the condensate.
Notice that $\left|\nabla\theta\right|\approx1/r$ around the core
of the vortices. Finally, Fig. (\ref{fig:Sim_diff_rads}) shows the
results for $R=6,7,8,9$.

\begin{figure}[H]
\begin{centering}
\begin{tabular}{cc}
\includegraphics[width=77mm,height=70mm]{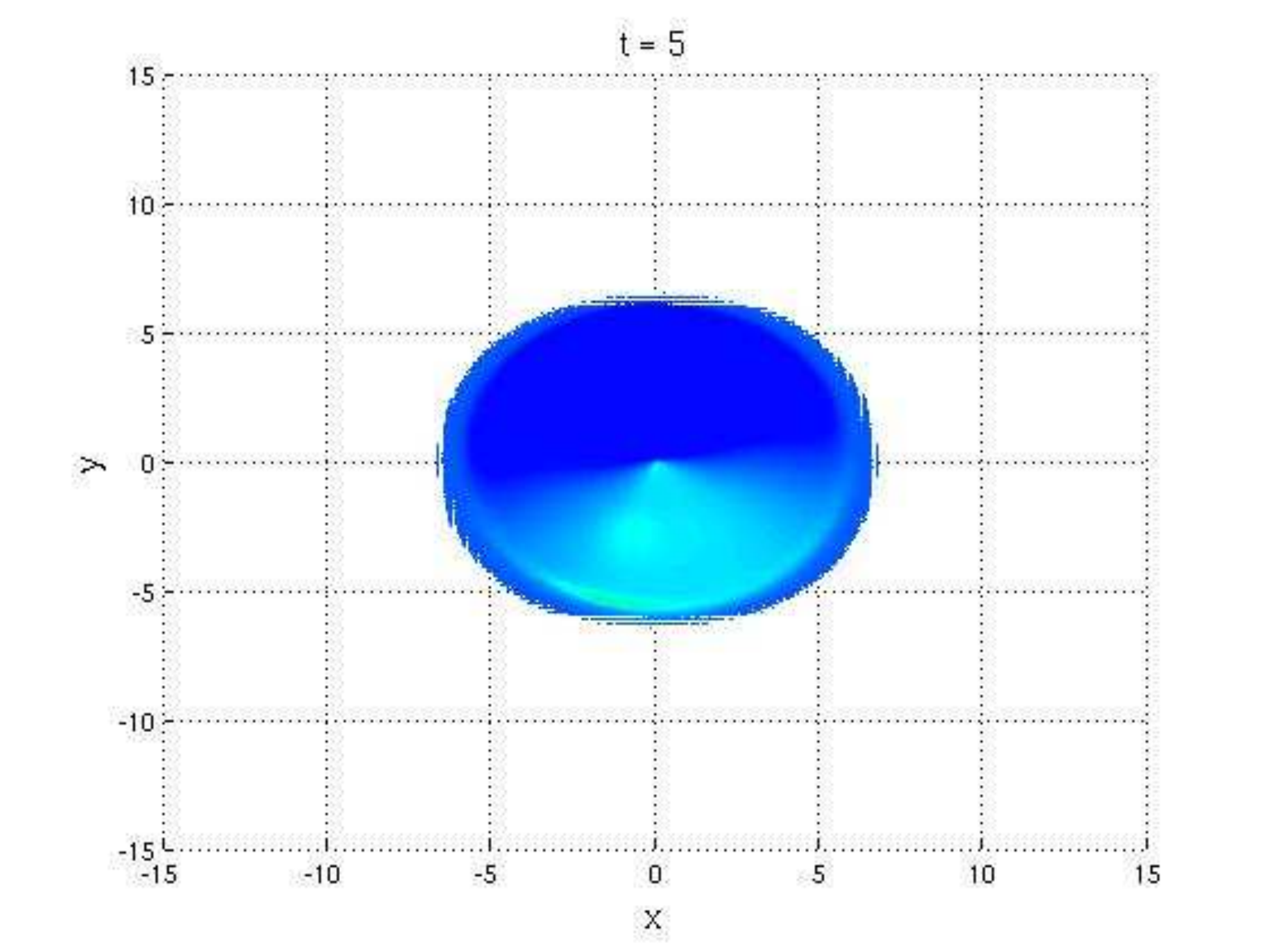} & \includegraphics[width=77mm,height=70mm]{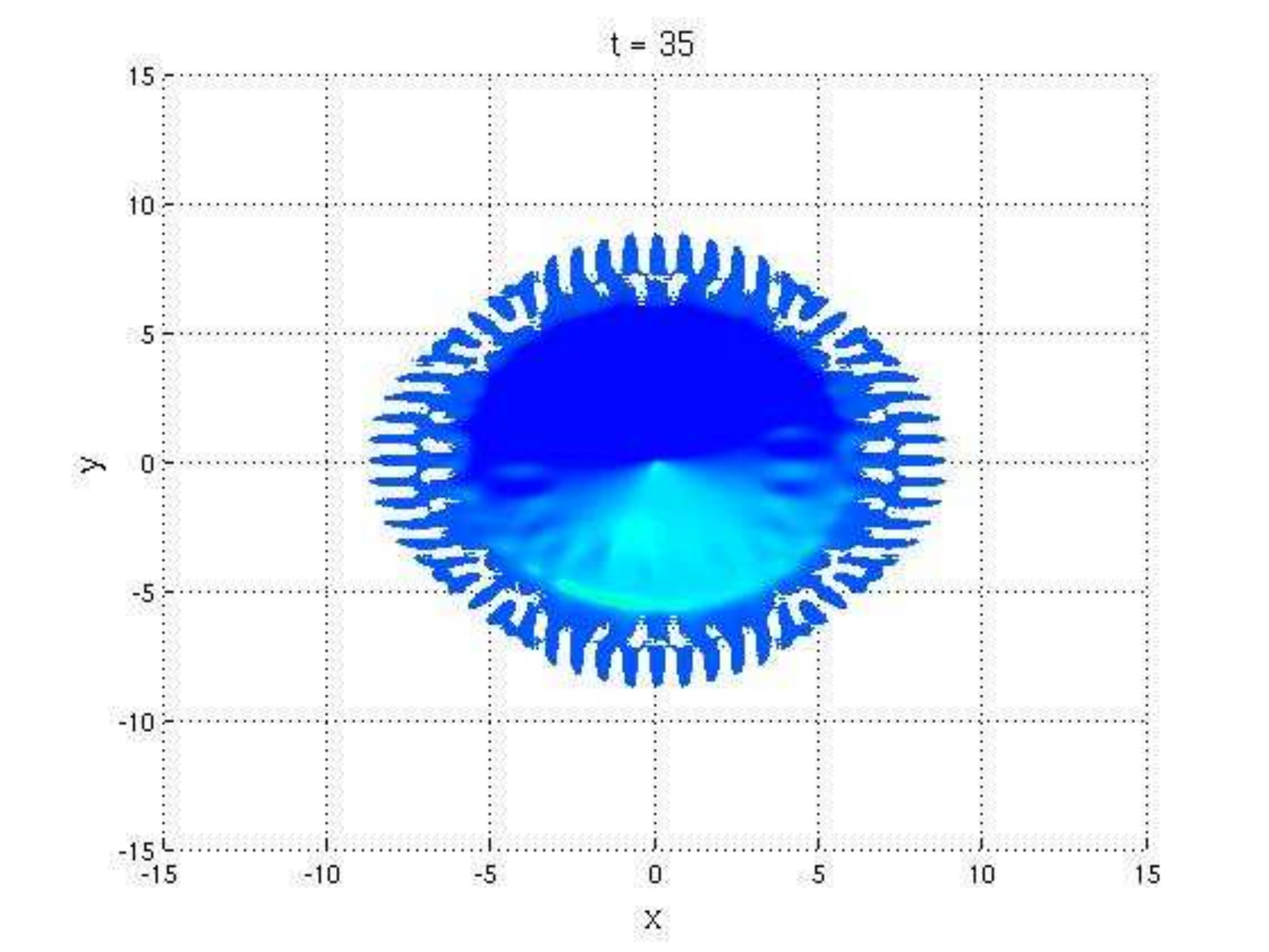}\tabularnewline
\end{tabular}
\par\end{centering}

\caption{Simulation results for $R=5$ (see also the supplementary videos Sim1.avi
and Sim2.avi). The plots show the density distribution at different
times (continued in Fig. (\ref{fig:Sim_R5b})). The white areas represent
regions where the density becomes zero. At the beginning, the condensate
is confined in a radius approximately equal to 5. After breaking of
symmetry, the condensate cloud experiences an expansion, and the remnants
of this expansion can be seen at $R\gtrsim5$. Notice that the white
dots inside the region $R\lesssim5$ at $t=70$ correspond to the
vortices that will eventually form the stable lattice. Some of these
vortices are expelled out of the structure during the transient evolution.
\label{fig:Sim_R5a}}

\end{figure}

\begin{figure}[H]
\begin{centering}
\begin{tabular}{cc}
\includegraphics[width=77mm,height=70mm]{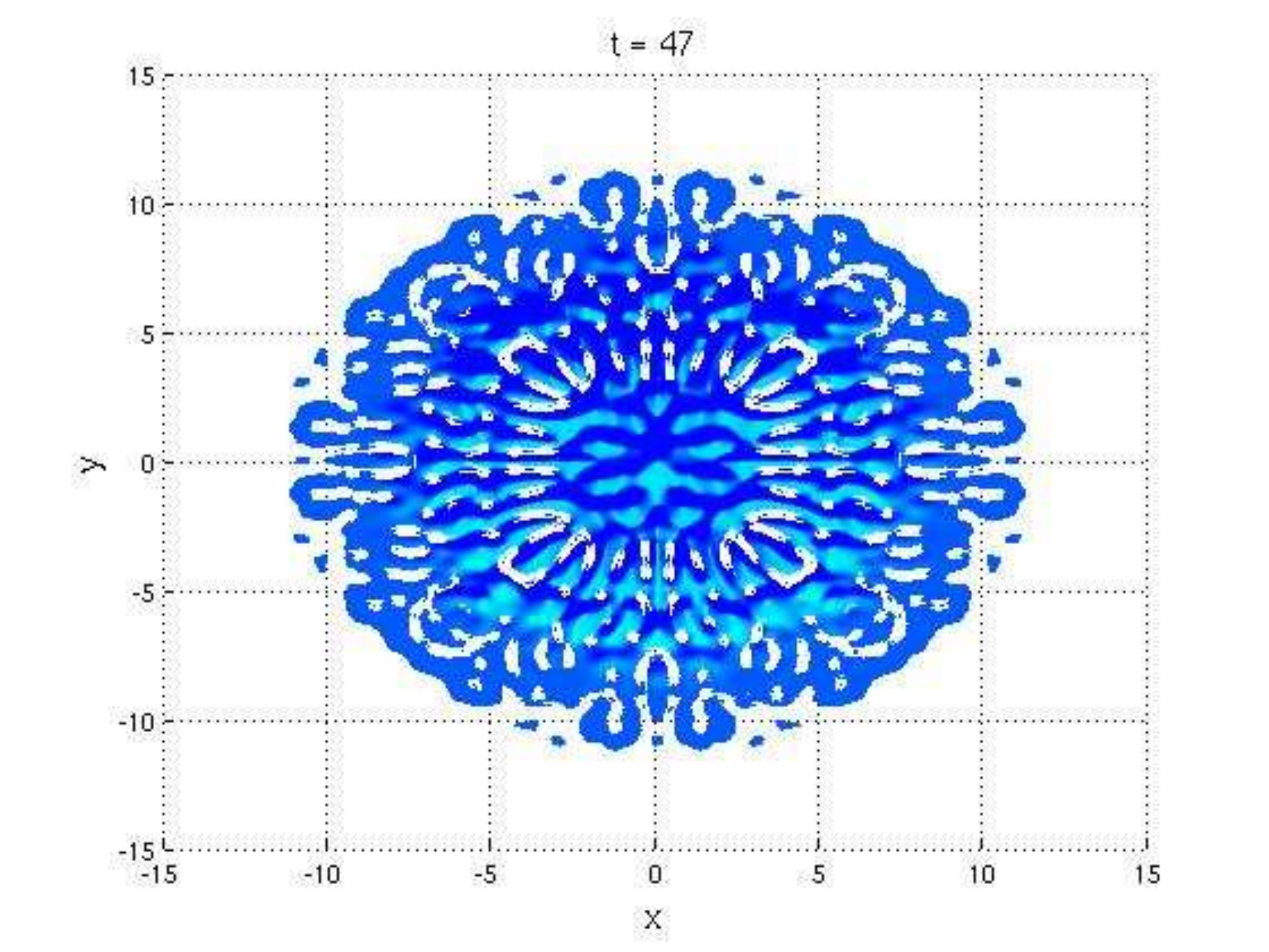} & \includegraphics[width=77mm,height=70mm]{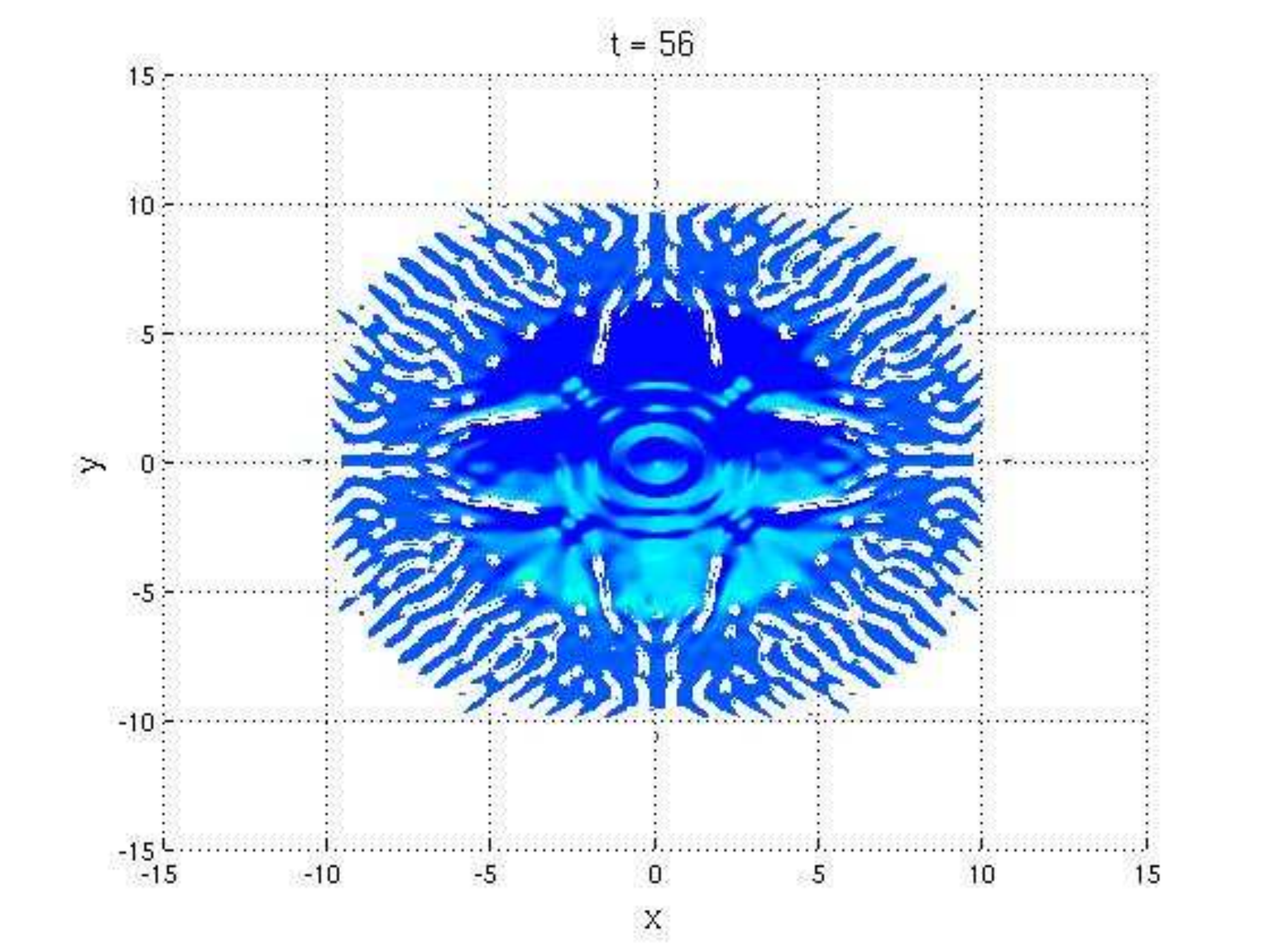}\tabularnewline
\includegraphics[width=77mm,height=70mm]{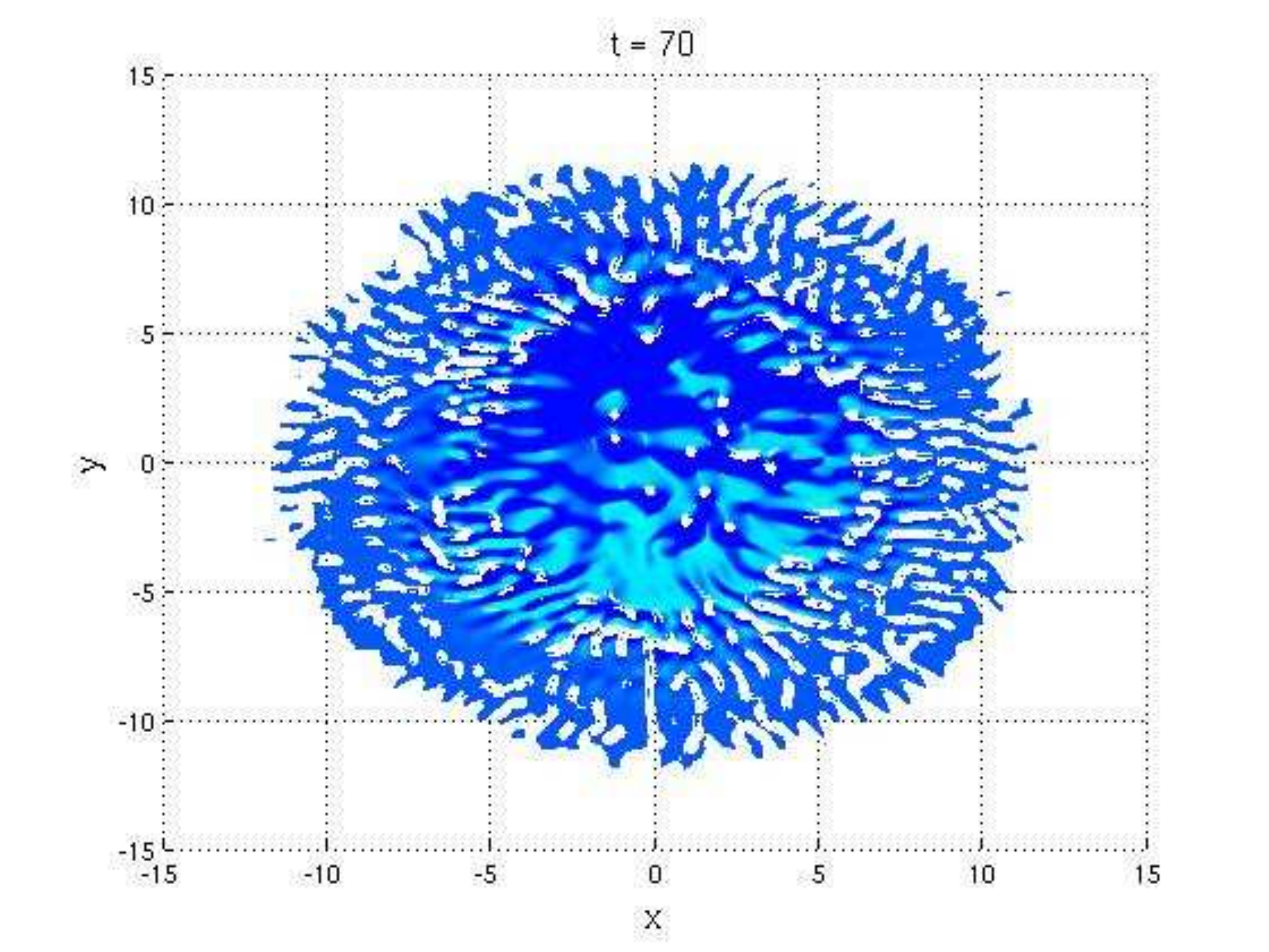} & \includegraphics[width=77mm,height=70mm]{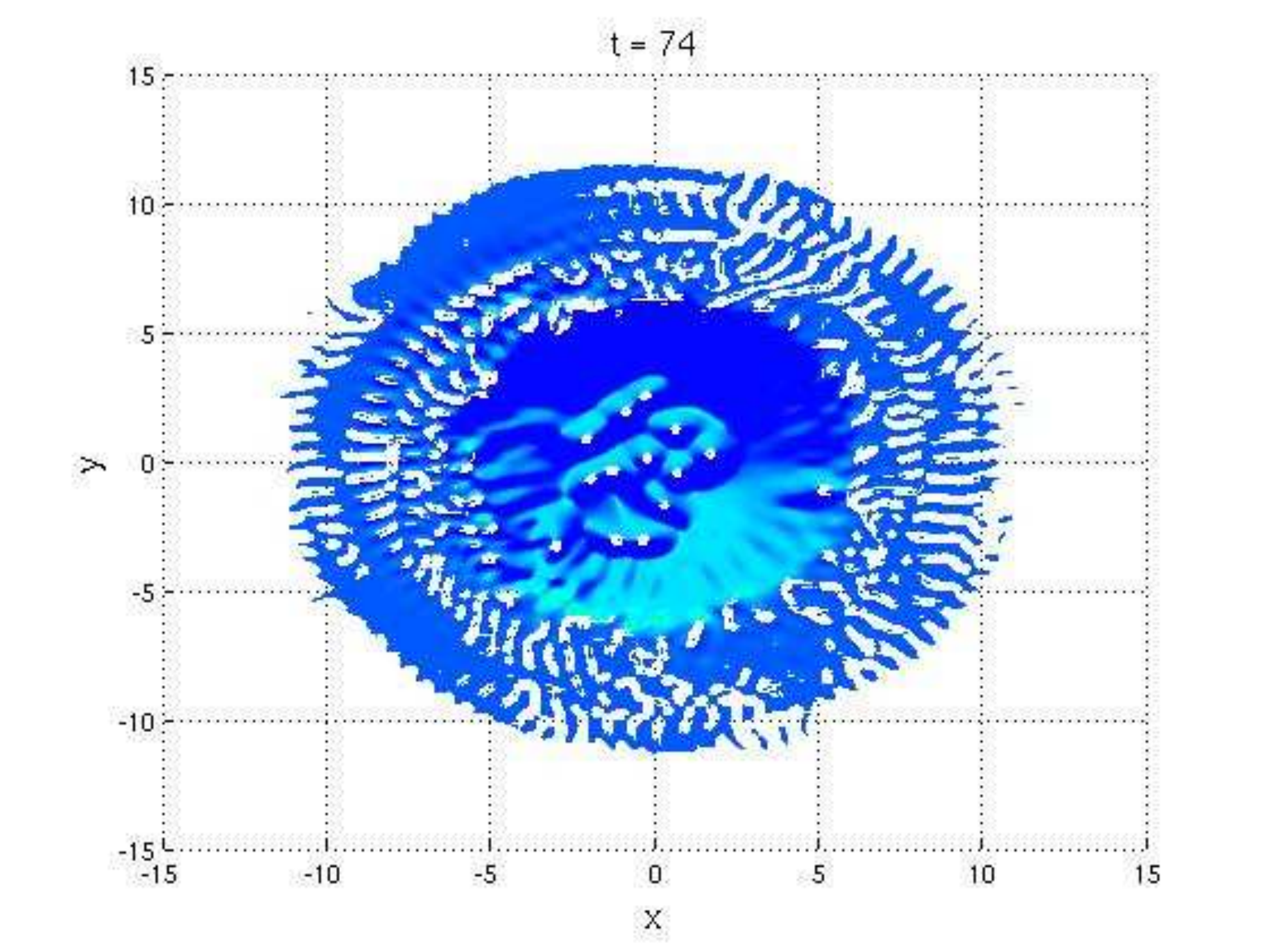}\tabularnewline
\includegraphics[width=77mm,height=70mm]{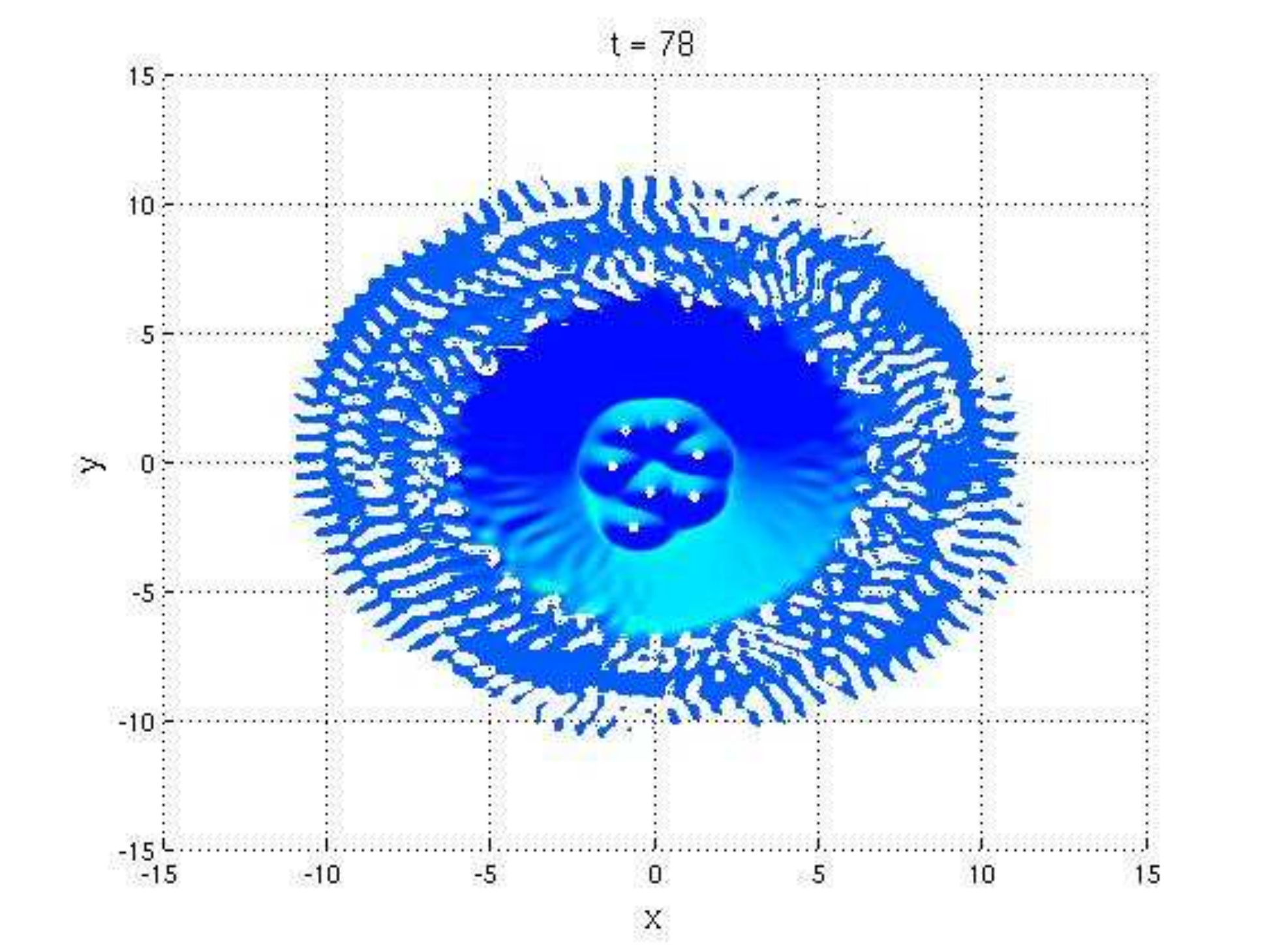} & \includegraphics[width=77mm,height=70mm]{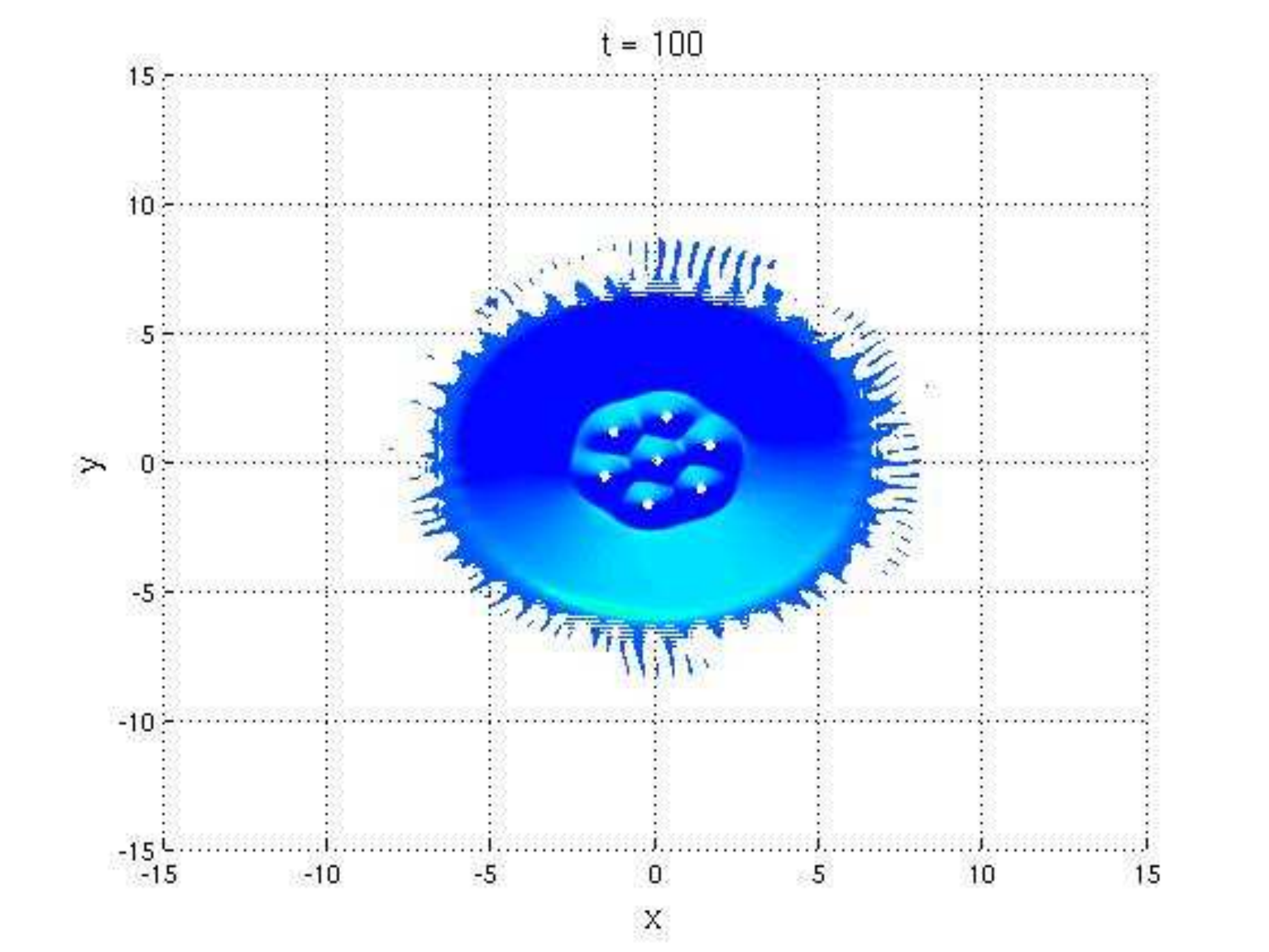}\tabularnewline
\end{tabular}
\par\end{centering}

\caption{Continuation of Fig. (\ref{fig:Sim_R5a}).\label{fig:Sim_R5b}}

\end{figure}

\begin{figure}[H]
\begin{centering}
\includegraphics[scale=0.25]{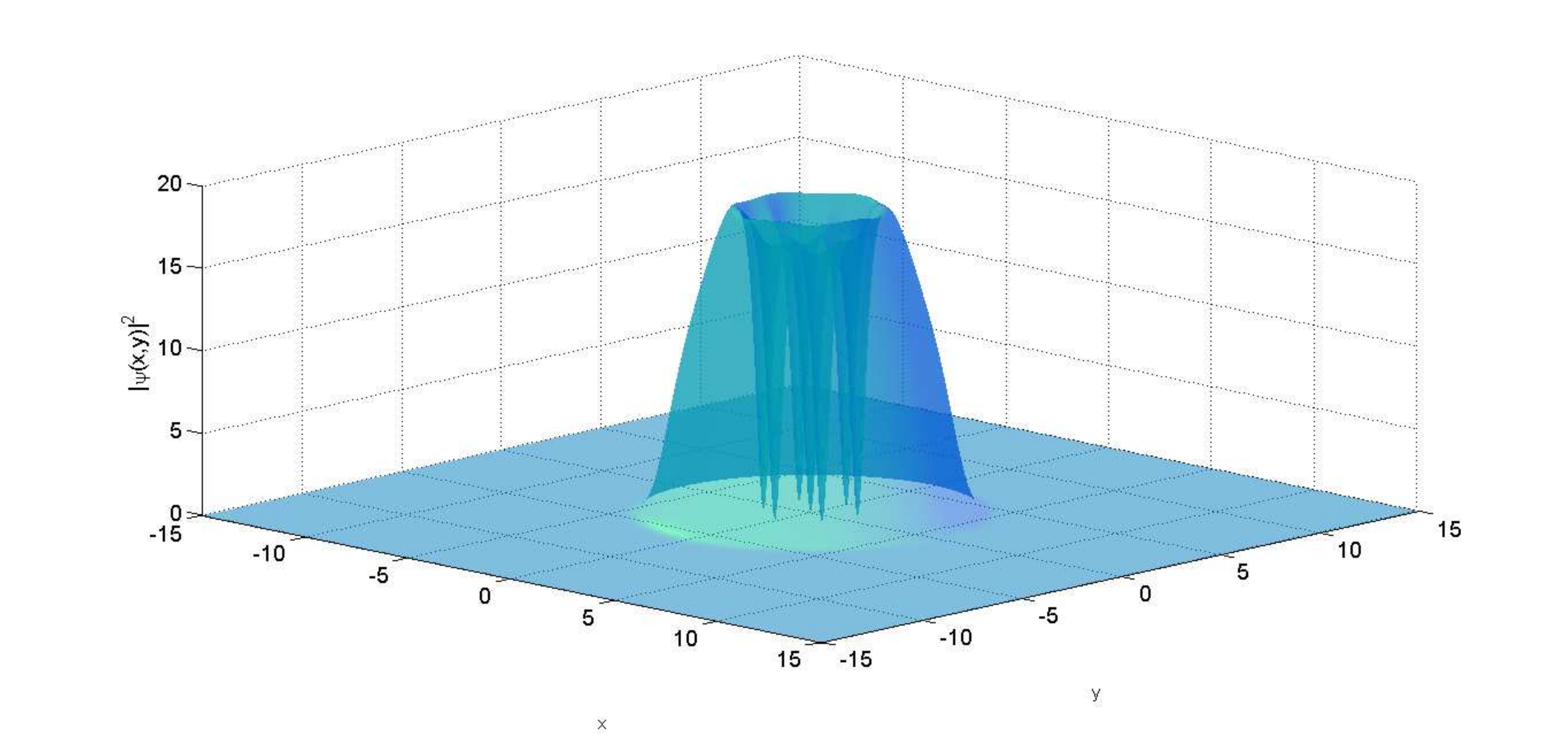}
\par\end{centering}

\caption{A transparent image corresponding to the density distribution at $t=140$
($R=5$).\label{fig:Vor_Core}}
 
\end{figure}

\begin{figure}[H]
\begin{centering}
\begin{tabular}{cc}
\includegraphics[scale=0.4]{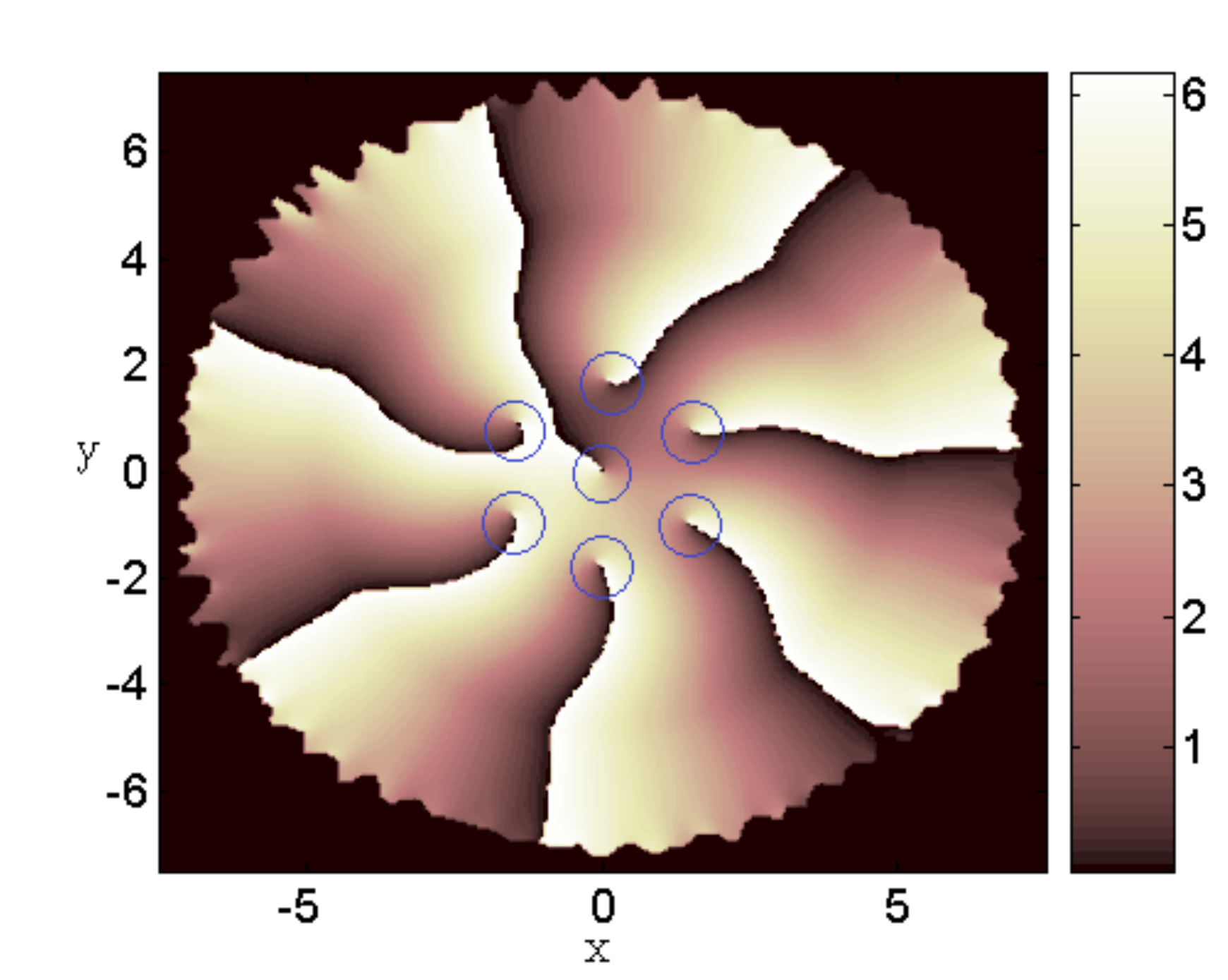} & \includegraphics[scale=0.4]{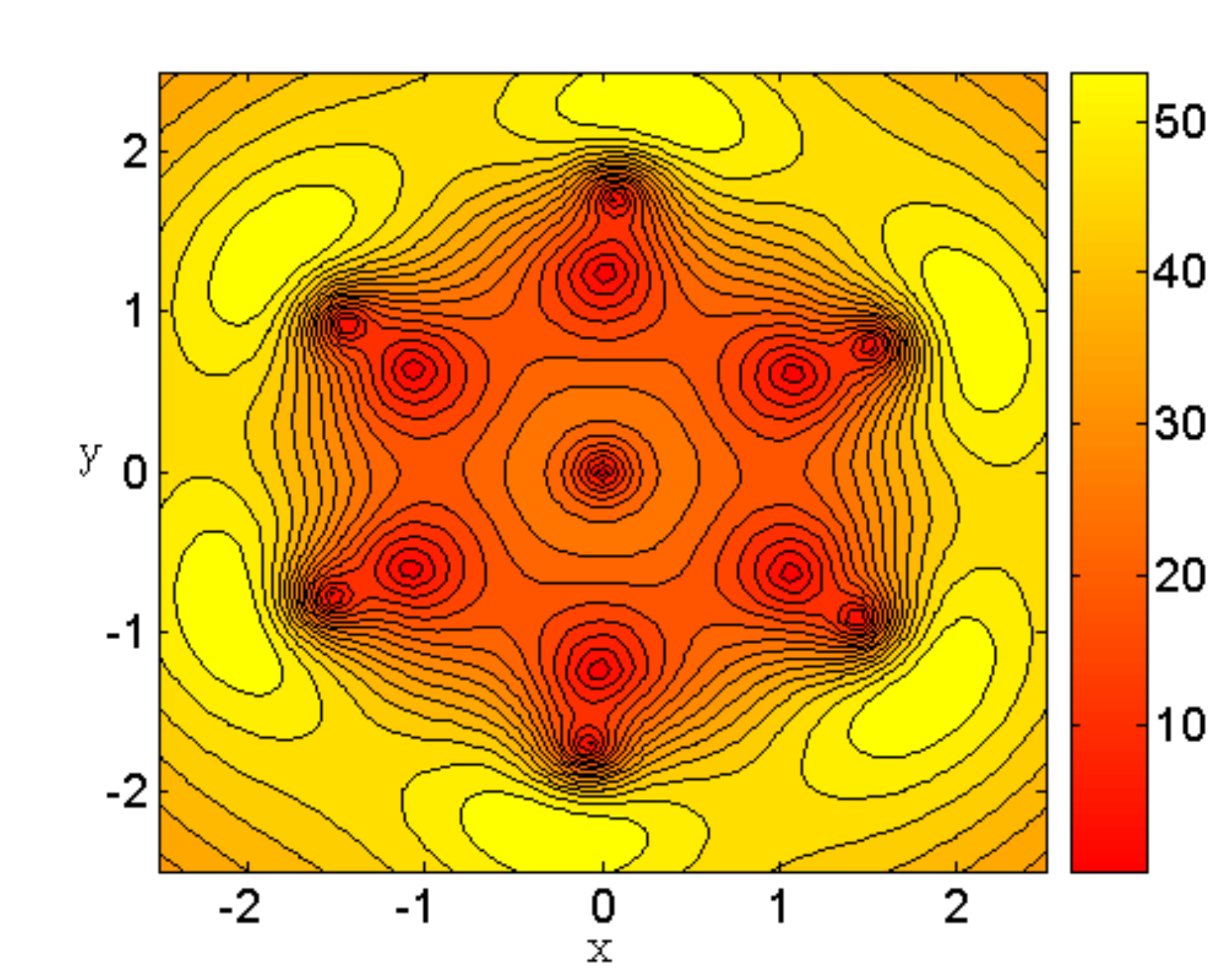}\tabularnewline
\end{tabular}
\par\end{centering}

\caption{Left: Phase $\theta(x,y)$ of the solution at $t=140$ ($R=5$), with
the blue circles indicating the core of the vortices. Right: The magnitude
of the current$\left|J\left(x,y\right)\right|$ at $t=140$. Notice
that the contour lines are closer around the core of the vortices.
\label{fig:Phase}}
 
\end{figure}

\begin{figure}[H]
\begin{centering}
\begin{tabular}{cc}
\includegraphics[scale=0.4]{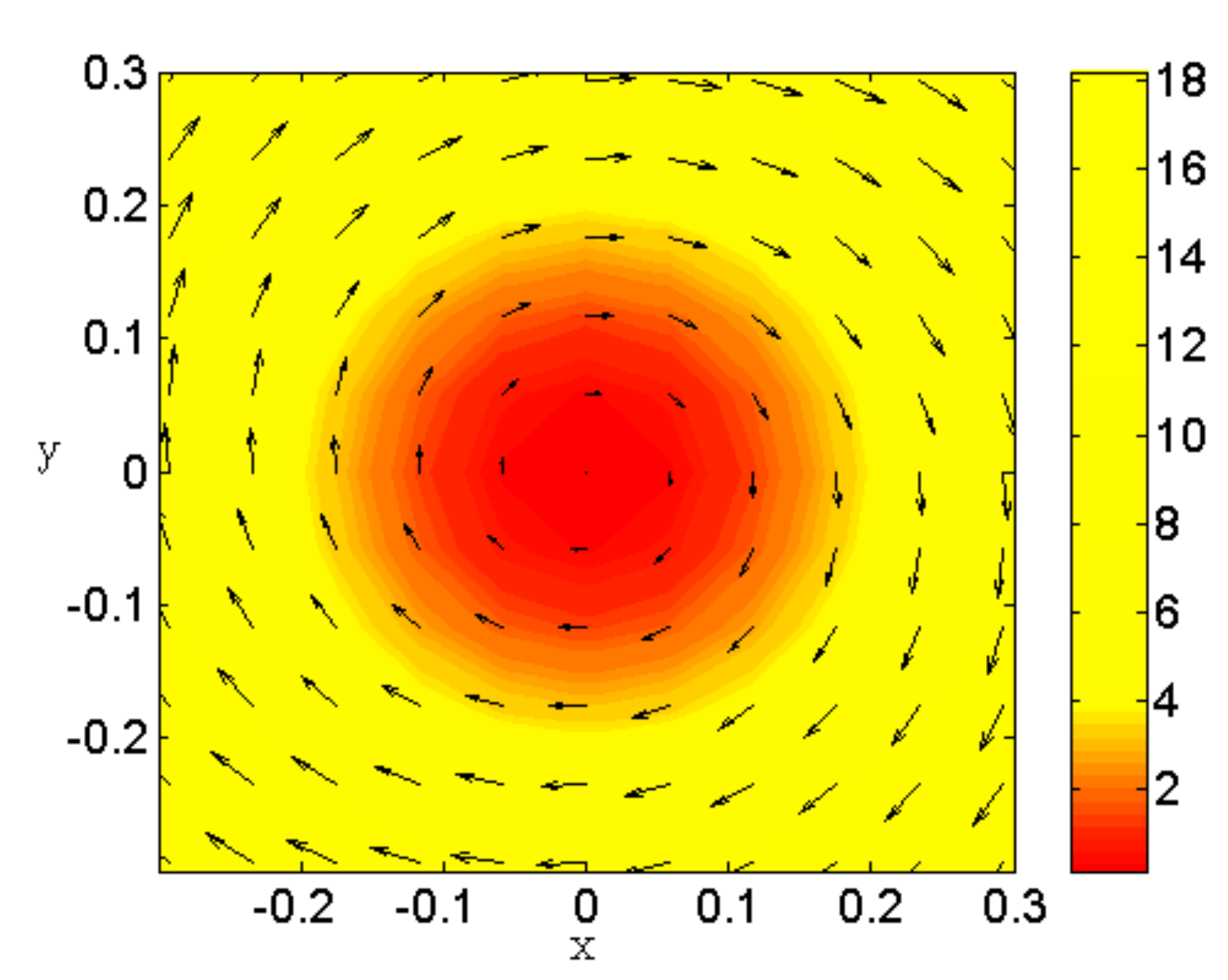} & \includegraphics[scale=0.4]{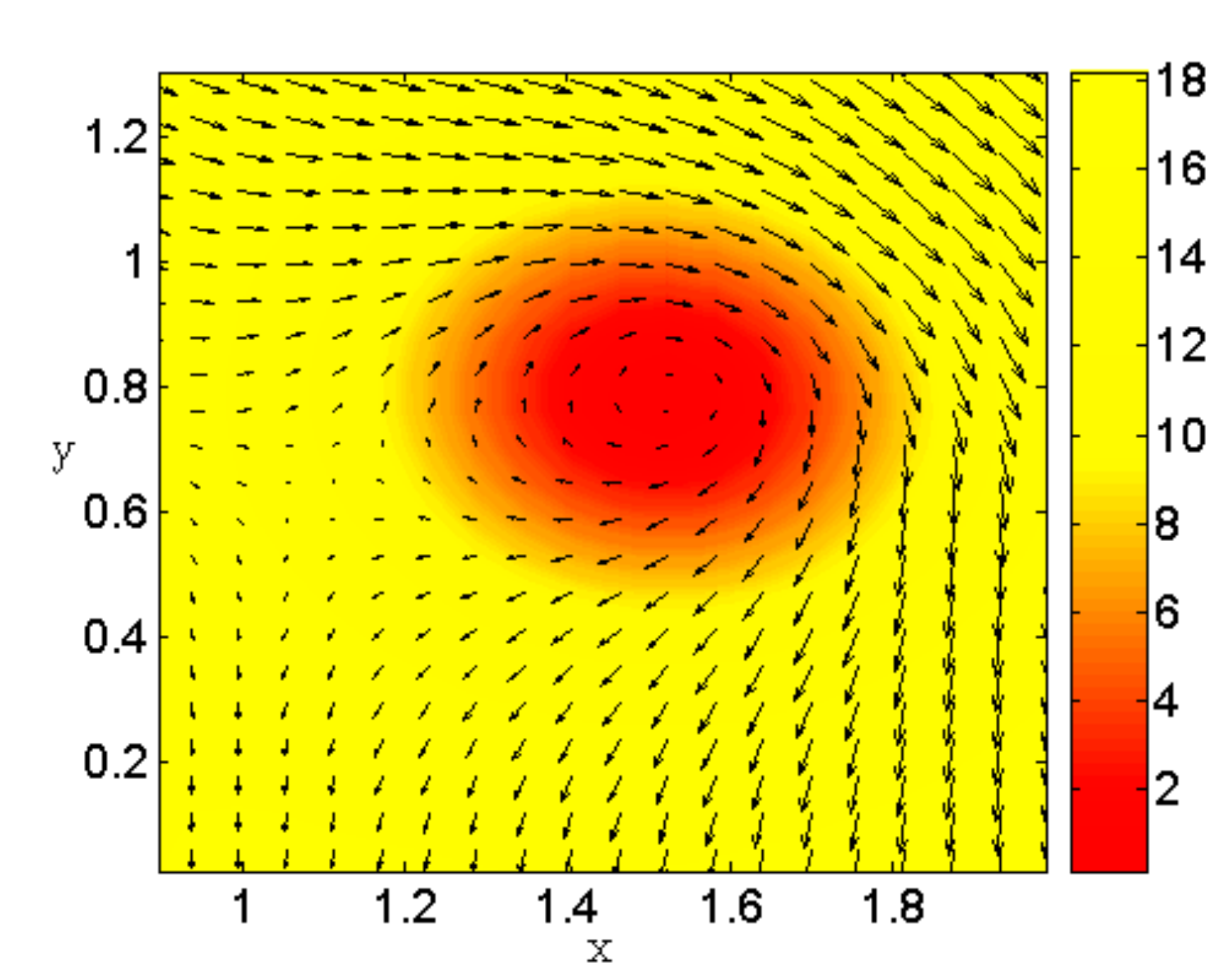}\tabularnewline
\end{tabular}
\par\end{centering}

\caption{Density distribution and arrows corresponding to the current $J$
in the central vortex (left) and one satellite vortex (right) of the
solution at $t=140$ ($R=5$).\label{fig:Vel_grad}}
 
\end{figure}

\begin{figure}[H]
\begin{centering}
\begin{tabular}{cc}
\includegraphics[scale=0.5]{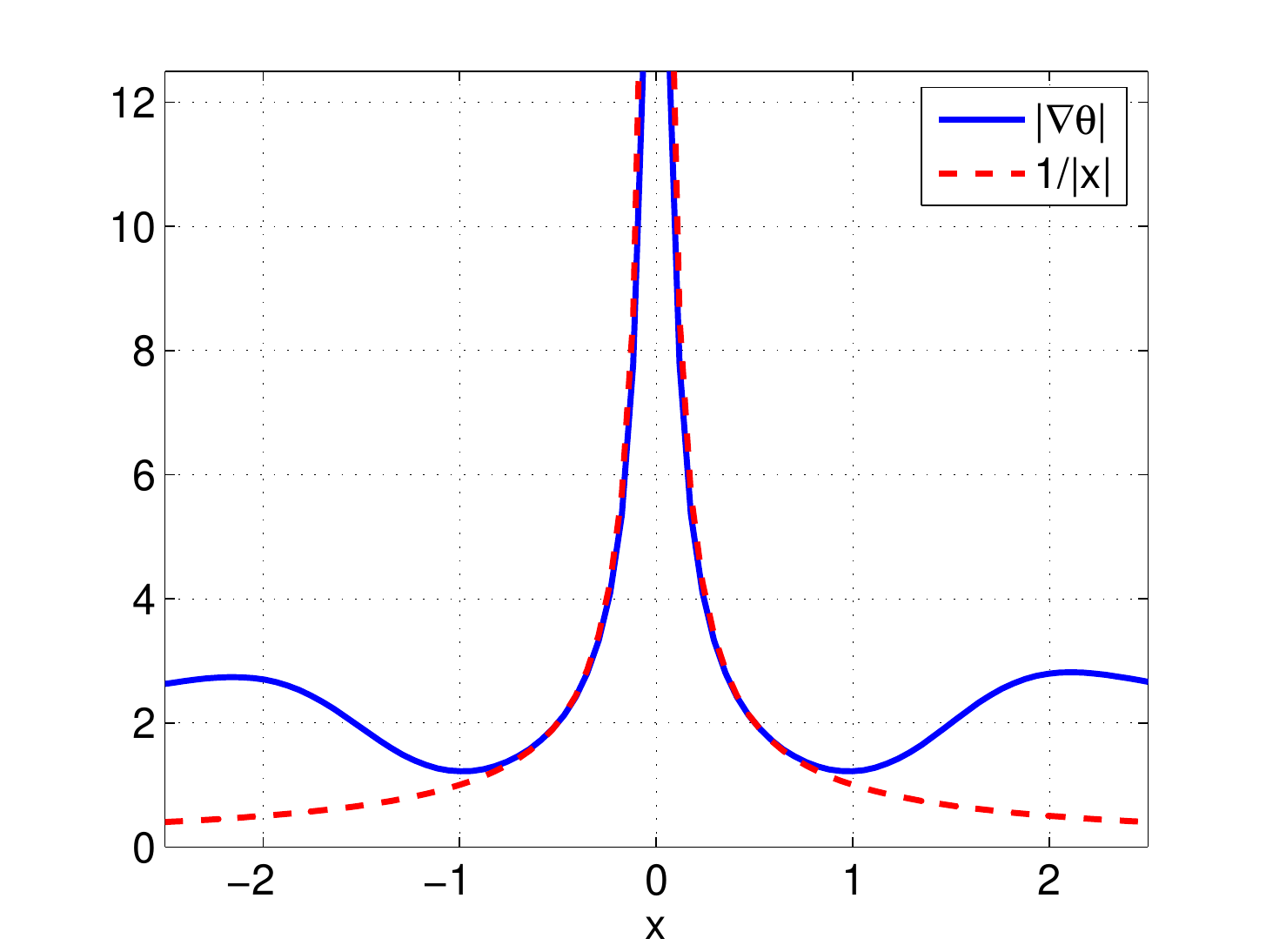} & \includegraphics[scale=0.37]{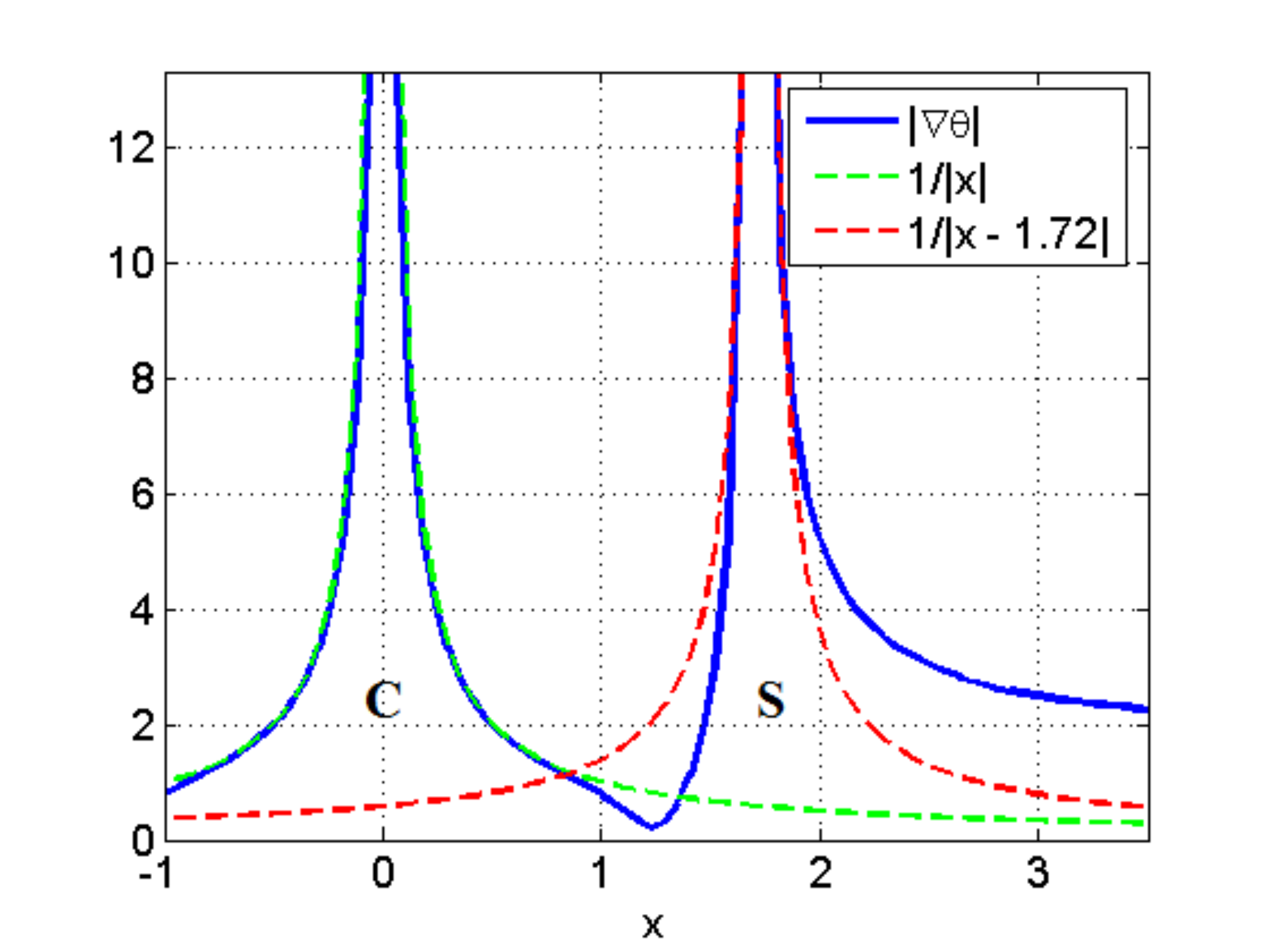}\tabularnewline
\end{tabular}
\par\end{centering}

\caption{1D plot for $\left|\nabla\theta\right|$ ($R=5$): central vortex
(left); central vortex ``C'' and satellite vortex ``S'' (right).
For both central and satellite vortices, $\left|\nabla\theta\right|\approx1/\left|x\right|=1/r$
around the cores.\label{fig:vel1D}}
 
\end{figure}

\begin{figure}[H]
\begin{centering}
\begin{tabular}{cc}
\includegraphics[width=72mm,height=65mm]{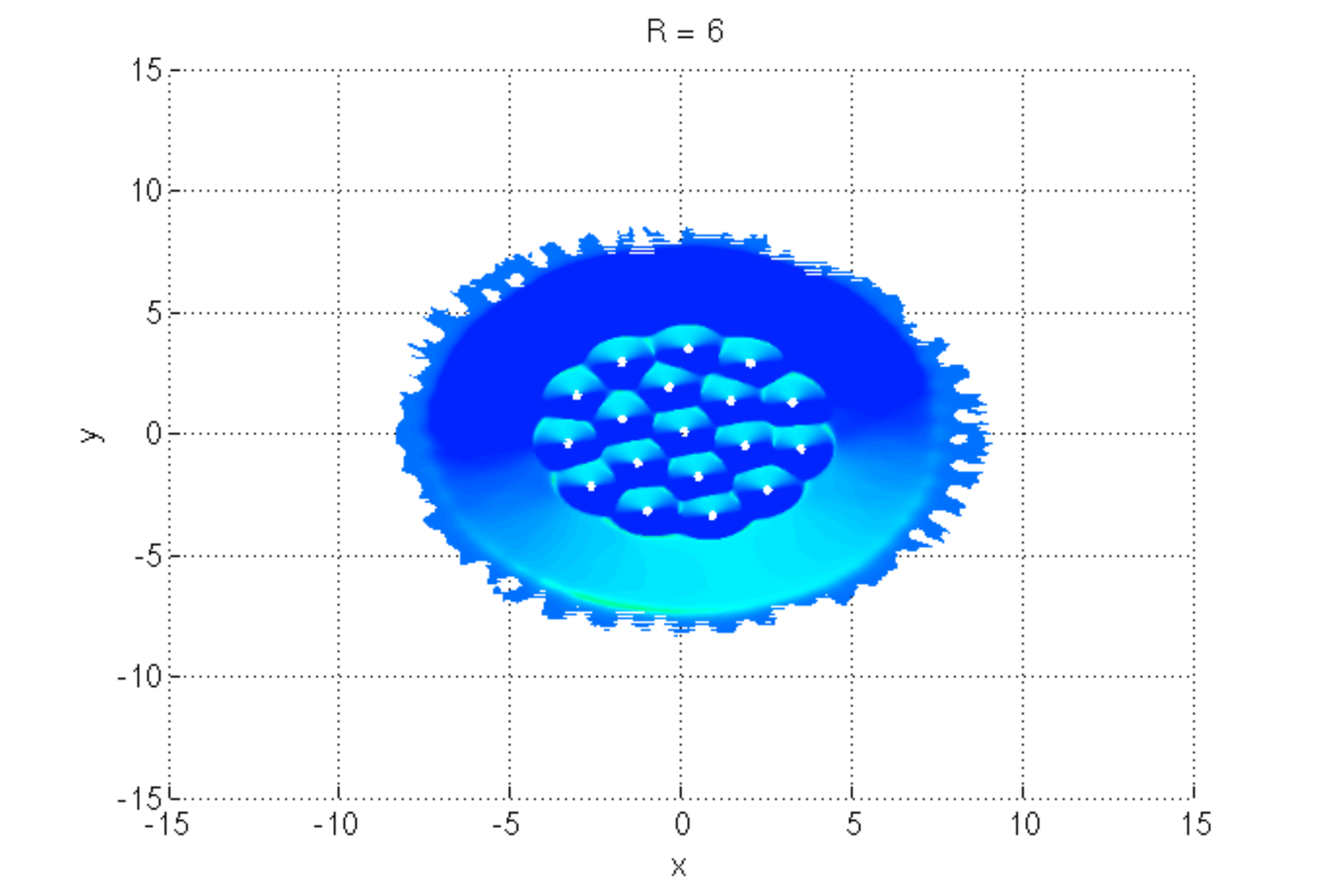} & \includegraphics[width=72mm,height=65mm]{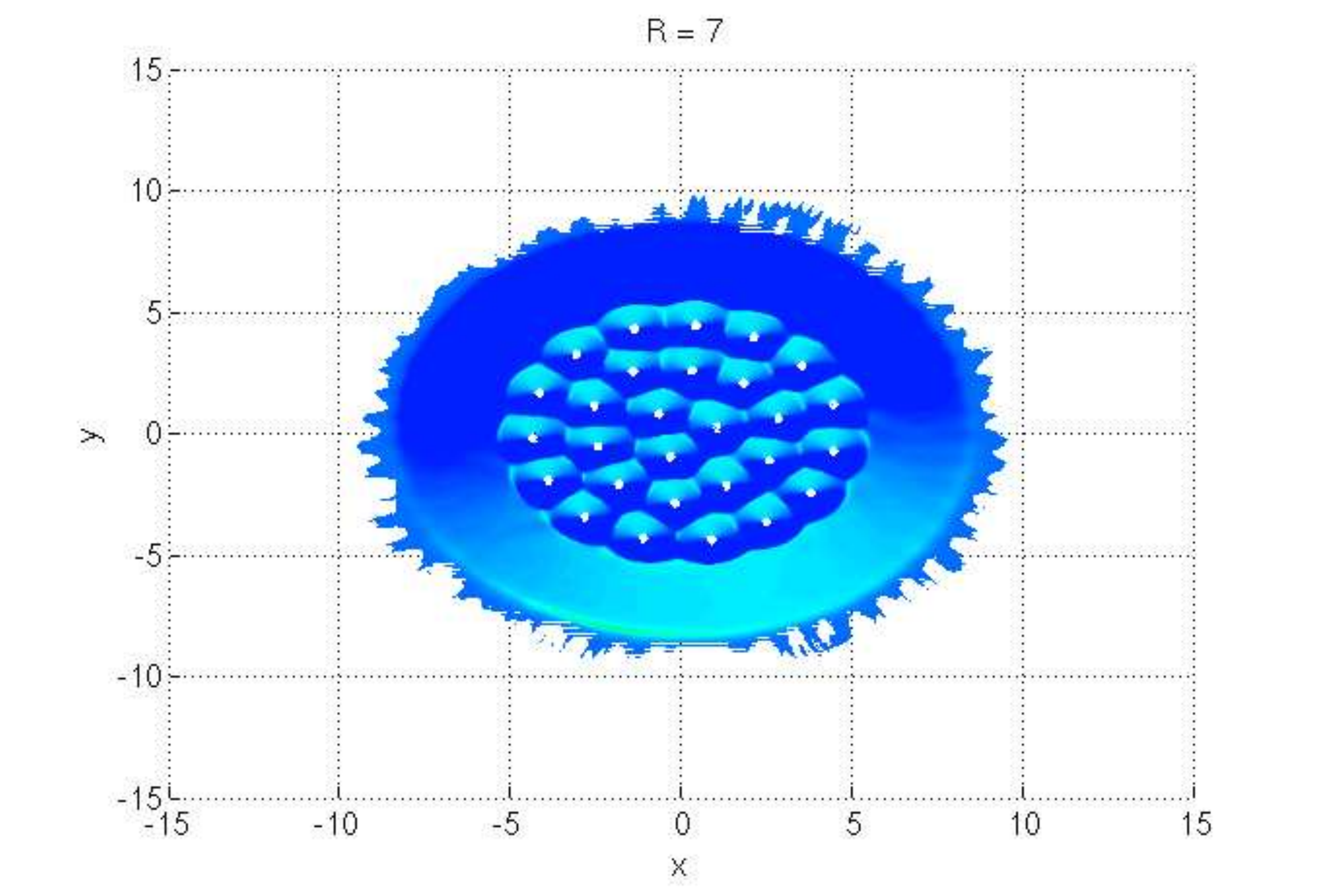}\tabularnewline
\includegraphics[width=72mm,height=65mm]{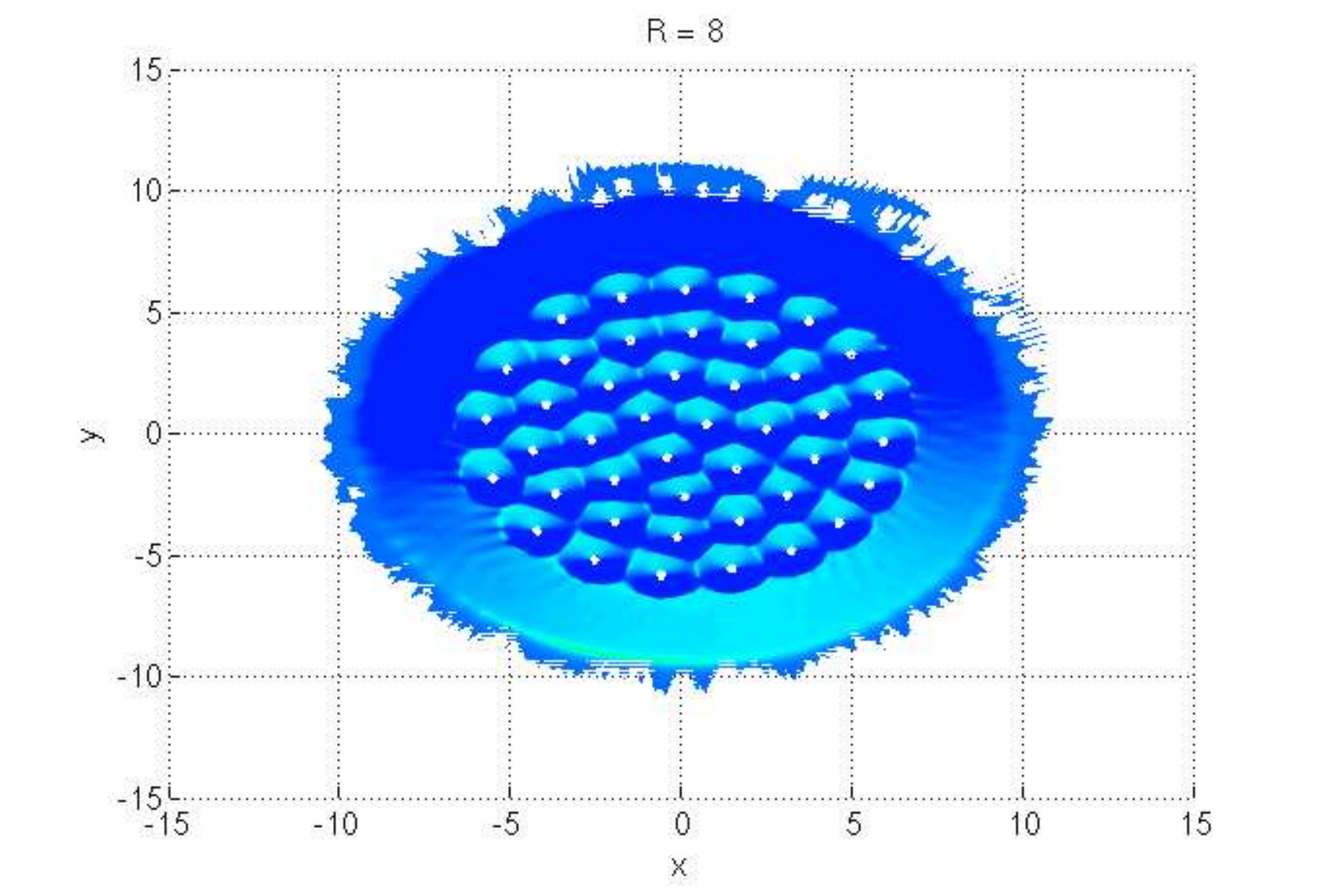} & \includegraphics[width=72mm,height=65mm]{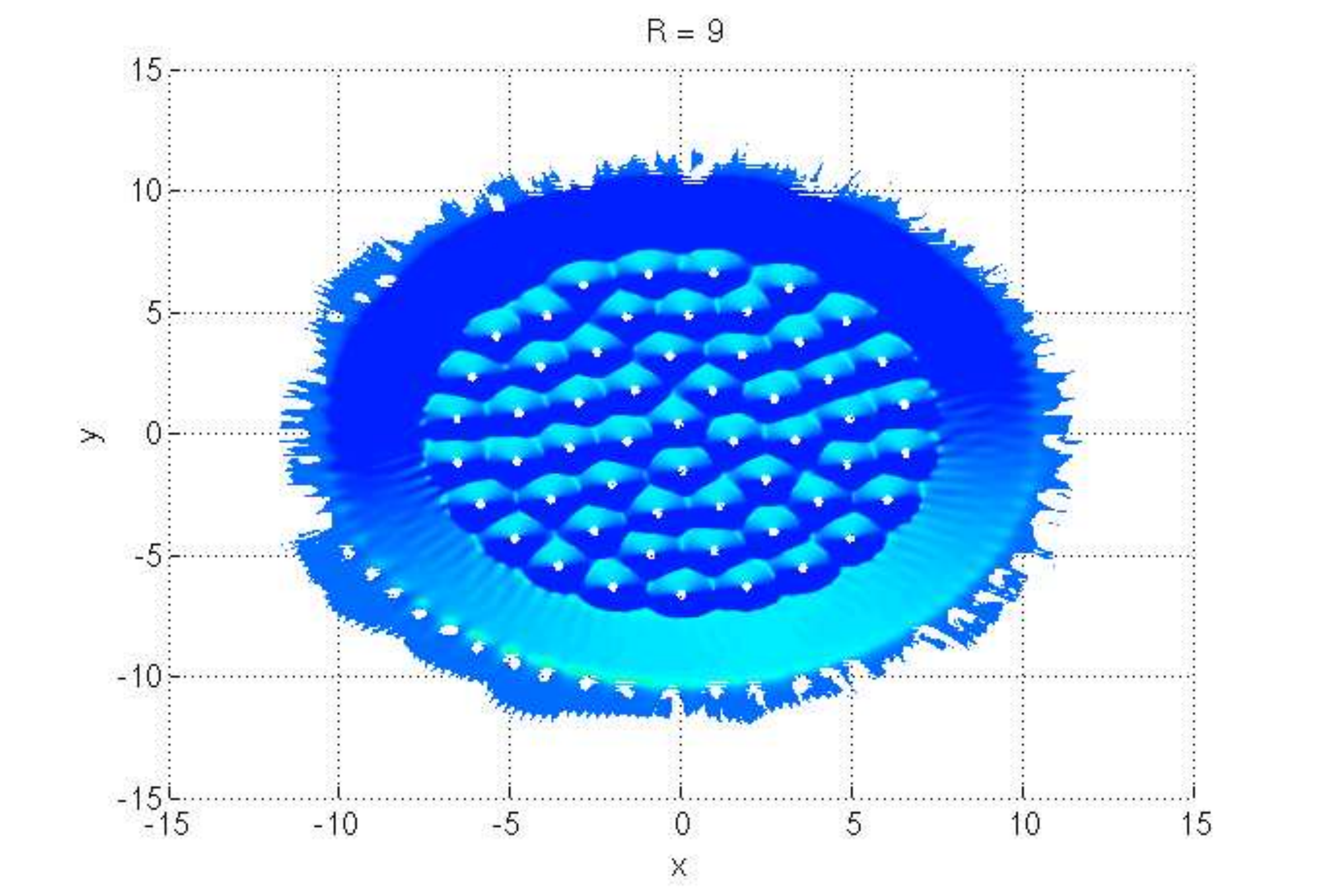}\tabularnewline
\end{tabular}
\par\end{centering}

\caption{Vortex lattices for $\alpha=4.4$, $\sigma=0.3$ and different values
of $R$.\label{fig:Sim_diff_rads}}
\end{figure}

\begin{figure}[H]
\begin{centering}
\includegraphics[scale=0.6]{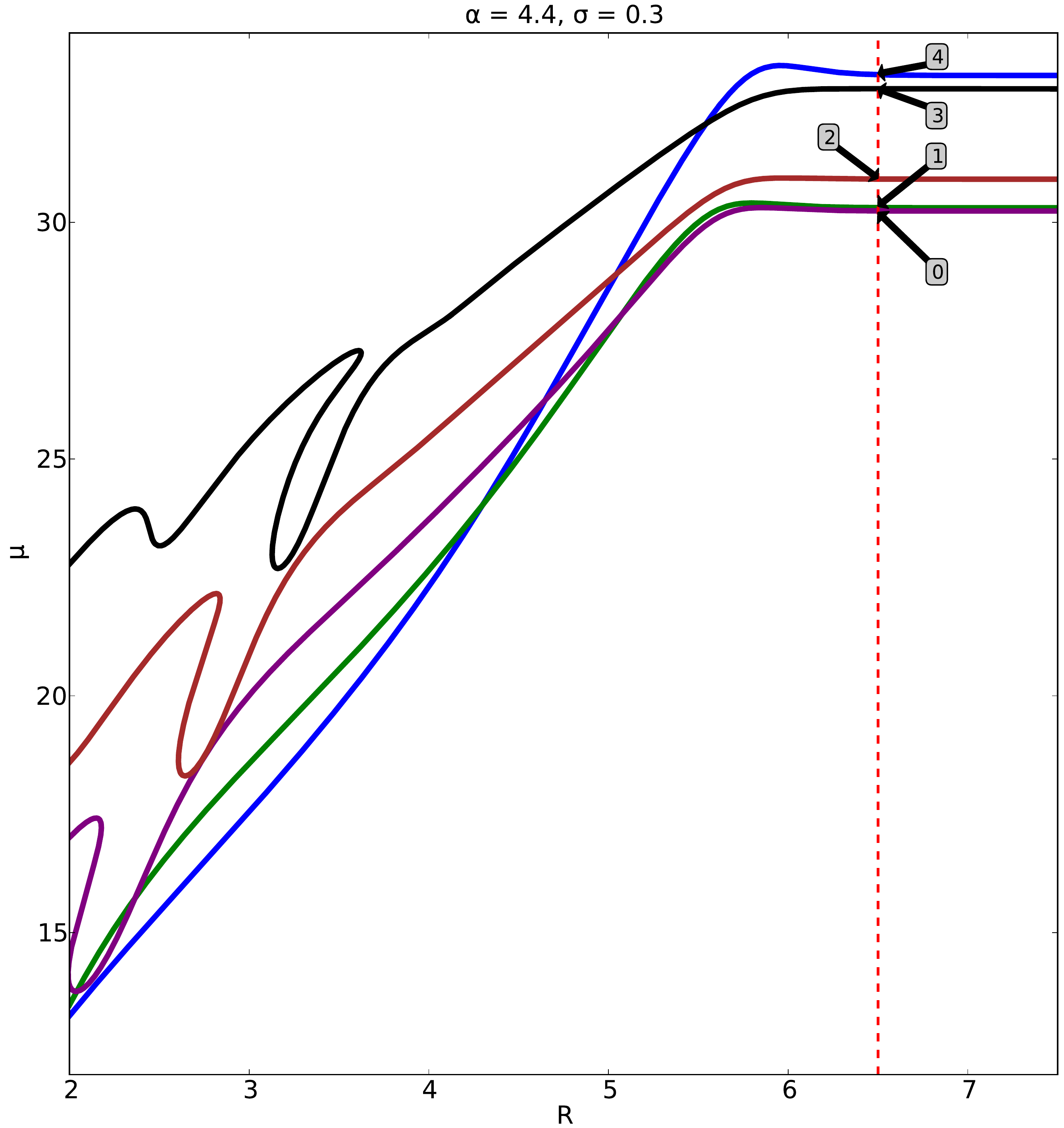}
\par\end{centering}

\caption{Continuation in $R$ starting from the solutions displayed in Figs.
\ref{fig:Ex_states} and \ref{fig:Cont_ex_states}. The density profiles
corresponding to the points labeled 0 to 4 are shown in Fig \ref{fig:Cont_R_profs}.\label{fig:Cont_R}}
\end{figure}

\subsection{Numerical continuation in $R$}

Figure (\ref{fig:Cont_R}) shows the continuation in $R$ starting
from the solutions displayed in Figs. (\ref{fig:Ex_states}) and (\ref{fig:Cont_ex_states}).
It can be seen that for all the branches, the solutions remain unchanged
for large values of $R$. Moreover, the blue branch, which corresponds
to the branch studied in the linear stability analysis performed in
Section 4, has the smallest $\mu$ up to $R\simeq4.4$. This is also
the point at which the branch becomes linearly unstable, as shown
in Fig. (\ref{fig:stab_curve}). On the other hand, Fig. (\ref{fig:Cont_R_profs})
shows the density profiles corresponding to the solutions labeled
0 to 4 in Fig. (\ref{fig:Cont_R}). Notice that although the continuation
process started in the multi-bump profiles shown in Figs. (\ref{fig:Ex_states})
and (\ref{fig:Cont_ex_states}), these solutions do not have this
multi-bump characteristic. Also, the blue branch is not the one with
the smallest $\mu$ for large values of $R$. In any case, for large
values of $R$, none of these radially symmetric solutions is linearly
stable. It is interesting to observe that even though the radially
symmetric solutions remain unchanged for large values of $R$, the
stable vortex lattice solutions change as $R$ increases, as seen
in Fig. (\ref{fig:Sim_diff_rads}). 

\begin{figure}[H]
\begin{centering}
\begin{tabular}{cc}
\includegraphics[scale=0.4]{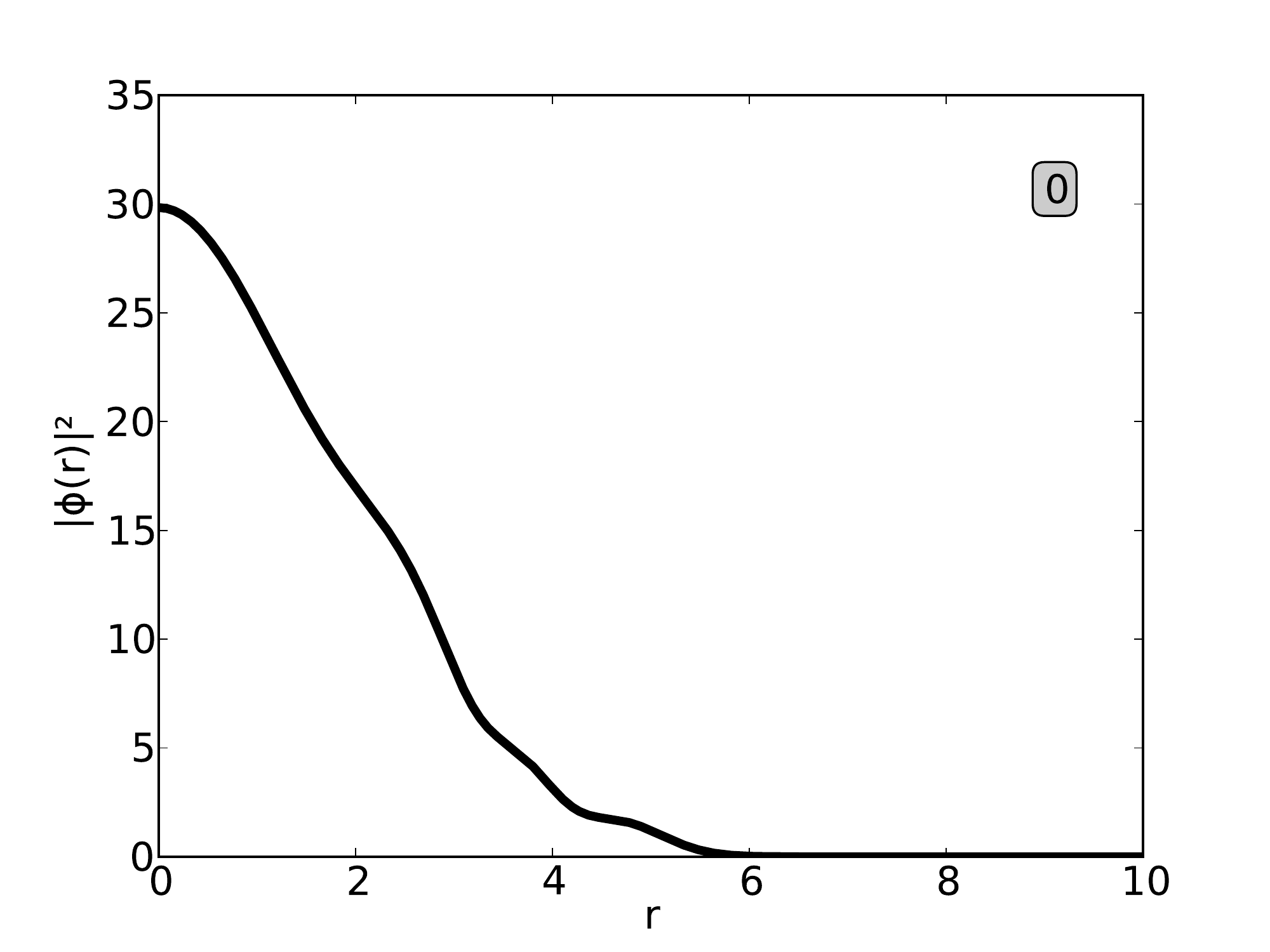} & \includegraphics[scale=0.4]{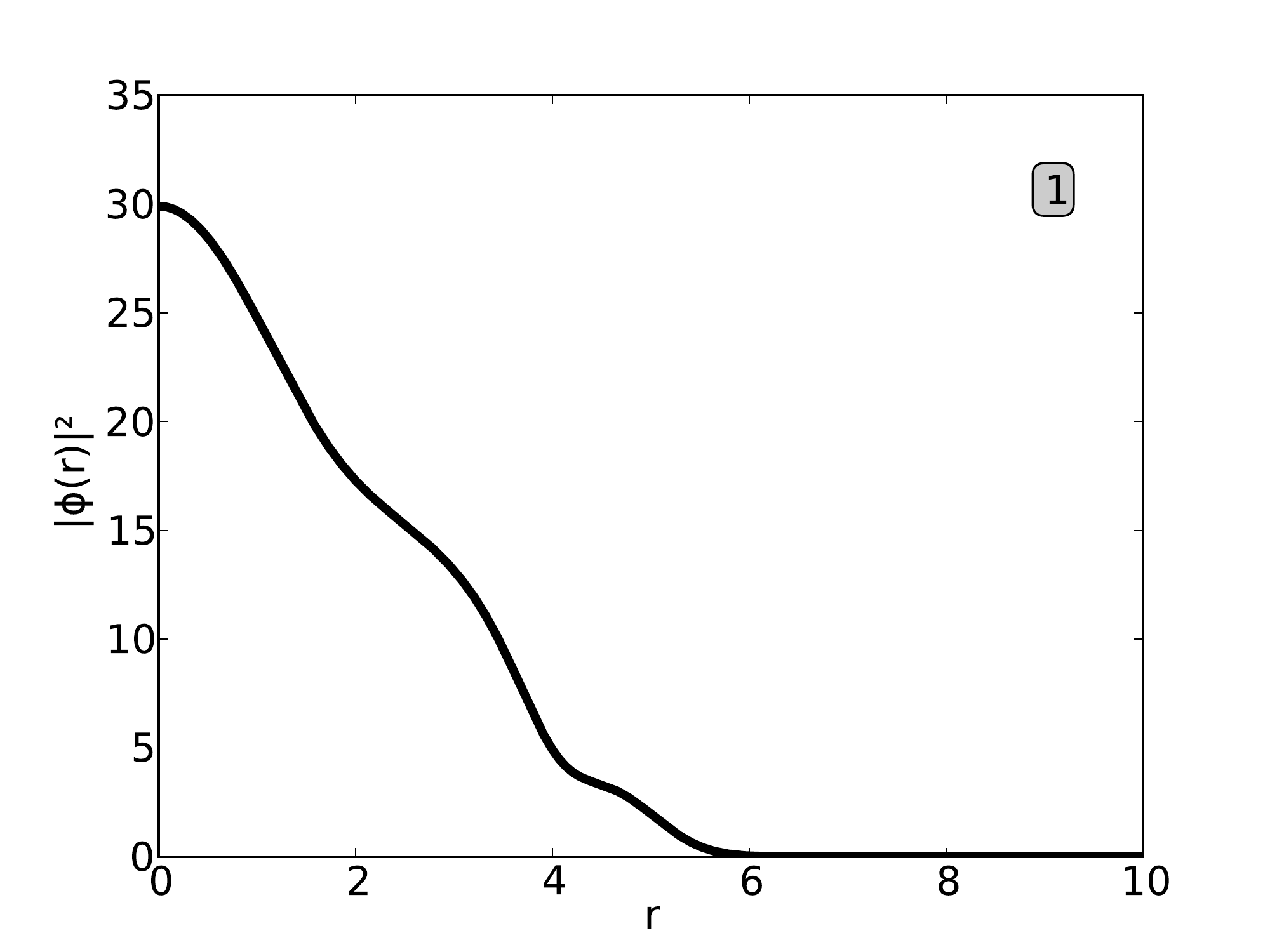}\tabularnewline
\includegraphics[scale=0.4]{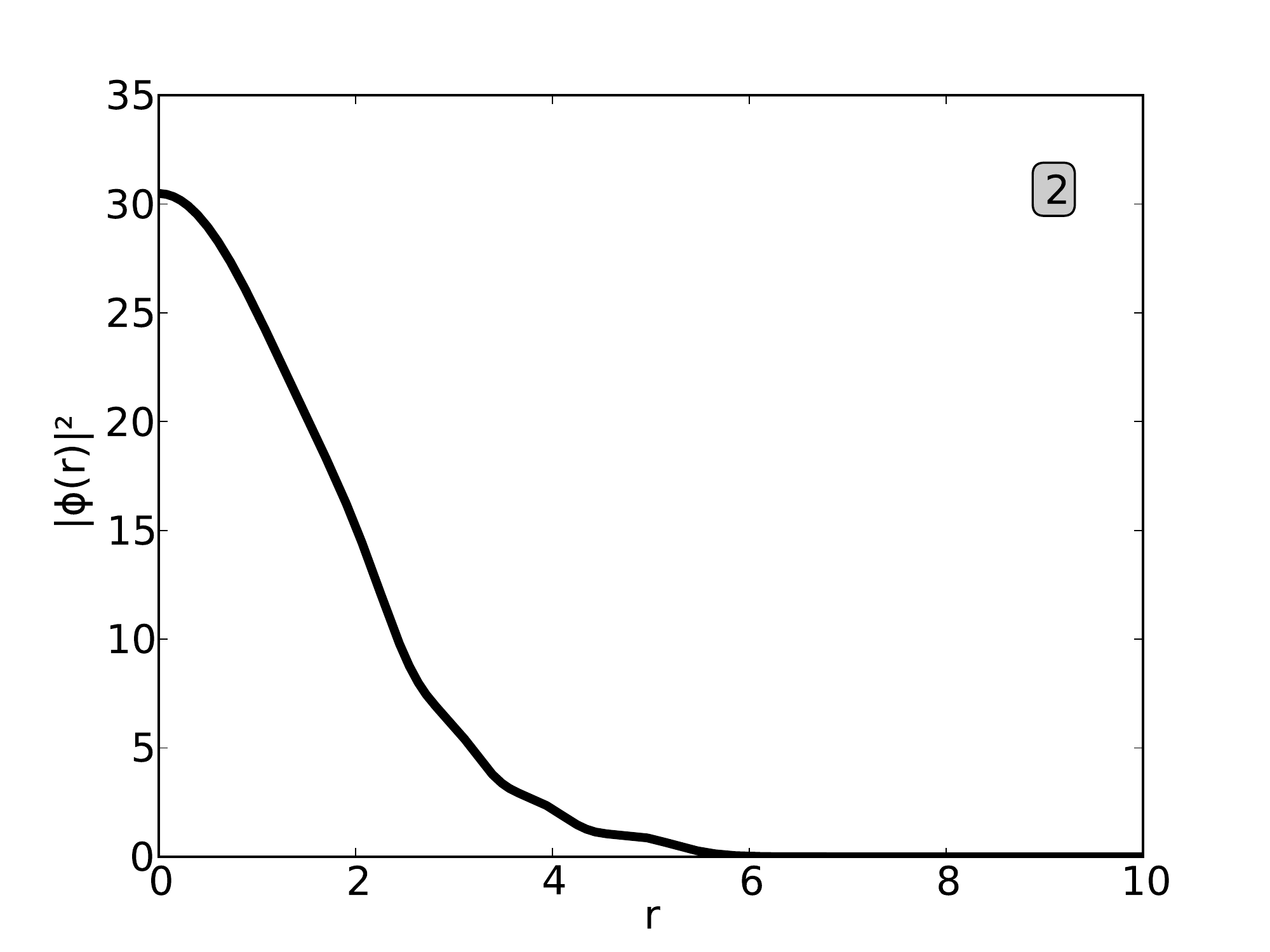} & \includegraphics[scale=0.4]{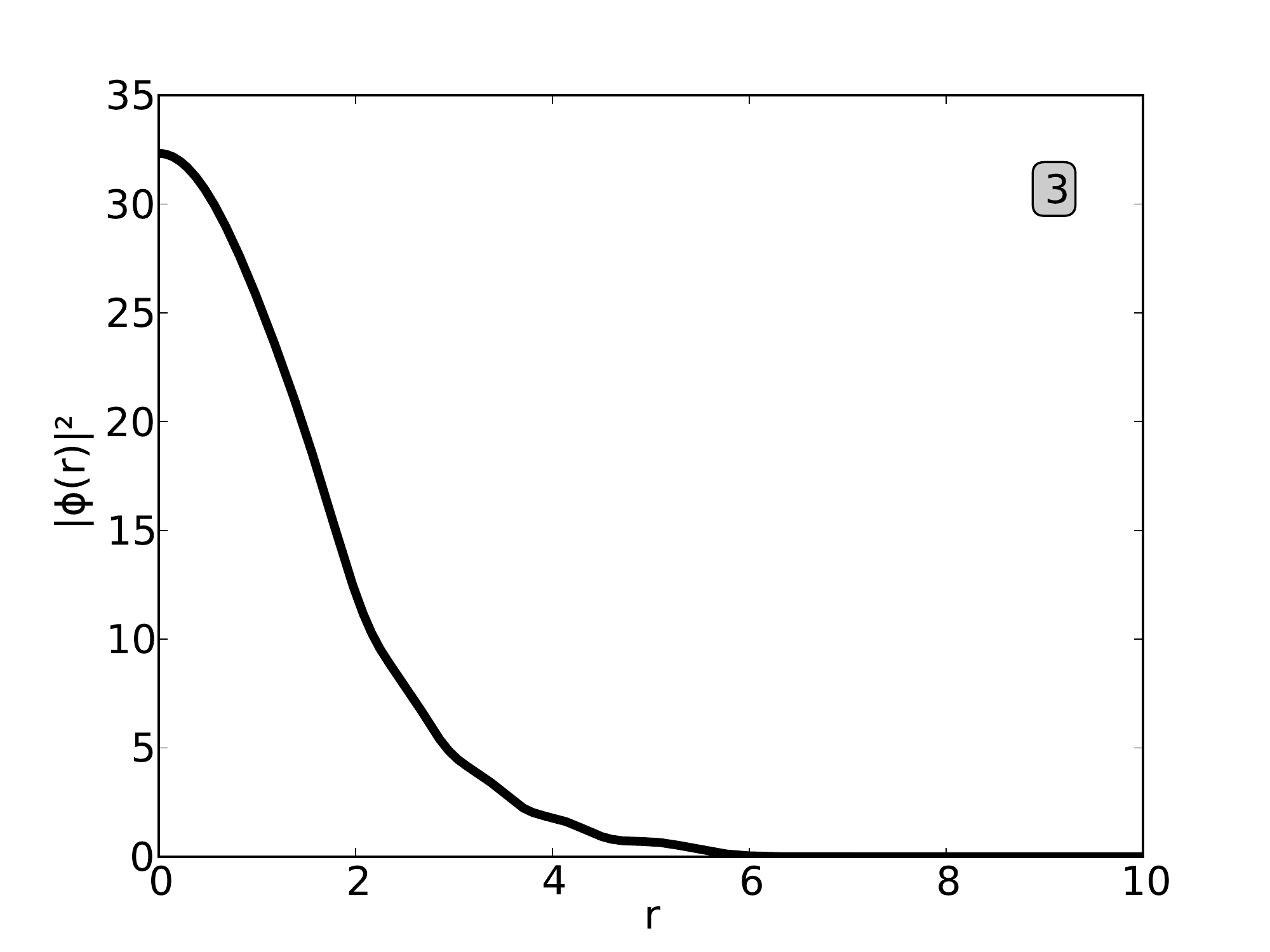}\tabularnewline
\multicolumn{2}{c}{\includegraphics[scale=0.4]{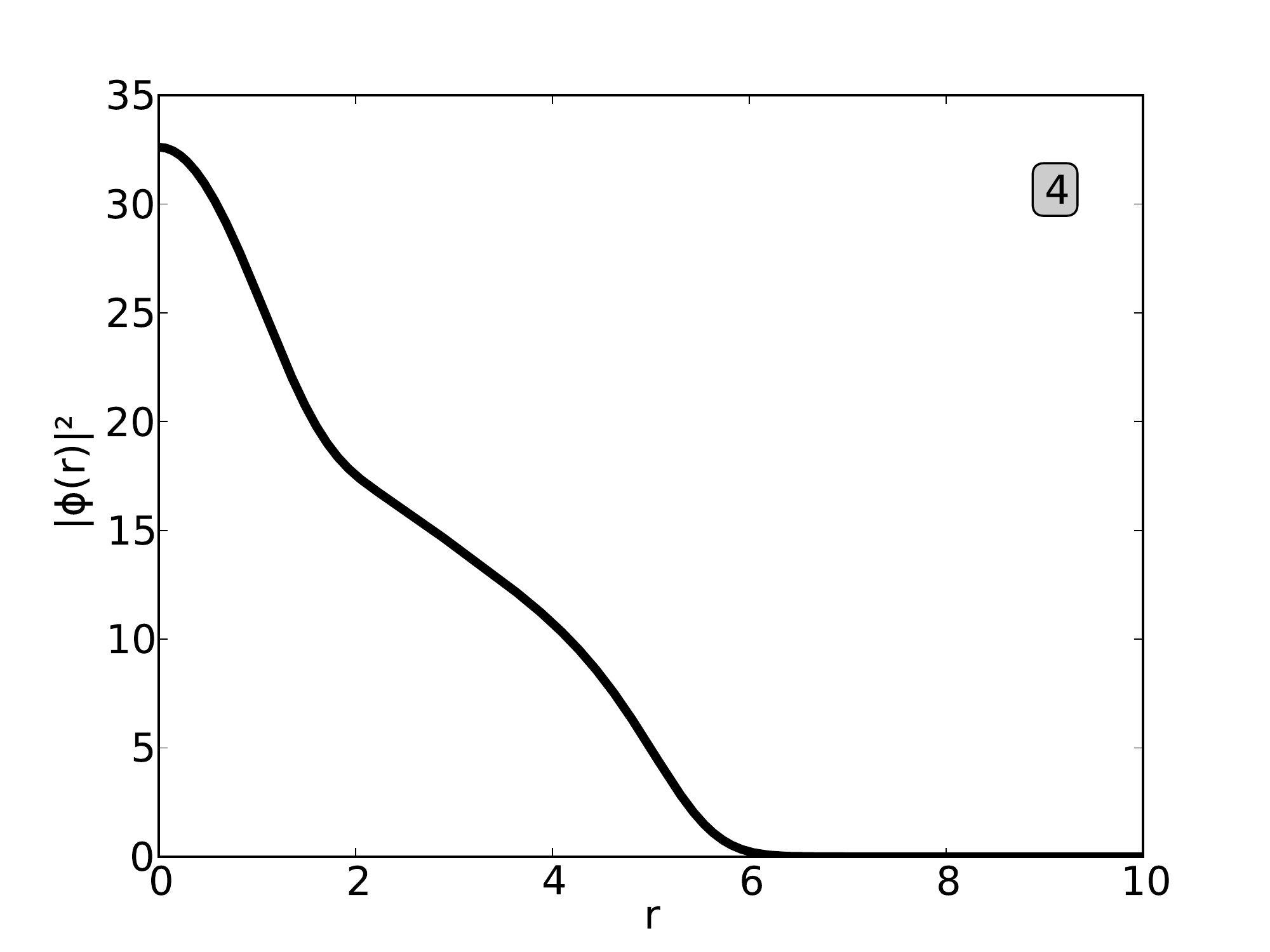}}\tabularnewline
\end{tabular}
\par\end{centering}

\caption{Density profiles corresponding to the solutions labeled 0 to 4 in
Fig. \ref{fig:Cont_R}.\label{fig:Cont_R_profs}}

\end{figure}

\subsection{Central vortex states}

Let us now investigate the existence of central vortex states similar
to the ones observed in a rotating BEC \citep{bao2005ground}. To
find these states, consider solutions of (\ref{eq:cGPE}) of the form
$\psi\left(x,t\right)=e^{-i\mu t}\psi_{m}\left(x\right)=e^{-i\mu t}\psi_{m}\left(r\right)e^{im\theta}$,
where $\mu$ is the chemical potential and $m=1,2,3,...$ is called
the winding index. Introducing this ansatz into (\ref{eq:cGPE}) gives
the following expression for $\psi_{m}$:

\begin{equation}
\mu\psi_{m}\left(r\right)=\left[-\frac{1}{r}\frac{d}{dr}\left(r\frac{d}{dr}\right)+\frac{m^{2}}{r^{2}}+V(r)+\left|\psi_{m}\left(r\right)\right|^{2}+i\left(\alpha\Theta_{R}-\sigma\left|\psi_{m}\left(r\right)\right|^{2}\right)\right]\psi_{m}\left(r\right),\label{eq:ODE_Vortex_St}
\end{equation}
with
\[
\begin{array}{ccc}
\psi{}_{m}\left(0\right)=0, &  & \lim_{r\rightarrow\infty}\psi_{m}\left(r\right)=0\end{array}.
\]

This problem is solved by the collocation method described in Section
3. Figure \ref{fig:Central_vortex_states} shows the results for $\alpha=4.4$,
$\sigma=0.3$, $R=5$ and $m=1$ to $8$. To study their linear stability,
small perturbations of the wave function can be expressed as
\begin{equation}
\begin{array}{cc}
\psi\left(r,\theta,t\right)=e^{-i\mu t}\left[\psi_{m}\left(r\right)e^{im\theta}+\varepsilon\left(u(r)e^{i\left(n\theta-\omega t\right)}+v^{*}\left(r\right)e^{i\left(n\theta+\omega^{*}t\right)}\right)\right], & n=1,2,...,\end{array}\label{eq:TrialWaveFun-1}
\end{equation}
with $\varepsilon\rightarrow0$. The perturbation is time dependent
with frequency $\omega$ and the complex amplitude functions $u(r)$
and $v\left(r\right)$. Proceeding as in the linear stability analysis
performed in Section 4 with $n=1,...,50$, it can be shown that all
of the central vortex states displayed in Fig. (\ref{fig:Central_vortex_states})
are linearly unstable. When these central vortex states are used as
initial conditions for the Strang-splitting spectral method, it is
possible to study the manifestation of these instabilities. In particular,
Figs. (\ref{fig:Sim_m1}) and (\ref{fig:Sim_m2}) show the results
for $m=1,2$. The $m=1$ vortex remains for a long time, and then
it undergoes symmetry breaking and gives rise to a vortex lattice.
The $m=2$ vortex has an additional interesting behavior: it remains
for a short time; the vortex then splits into two $m=1$ vortices;
and finally this solution breaks and generates a vortex lattice. The
splitting for the case $m=2$ was observed experimentally in \citep{d2010persistent}
in an exciton-polariton condensate.

\begin{figure}[H]
\begin{centering}
\includegraphics[scale=0.72]{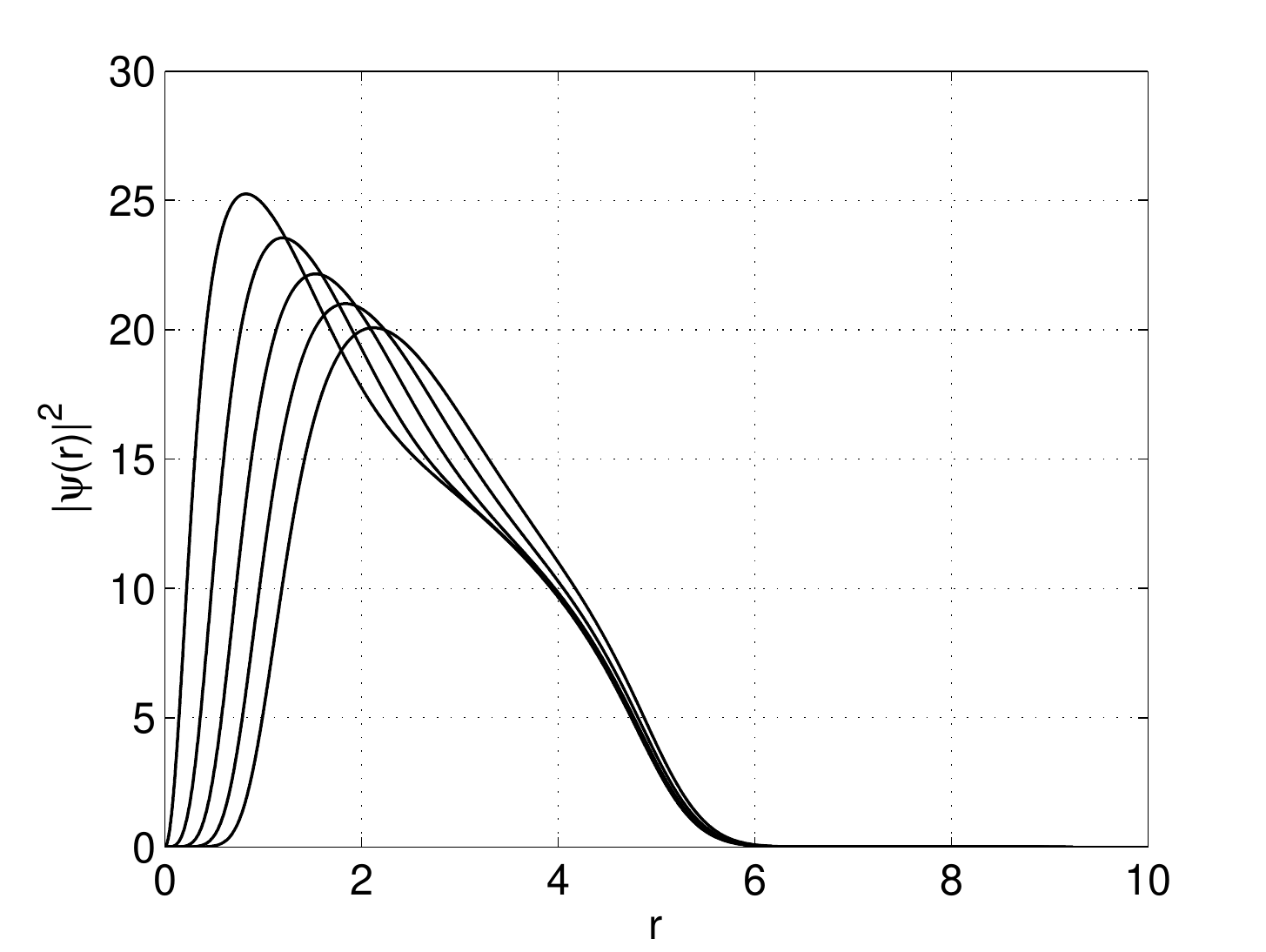}
\par\end{centering}

\caption{Density distributions corresponding to the central vortex states for
$\alpha=4.4$, $\sigma=0.3$, $R=5$ and $m=1$ to 5 (in order of
decreasing peaks).\label{fig:Central_vortex_states}}
 
\end{figure}

\begin{figure}[H]
\begin{centering}
\begin{tabular}{cc}
\includegraphics[width=55mm,height=48mm]{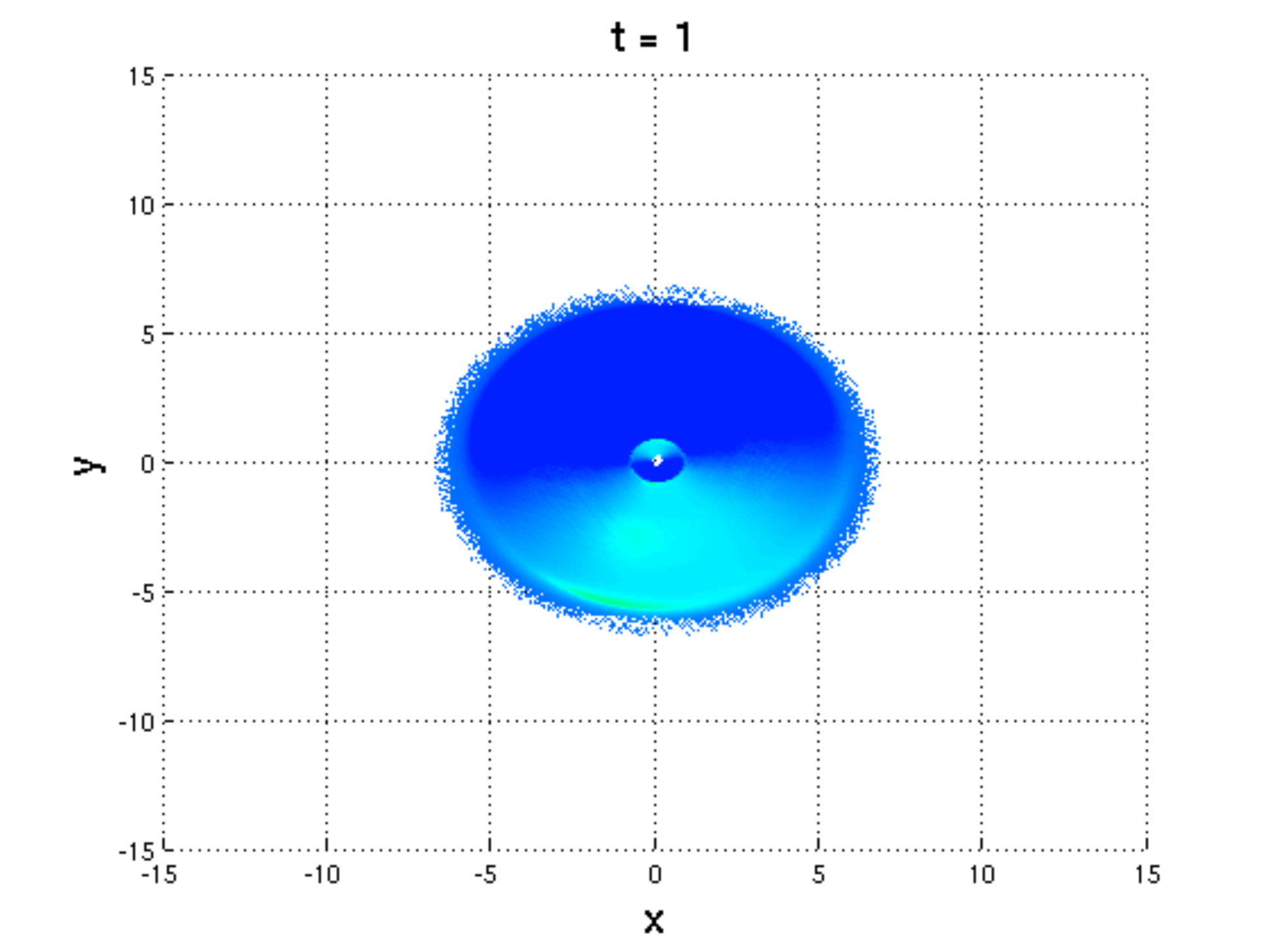} & \includegraphics[width=55mm,height=48mm]{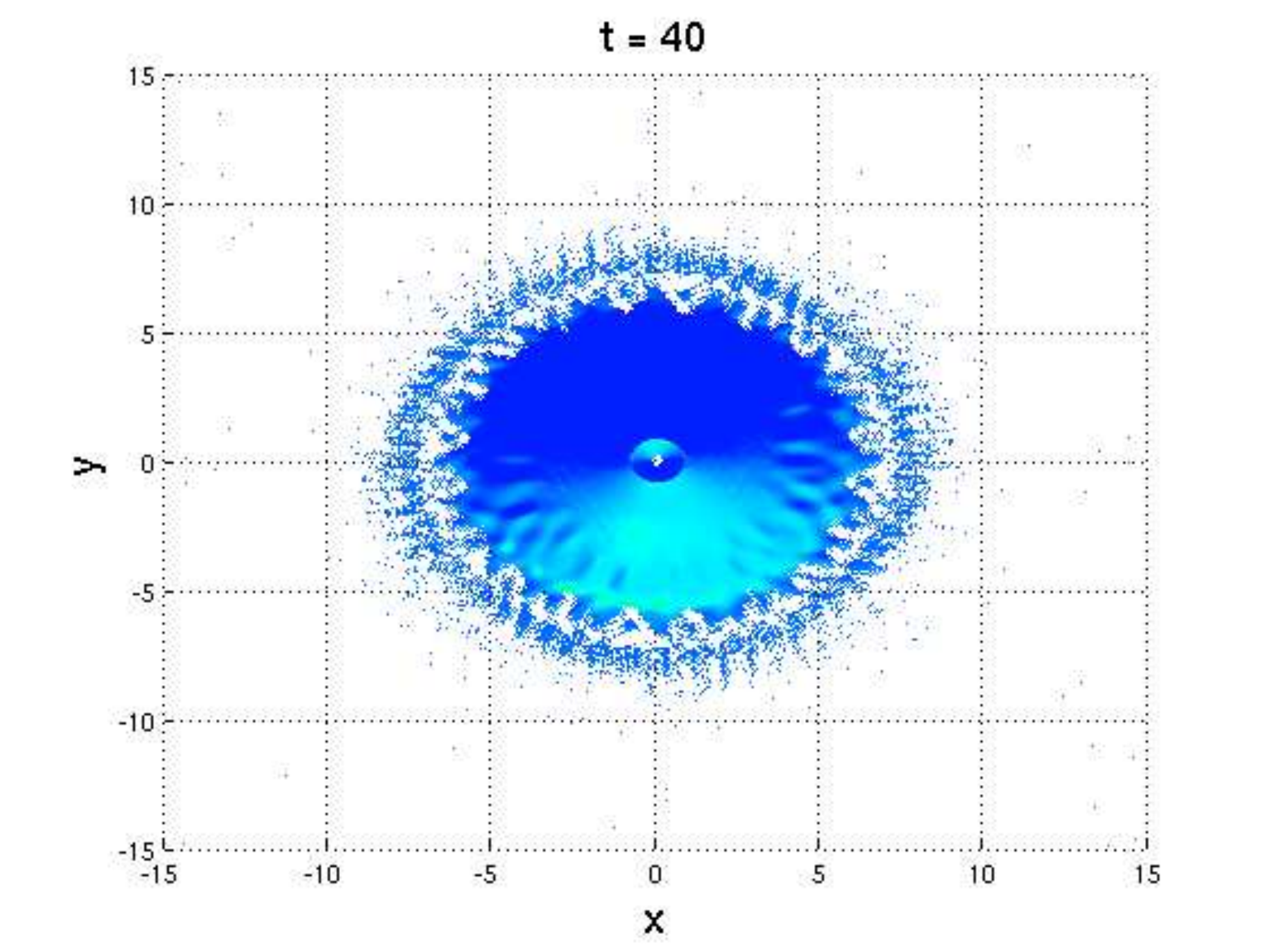}\tabularnewline
\includegraphics[width=55mm,height=48mm]{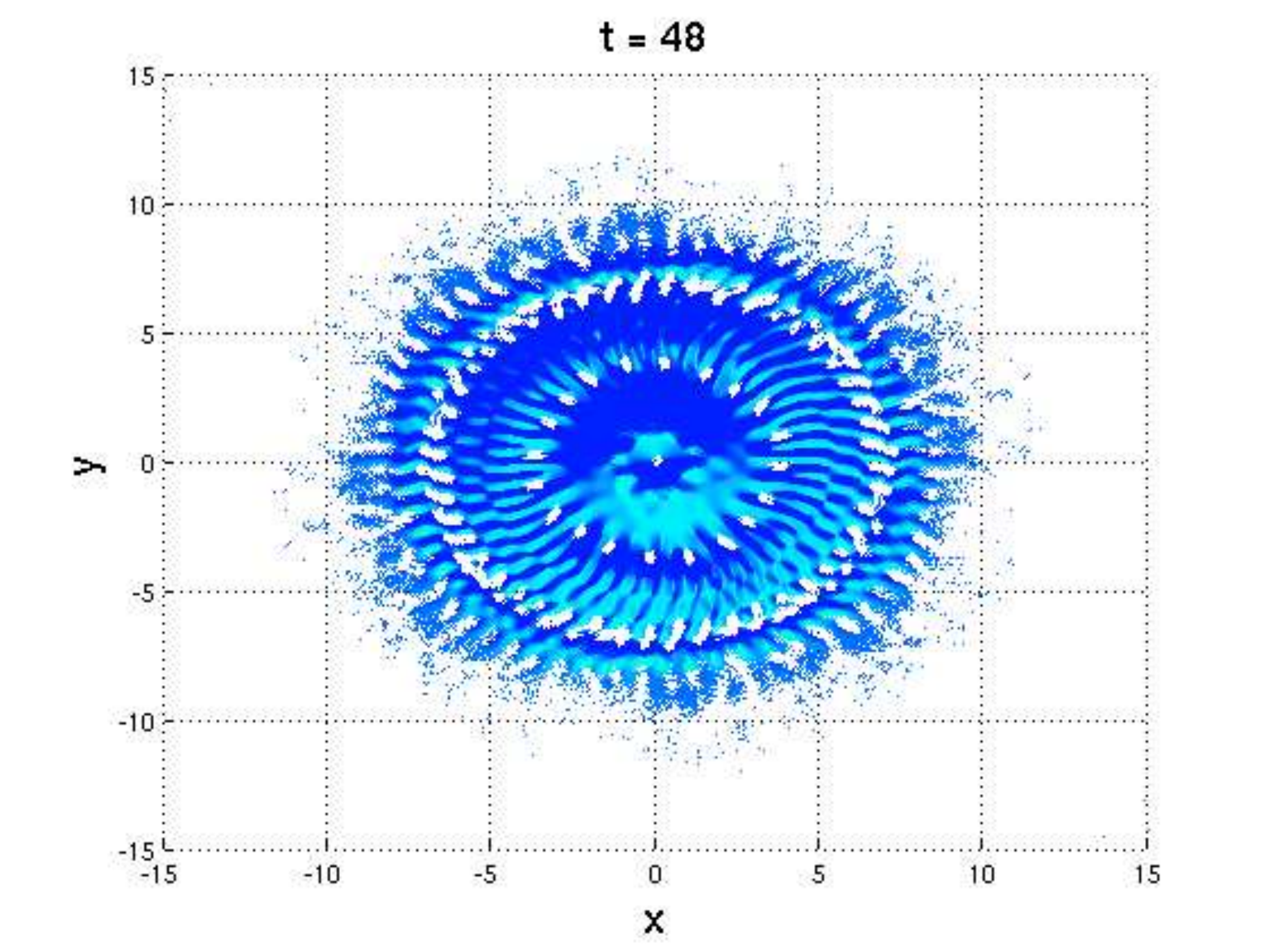} & \includegraphics[width=55mm,height=48mm]{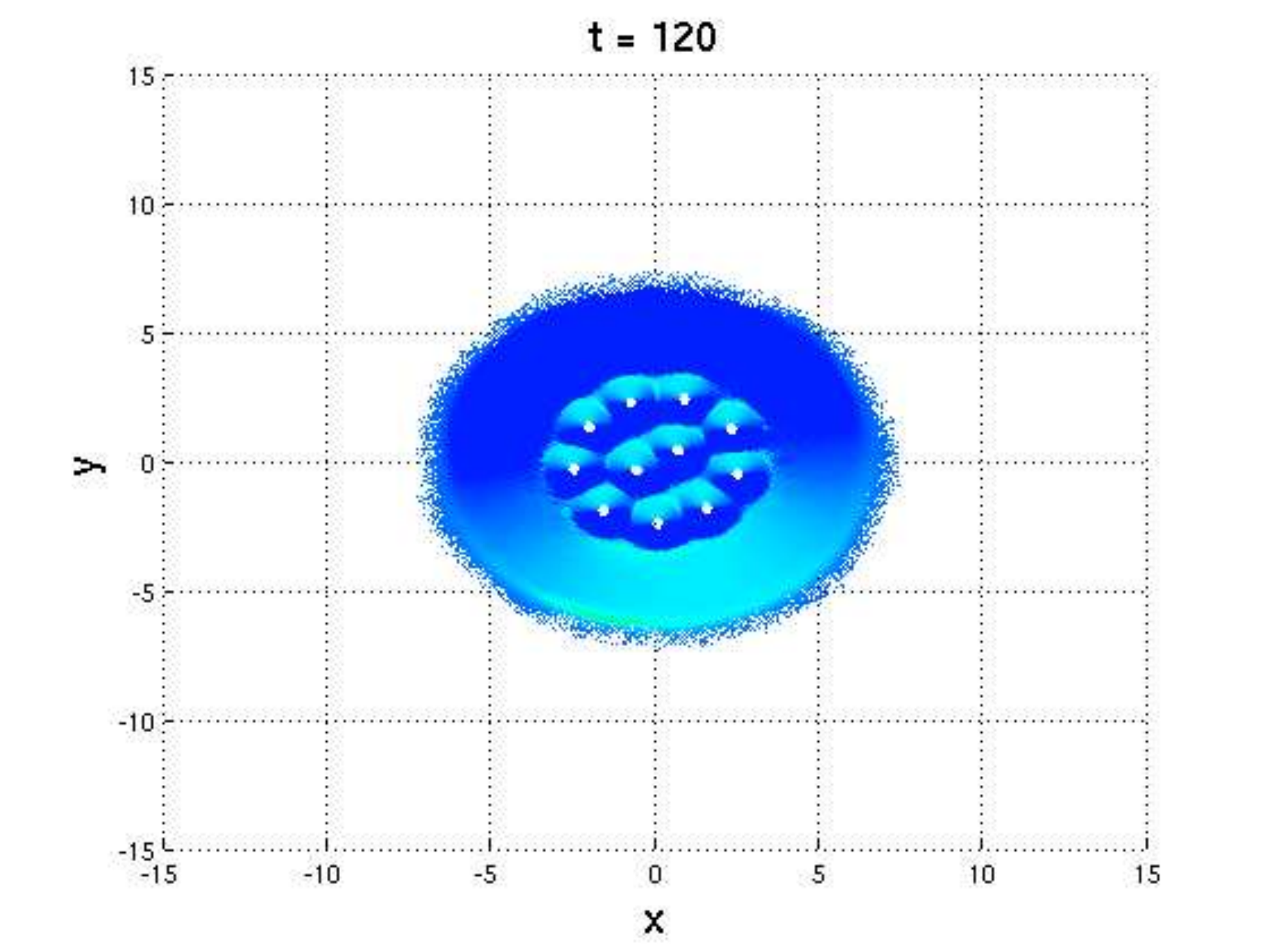}\tabularnewline
\end{tabular}
\par\end{centering}

\caption{Simulation result for the $m=1$ central vortex state (see also supplementary
video Sim3.avi).\label{fig:Sim_m1}}

\end{figure}

\begin{figure}[H]
\begin{centering}
\begin{tabular}{cc}
\includegraphics[width=55mm,height=48mm]{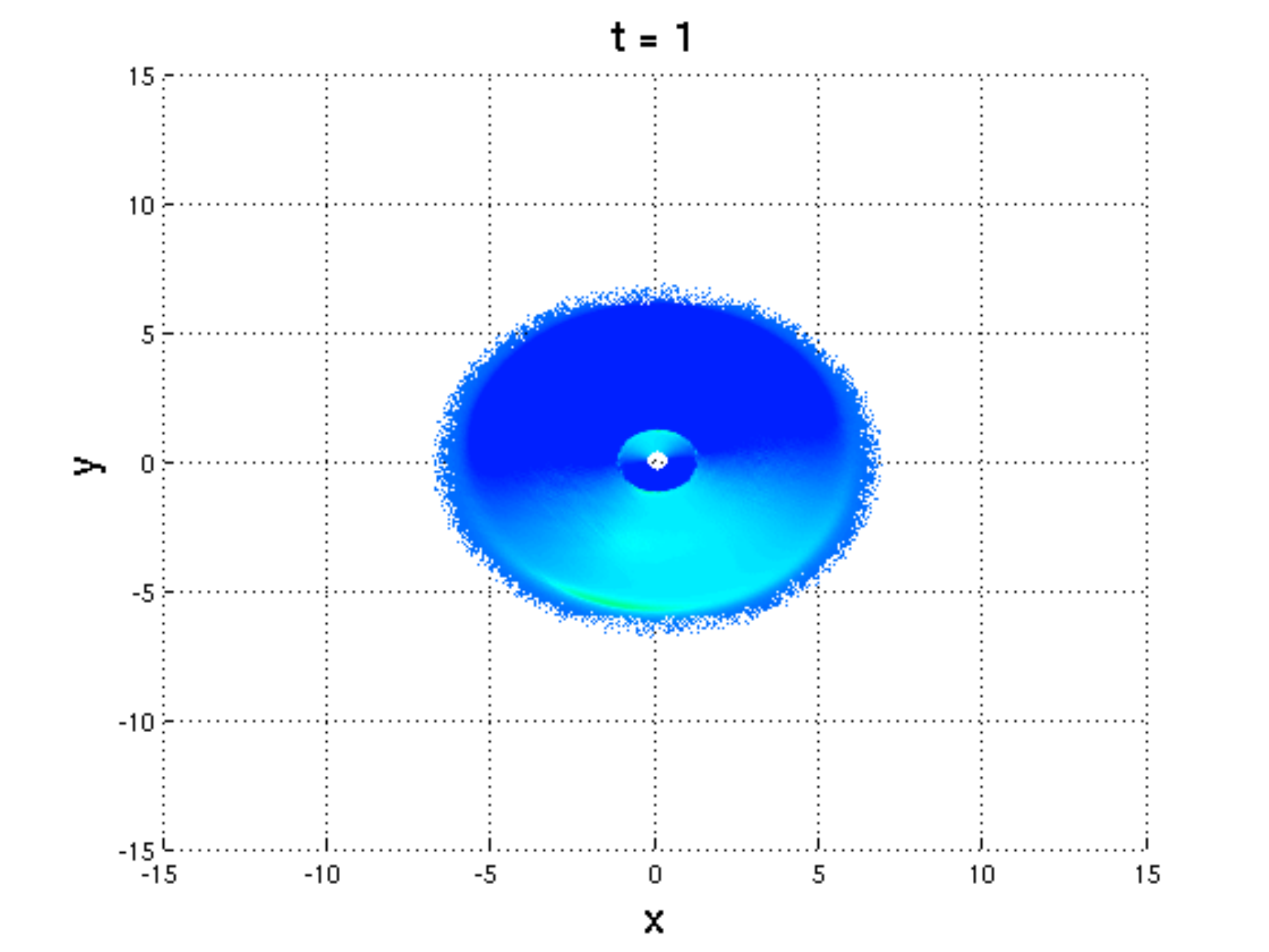} & \includegraphics[width=55mm,height=48mm]{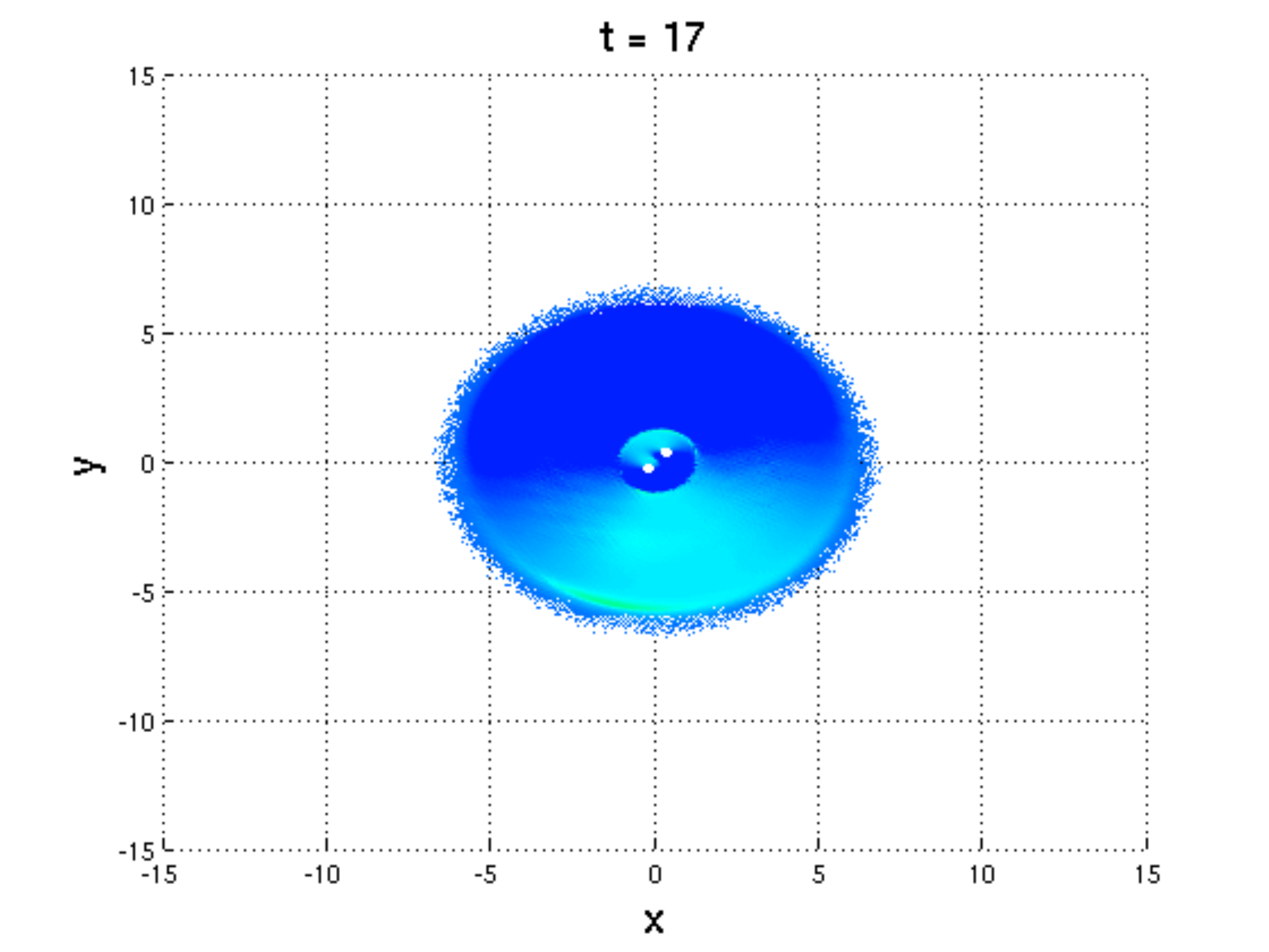}\tabularnewline
\includegraphics[width=55mm,height=48mm]{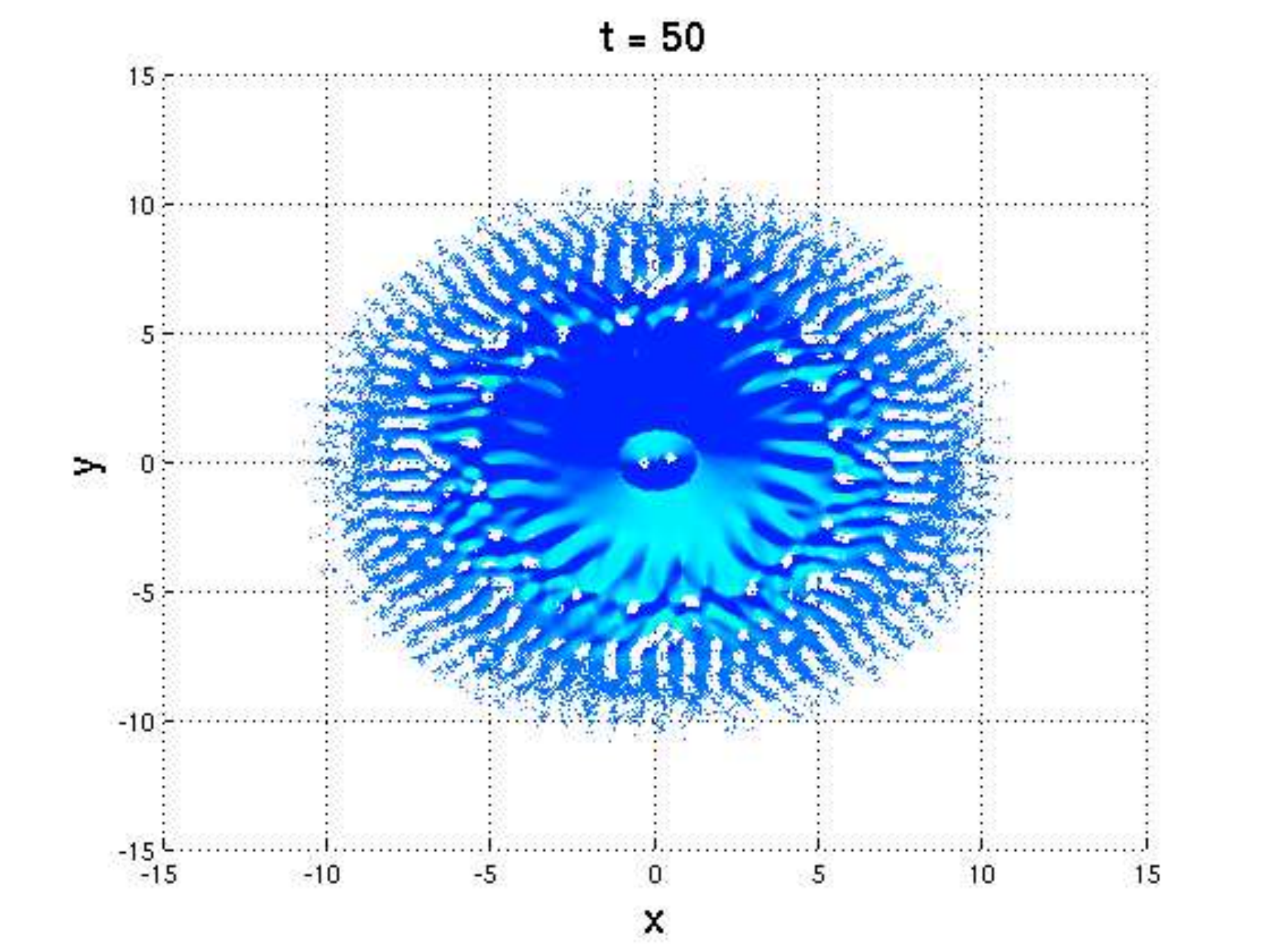} & \includegraphics[width=55mm,height=48mm]{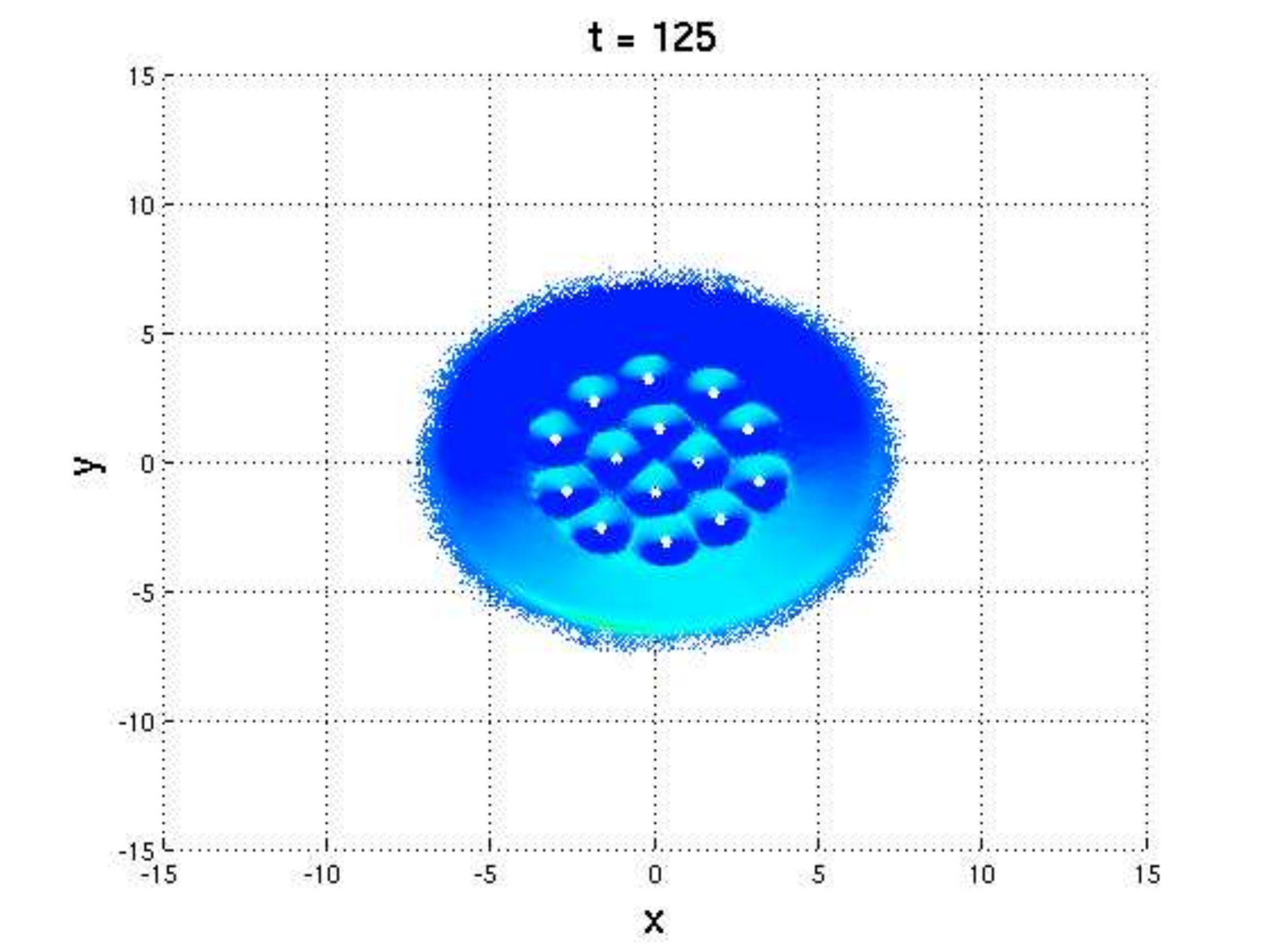}\tabularnewline
\end{tabular}
\par\end{centering}

\caption{Simulation result for the $m=2$ central vortex state (see also supplementary
video Sim4.avi).\label{fig:Sim_m2}}
\end{figure}

\section{Conclusions}

The general behavior of the 2D complex GP equation was presented using
different numerical techniques. The first step for the numerical characterization
of this equation was the analysis of its radially symmetric solutions.
Due to the lack of conservation of mass and energy, the standard techniques
used for the study of the radially symmetric ground state solution
of the GP equation, such as the gradient flow method, do not apply
in this case. In order to overcome this problem, the stationary complex
GP equation (\ref{eq:st_cGPE}) was studied directly using a collocation
method. Once the radially symmetric solutions were obtained, their
linear stability was investigated. Afterwards, numerical continuation
was applied to one of the linearly stable solutions to study its excited
states. We found that the linearly stable solution has the smallest
chemical potential in the group of solutions found, which leads us
to define it as the ground state of the system and the rest of the
solutions as excited states. However, a rigorous definition of the
energy is required to corroborate this statement.

We used numerical integration to study the nonlinear evolution of
the linearly unstable solutions. We observed the emergence of complicated
vortex lattices after symmetry breaking (this fact was also shown
by \citet{keeling2008spontaneous}). These lattices remain rotating
with a constant angular velocity, becoming the stable solution of
the system. 

Finally, we computed the central vortex solutions of the complex GP
equation. An interesting result was the splitting of the $m=2$ central
vortex into two $m=1$ vortices, which was observed experimentally
by \citet{d2010persistent} in a Bose-Einstein condensation of exciton-polaritons.
From all this, it is possible to see the rich behavior of the complex
GP equation, which in some way combines the characteristics of both
the GP and the complex Ginzburg-Landau equations.

\section{Acknowledgments}

The first author acknowledges the assistance and comments from D.
Ketcheson, P. Antonelli, N. Berloff, B. Sandstede, B. Oldeman and
the Research Computing Group from KAUST. The work of the last author
has been supported by the Hertha-Firnberg Program of the FWF, grant
T402-N13.

\section{References}

\bibliographystyle{plainnat}
\bibliography{bib}

\end{document}